\documentclass[a4paper,10pt]{article}
\usepackage[english]{babel}

\usepackage[utf8]{inputenc}
\usepackage{amsmath,mathrsfs}
\usepackage{indentfirst}
\usepackage{fancyhdr}
\usepackage{amssymb}
\usepackage{amsthm}
\usepackage[dvips]{graphicx}
\usepackage{color}
\usepackage{cancel}
\usepackage[dvipsnames]{xcolor}
\usepackage{eucal}
\usepackage{latexsym}
\usepackage{chngcntr}
\usepackage{faktor}
\usepackage{microtype}
\usepackage{newlfont}
\usepackage{enumitem}
\usepackage{hyperref,geometry,mathtools, soul} 
\usepackage{todonotes}

\usepackage{amsmath, amsfonts, amsthm,amssymb,  bbm}
\usepackage{amsmath}

\usepackage{tikz-cd}
\usepackage{multicol}
\usepackage{cancel}
\usepackage{mathtools}
\usepackage{graphicx}
\usepackage{MnSymbol}

\DeclareMathOperator*{\sgn}{sgn}

\newcommand{\e}{\epsilon}
\newcommand{\im}{\mathrm{i}\,}

\newcommand{\molt}[2]{\left(#1\,,\,#2\right)}
\newcommand{\lie}[2]{\left[#1\,,\, #2\right]}

\newcommand\restr[2]{{% we make the whole thing an ordinary symbol
  \left.\kern-\nulldelimiterspace % automatically resize the bar with \right
  #1 % the function
  \vphantom{\big|} % pretend it's a little taller at normal size
  \right|_{#2} % this is the delimiter
  }}  
  \newcommand{\BVe}[4]{\left({\cal B}_{\mu,\e}\, \tf^{#1}_{{#2}} \ , \tf^{#3}_{#4} \right) }

\allowdisplaybreaks
\theoremstyle{plain}
\newtheorem{lem}{Lemma}
\newtheorem{teo}[lem]{Theorem}
\newtheorem{prop}[lem]{Proposition}

\theoremstyle{definition}

\newtheorem{sia}[lem]{Definition}

\newtheorem{rmk}[lem]{Remark}

\renewcommand{\bar}{\overline}

\newcommand{\vet}[2]{\begin{bmatrix}#1 \\ #2 \end{bmatrix}}
\newcommand{\uno}{\mathrm{Id}}

\newcommand{\bR}{\mathbb{R}}
\newcommand{\bT}{\mathbb{T}}
\newcommand{\bZ}{\mathbb{Z}}
\newcommand{\bN}{\mathbb{N}}

\newcommand{\bC}{\mathbb{C}}

\newcommand{\tf}{\mathtt{f}}

\newcommand{\cL}{\mathcal{L}}
\newcommand{\cO}{\mathcal{O}}

\newcommand{\cJ}{\mathcal{J}}
\newcommand{\cB}{\mathcal{B}}
\newcommand{\cV}{\mathcal{V}}
\newcommand{\tB}{\mathtt{B}}
\newcommand{\tJ}{\mathtt{J}}
\newcommand{\tL}{\mathtt{L}}
\newcommand{\de}{\mathrm{d}}
\newcommand{\pa}{\partial}

\newcommand{\cH}{\mathcal{H}}

\newcommand{\cU}{\mathcal{U}}
\newcommand{\cW}{\mathcal{W}}

\newcommand{\off}{\varnothing}

\newcommand{\bro}{\bar\rho}
\newcommand{\Gc}{\mathfrak{c}}

\newcommand{\sL}{\mathscr{L}}

\newcommand{\ta}{{\mathtt{a}}}
\newcommand{\tb}{{\mathtt{b}}}
\newcommand{\tc}{{\mathtt{c}}}

\numberwithin{equation}{section}
\counterwithin{lem}{section}
\title{\bf Full description of  Benjamin-Feir instability \\
of Stokes waves in deep water
}
\begin{document}

 \author{Massimiliano Berti, Alberto Maspero, Paolo Ventura\footnote{
International School for Advanced Studies (SISSA), Via Bonomea 265, 34136, Trieste, Italy. 
 \textit{Emails: } \texttt{berti@sissa.it},  \texttt{alberto.maspero@sissa.it}, \texttt{paolo.ventura@sissa.it}
 }}

\date{}

\maketitle
\begin{abstract}
Small-amplitude,  traveling, space periodic  solutions --called Stokes waves-- 
of  the 2 dimensional gravity water waves equations in deep water are 
linearly unstable with respect to long-wave perturbations,
as predicted by Benjamin and Feir in 1967. 
We completely  describe the behavior of 
 the four eigenvalues close to zero  of the linearized equations 
 at the Stokes wave, 
 as  the Floquet exponent is turned on.
 We prove in particular the conjecture that  a pair of non-purely imaginary eigenvalues
depicts a closed figure ``8'', parameterized by the Floquet exponent, in full agreement with numerical simulations. 
Our new spectral approach to the Benjamin-Feir instability phenomenon 
uses a symplectic version of Kato's theory of similarity transformation to  
reduce the  problem to determine the eigenvalues of a $ 4 \times 4 $ complex 
Hamiltonian and reversible 
matrix.  Applying a procedure inspired by KAM theory,
we block-diagonalize such matrix  into 
a pair of $2 \times 2  $ Hamiltonian and reversible 
matrices, thus obtaining 
the full description 
of its eigenvalues. 
\end{abstract}

\tableofcontents

\section{Introduction}

Since the pioneering work of Stokes \cite{stokes} in 1847, 
a huge literature has established the existence of  steady space periodic traveling  waves,
namely solutions 
 which look stationary in a moving frame.  Such solutions are called Stokes waves. 
 A problem of fundamental importance in fluid mechanics regards their stability/instability  subject to long space periodic  perturbations.
 In 1967 Benjamin and Feir \cite{Benjamin,BF} discovered, with heuristic arguments, 
 that a long-wave 
perturbation of a small amplitude space periodic Stokes wave is unstable; see also the 
the independent results by Lighthill   \cite{Li} and Zakharov \cite{Z0,ZK} and the survey \cite{ZO} for an historical overview. 
This phenomenon is nowadays called 
``Benjamin-Feir" --or modulational-- instability, and it is supported by an enormous  amount of  physical observations and numerical simulations, see e.g. \cite{DO,Ak,KDZ,CDT} and references therein. 

It took almost thirty years to get the first  rigorous  proof of the Benjamin-Feir instability for the  water waves equations in two dimensions, obtained   by 
Bridges-Mielke \cite{BrM}  in finite depth,  % (see also the preprint by  Hur-Yang \cite{HY}), 
and  fifty-five years for  the  infinite depth case, 
proved last year by   Nguyen-Strauss \cite{NS}.  

\smallskip 
The problem is mathematically formulated 
 as follows.  Consider the pure gravity water waves equations  for a bidimensional fluid in  deep water
 and a $2\pi$-periodic Stokes wave solution with amplitude $0< \e \ll 1$. 
The linearized  water waves equations at the Stokes waves
are, in the inertial reference frame moving with 
the  speed $ c $ of the Stokes wave, a linear time independent system of the form
$ h_t = \mathcal{L}_{\e}  h  $ where $ \mathcal{L}_{\e}  $
is a linear operator  with $ 2 \pi $-periodic coefficients,  
see \eqref{cLepsilon}\footnote{The operator $ \mathcal{L}_{\e}  $ 
in \eqref{cLepsilon} is actually  obtained 
conjugating the linearized water waves equations in the Zakharov  formulation 
% at the Stokes wave 
via the  ``good unknown of Alinhac" \eqref{Alin} and the 
Levi-Civita \eqref{LC} invertible transformations.}.
The operator 
$ \mathcal{L}_{\e} $ possesses the eigenvalue $ 0 $ 
with  algebraic multiplicity four 
 due to  symmetries of the water waves equations 
(that we describe in the next section).  
The problem is to prove that  
 $ h_t = \mathcal{L}_{\e}  h  $  
has  solutions  of the form $h(t,x) = \text{Re}\left(e^{\lambda t} e^{\im \mu x} v(x)\right)$
where $v(x)$ is a  $2\pi$-periodic function, $\mu$ in $ \bR$ (called Floquet exponent) 
and $\lambda$ has positive real part, thus $h(t,x)$ grows exponentially in time.
By Bloch-Floquet theory, such $\lambda$ is an  eigenvalue 
  of the operator  $ \mathcal{L}_{\mu,\e} 
:= e^{-\im \mu x } \,\mathcal{L}_{\e} \, e^{\im \mu x } $
acting on $2\pi$-periodic functions. 

\smallskip
 The main result of this paper provides the full description of the four 
 eigenvalues  close to 
zero of the operator $ \mathcal{L}_{\mu,\e} $ when  $ \e $ and $ \mu $ are small enough, see Theorem \ref{TeoremoneFinale},
thus concluding the analysis started in 1967 by Benjamin-Feir.
We first state the following result
which focuses on the 
Benjamin-Feir {\it unstable} eigenvalues.

Along  the paper we denote by $r(\e^{m_1} \mu^{n_1}, \ldots, \e^{m_p} \mu^{n_p})$
a real analytic function fulfilling  for some $C >0$ and $\e, \mu$ sufficiently small, the estimate  $| r(\e^{m_1} \mu^{n_1}, \ldots, \e^{m_p} \mu^{n_p}) | \leq 
C \sum_{j=1}^p
 |\e|^{m_j} |\mu|^{n_j}
$.

 \begin{teo}\label{thm:simpler}
There exist $ \e_1, \mu_0 > 0 $  and 
an analytic  function $\underline \mu: [0,\e_1)\to [0,\mu_0)$, 
 of the form $ \underline  \mu(\e)  =   2\sqrt{2} \e(1+r(\e)) $, such that, 
 for any  $  \e \in [0, \e_1)  $, the 
 operator  $\cL_{\mu,\e}$ has two eigenvalues
  $\lambda^\pm_1 (\mu,\e)$ of the form 
\begin{equation}\label{eigelemu}
 \begin{cases}
  \frac12\im\mu+\im r(\mu\e^2,\mu^2\e,\mu^3)\pm
  \frac{\mu}{8}\sqrt{8\e^2\big(1+r_0(\e,\mu)\big)-\mu^2\big(1+r_0'(\e,\mu)\big)} \, ,  & \forall   \mu \in [0, \underline \mu (\e)) \, , \\
 \frac12\im \underline \mu (\e)+\im r(\e^3) \, ,  & \mu= \underline \mu  (\e) \, , \\
 \frac12\im\mu+\im r(\mu\e^2,\mu^2\e,\mu^3)\pm\im\frac{\mu}{8}\sqrt{\mu^2\big(1+r_0'(\e,\mu)\big)-8\e^2\big(1+r_0(\e,\mu)\big)} \, ,  & \forall \mu \in ( \underline \mu  (\e), \mu_0) \, .
\end{cases}
\end{equation}
The function $ 8\e^2\big(1+r_0(\e,\mu)\big)-\mu^2\big(1+r_0'(\e,\mu) )  $  is $>0$, respectively $<0$,  provided $0<\mu < \underline{\mu}(\e)$,  respectively $\mu > \underline{\mu}(\e)$.
\end{teo}
Let us make some comments on  the result.
\\[1mm]
\indent 1.
According to \eqref{eigelemu}, 
 for values of the Floquet parameter $ 0<\mu <  \underline \mu (\e) $ the eigenvalues 
$\lambda^\pm_1 (\mu, \epsilon) $ have opposite non-zero real part.
  As $ \mu $ tends to $  \underline  \mu (\e)$, the two eigenvalues $\lambda^\pm_1 (\mu,\epsilon) $ 
 collide on the imaginary axis {\it far} from $ 0 $ (in the upper semiplane $ \text{Im} (\lambda) > 0 $), 
along which they  keep moving  for $ \mu > \underline  \mu (\e) $,
see Figure \ref{figure-eigth}. For $ \mu < 0 $ the operator $ {\mathcal L}_{\mu,\e} $
possesses the symmetric eigenvalues 
$ \overline{\lambda_1^{\pm} (-\mu,\e)} $ in the semiplane $ \text{Im} (\lambda) < 0 $. 
\\[1mm]
\indent 2. Theorem \ref{thm:simpler} proves the long-standing conjecture that the
unstable eigenvalues $\lambda^\pm_1 (\mu,\epsilon )$ depict a complete  figure ``8'' as $\mu$ varies in the interval $[ - \underline \mu(\e) , \underline \mu(\e)]$, 
see Figure \ref{figure-eigth}. For $ \mu \in [0, \underline \mu(\e)]$  
we obtain the upper part of  
the figure  ``8'', which is 
well approximated by the curves 
$\mu \mapsto (\pm\frac{\mu}{8}\sqrt{8\e^2 -\mu^2}, \frac12 \mu)$, 
in complete accordance with the  numerical simulations by Deconinck-Oliveras  \cite{DO}.
For $ \mu \in [\underline \mu(\e), \mu_0 ]$ the purely imaginary
eigenvalues are approximated 
by 
$ \im \frac{\mu}{2} ( 1 \pm \frac{1}{4}\sqrt{\mu^2 - 8\e^2})$.   
The higher order corrections of the eigenvalues $ \lambda_1^\pm (\mu,\e ) $ in \eqref{eigelemu}, 
provided by the analytic functions $ r_0(\e,\mu), r_0'(\e,\mu) $,  
are explicitly computable. 
Theorem \ref{thm:simpler} is the 
first rigorous proof of  the ``Benjamin-Feir figure 8'', not only for the water waves equations, but also  in any 
model exhibiting modulational instability, 
that we quote at the end of this introduction
(for the focusing % integrable 
$ 1d$ NLS
equation Deconinck-Upsal \cite{DU} 
showed the presence of a figure ``8''  for elliptic  solutions,
exploiting the integrable structure of the equation).   
\begin{figure}[h]
\centering
\includegraphics[width=5.5cm]{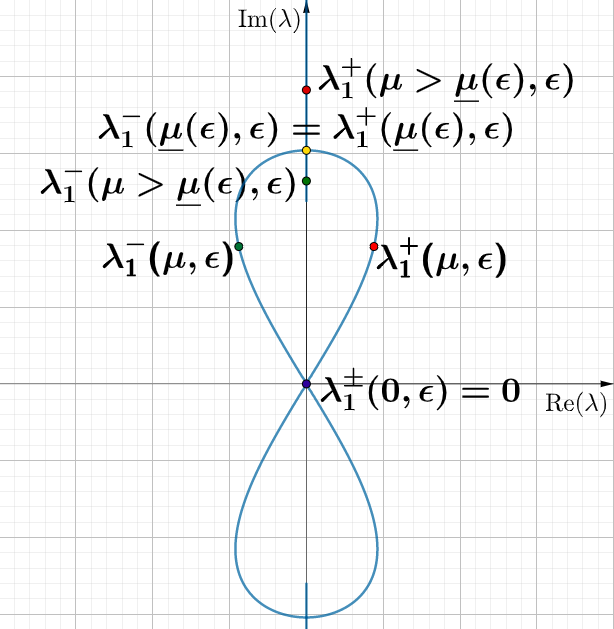}
\caption{Traces of the  eigenvalues $\lambda^\pm_1 (\mu,\epsilon )$ in the complex $\lambda$-plane at fixed $|\e| \ll 1 $ as $\mu$ varies. For $ \mu \in (0, \underline{\mu}(\e))  $ the eigenvalues fill 
the portion of the $ 8 $ in $ \{ \text{Im} (\lambda) > 0 \} $ and for $ \mu \in ( - \underline{\mu}(\e),0 ) $ the symmetric portion in 
$ \{ \text{Im} (\lambda) < 0 \} $.}
\label{figure-eigth}
\end{figure} 

3.  Nguyen-Strauss result in \cite{NS} describes  the portion of  unstable eigenvalues very close to the origin, namely the cross amid the ``8''. 
Formula  \eqref{eigelemu} prolongs these local  branches of eigenvalues  far from the bifurcation, until they collide again on the imaginary axis. 
Note that as $  0<\mu  \ll \e  $ the eigenvalues $ \lambda^\pm_1 (\mu, \epsilon) $ in \eqref{eigelemu} 
have the same asymptotic expansion given  in Theorem 1.1 of    \cite{NS}.
\\[1mm]
\indent 4. 
The eigenvalues \eqref{eigelemu} are {\it not analytic}  in $(\mu, \epsilon)$
close to the
value $(\underline{\mu}(\epsilon),\epsilon)$ where  $  \lambda^\pm_1 (\mu, \epsilon) $
 collide  at the top of the 
figure ``8'' far from $0$ (clearly they are continuous).
 In previous approaches the eigenvalues are a priori supposed to 
be  analytic in $(\mu, \epsilon)$, and that restricts their validity
 to suitable regimes. 
We remark that  \eqref{eigelemu} are the eigenvalues of the 
$ 2 \times 2 $ matrix ${ \mathtt U}$ given in Theorem \ref{TeoremoneFinale},   
which {\it is analytic} in $(\mu, \epsilon)$. 
\\[1mm]
\indent 5.  In Theorem \ref{TeoremoneFinale} we actually prove the expansion of  the unstable eigenvalues of $ \cL_{\mu,\e} $  for any value 
of the parameters $(\mu,\e)$ in a rectangle $ [0,\mu_0) \times [0,\e_0 )$. The analytic curve
$ \underline  \mu(\e)  =   2\sqrt{2} \e(1+r(\e)) $, tangent at $ \e = 0 $ to the straight line 
$ \mu = 2\sqrt{2} \e $ divides such rectangle in the 
``unstable" region where there exist  eigenvalues of 
$ {\cL}_{\mu,\e}$ with non-trivial real part, from the ``stable" one 
where all the eigenvalues of $ {\cL}_{\mu,\e}$ are purely imaginary, see Figure \ref{fig2}.
\begin{figure}[h]
\centering
\includegraphics[width=10.5cm]{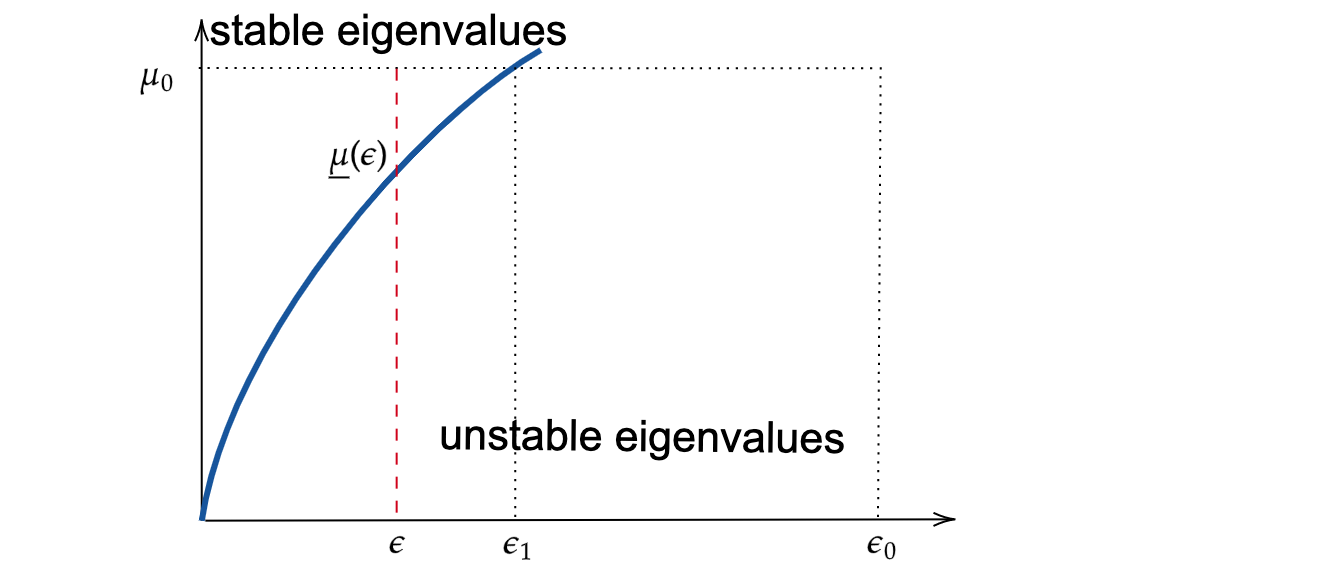}
\caption{The blue line is the analytic curve defined implicitly by  $8\e^2\big(1+r_0(\e,\mu)\big)-\mu^2\big(1+r_0'(\e,\mu) )  = 0$. For values of $\mu$ below this curve, the two eigenvalues $\lambda^\pm_1(\mu,\e)$ have opposite real part. For $\mu$ above the curve, $\lambda^\pm_1(\mu,\e)$ are purely imaginary.}
\label{fig2}
\end{figure} 
\\[1mm]
\indent 6. For  larger values  of the Floquet parameter
 $  \mu $, due to Hamiltonian reasons,   the eigenvalues 
 will remain on the imaginary axis until 
  the Floquet exponent $ \mu $  
  reaches  values close to the next ``collision" between
   two other eigenvalues of $ {\mathcal L}_{0,\mu} $.
   For water waves  in infinite  depth 
   this value is close to $ \mu = 1 / 4 $ and corresponds to 
   eigenvalues close to $ \im 3 / 4 $. 
   These unstable eigenvalues depict
   ellipse-shaped curves, called islands,  that have been described 
   numerically in \cite{DO} and supported by formal expansions   in $ \e $ in \cite{CDT}, see also \cite{Ak}.
\\[1mm]
\indent 7.  In Theorem \ref{thm:simpler} we have described just the two unstable eigenvalues of $\cL_{\mu,\e}$ close to zero.
 There are also two larger 
purely imaginary eigenvalues of order $ \cO(\sqrt{\mu}) $, see Theorem \ref{TeoremoneFinale}.  
We remark that our approach describes 
{\it all} the eigenvalues of $ {\mathcal L}_{\mu,\epsilon} $ close to $0 $ (which are $ 4 $).

\smallskip

Any rigorous 
proof of the Benjamin-Feir instability has to face the difficulty that the  perturbed  eigenvalues bifurcate from  the {\it defective} eigenvalue zero. 
Both Bridges-Mielke \cite{BrM} (see also 
the preprint by Hur-Yang \cite{HY} in finite depth) and 
Nguyen-Strauss \cite{NS},   reduce 
the spectral  problem to a  finite dimensional one, here a $4\times 4$ matrix, 
and, in a suitable regime of values of $ (\mu, \e)$,  
prove the existence of eigenvalues with non-zero real part.
The paper \cite{BrM}, dealing with water waves in finite depth, bases its analysis on 
spatial dynamics and a Hamiltonian  center manifold reduction, as \cite{HY}. 
Such approach fails 
 in infinite depth  (we quote however 
 \cite{IK} for an analogue in infinite depth which carries most of the properties of a center manifold). The proof in \cite{NS} is  based on a Lyapunov-Schmidt decomposition and applies also to the infinite depth case.

Our approach is completely different.
Postponing its detailed description after 
the statement of Theorem \ref{TeoremoneFinale},  we only anticipate some  of its main ingredients. 
The first one is   Kato's theory of similarity transformations  \cite[II-\textsection 4]{Kato}. 
This method is perfectly suited to study splitting of multiple isolated
eigenvalues, for which regular perturbation theory might fail. It  has been used, in a similar context, in the   study  of  infinite dimensional integrable systems   \cite{K, KP, toda, Ma_tame}.

In this paper we
 implement Kato's theory for the  complex operators $\cL_{\mu,\e}$ which have an   Hamiltonian and reversible structure, inherited by  the 
 Hamiltonian \cite{Zak1,CS} and reversible \cite{BM,BBHM,BFM} nature 
of the water waves equations.
We show how Kato's theory  can be used  to  prolong, in an analytic way,  a 
symplectic and reversible basis of the  generalized eigenspace of the unperturbed operator 
$ {\mathcal L}_{0,0} $ into a ($\mu,\e$)-dependent 
symplectic and reversible basis of the corresponding invariant subspace of 
$ {\mathcal L}_{\mu,\e} $.
Thus the restriction of the canonical complex symplectic form to this subspace, 
is  represented, in this symplectic basis,  
by the {\it constant}  symplectic matrix $ {\mathtt J}_4 $ defined in 
\eqref{Lform}, which is {\it independent} of $ (\mu,\e)$. This feature simplifies considerably 
perturbation theory. 

In this way  the spectral problem is reduced to determine the eigenvalues of a   
$4 \times 4$ 
matrix, which  depends  analytically in 
$ \mu, \epsilon $ and it is Hamiltonian and reversible. 
These properties imply strong algebraic features on the  matrix entries, for which we provide  detailed  expansions. 
Next, inspired by KAM  ideas, instead of looking for  zeros of the characteristic polynomial 
of the reduced matrix 
 (as in the periodic Evans function  approach \cite{BJ, HY} or  in  \cite{HJ, NS}), 
we
conjugate it to a block-diagonal matrix whose  $ 2 \times 2  $
diagonal blocks are  Hamiltonian and reversible.
One of these two blocks has  the eigenvalues given by \eqref{eigelemu}, proving the Benjamin-Feir instability figure ``8'' phenomenon.

\smallskip

Let us mention that 
modulational instability has been studied also  for a variety of approximate water waves models, such as KdV, gKdV, NLS and the Whitham equation by, for instance, Whitham \cite{Wh}, Segur, Henderson, Carter and Hammack \cite{SHCH}, Gallay and Haragus \cite{GH}, Haragus and Kapitula \cite{HK}, Bronski and Johnson \cite{BJ}, Johnson \cite{J}, Hur and Johnson \cite{HJ}, Bronski, Hur and Johnson \cite{BHJ},  Hur and Pandey \cite{HP},   Leisman,  Bronski,   Johnson   and
 Marangell \cite{LBJM}.
Also in these approximate models numerical simulations predict 
a figure ``8'' similar to that in  
Figure \ref{figure-eigth} for the bifurcation of the unstable eigenvalues close to zero. 
However, 
in none of  these approximate models (except for the integrable NLS in \cite{DU})
the complete picture of the Benjamin-Feir
instability has  been rigorously proved so far. 
We expect that the approach developed in this paper could be
applicable 
for such equations as well, and also to include the effects of surface tension  in water waves equations (see e.g. \cite{Ak}).

To conclude this introduction,  we  
mention the  nonlinear modulational instability 
result of Jin, Liao, and Lin  \cite{JLL}  for 
several  approximate water waves
models and the preprint by Chen and Su \cite{ChS} for Stokes waves in deep water.
For nonlinear instability results of traveling 
solitary water waves  decaying at infinity on $ \bR $  (not periodic)  we quote \cite{RT} and reference therein.
\\[1mm]{\bf Acknowledgments.}
We thank Bernard Deconinck, Walter Strauss, Huy Nguyen  and Vera Hur   for 
several useful discussions  
that introduced   us to  the fascinating problem of Benjamin-Feir instability.
We also thank David Nicholls, John Toland, Pavel Plotnikov and Erik Wahl\'en
for pointing us some references.

\section{The full water waves Benjamin-Feir  spectrum}\label{sec:2}

In order to give the complete statement  of our  spectral
result, we begin with recapitulating some well known facts about the pure  gravity water waves equations. 
\\[1mm]{\bf The water waves equations and the Stokes waves.}
We consider the Euler equations for a 2-dimensional incompressible,
inviscid, irrotational fluid under the action of  gravity. The fluid fills the
region $
{ \mathcal D}_\eta := \left\{ (x,y)\in \bT\times \bR\;:\; y< \eta(t,x)\right\} $, $ \bT :=\bR/2\pi\bZ $,   with infinite depth
and  space periodic boundary conditions. 
The irrotational velocity field is the gradient  
of a harmonic scalar potential $\Phi=\Phi(t,x,y) $  
determined by its trace $ \psi(t,x)=\Phi(t,x,\eta(t,x)) $ at the free surface
$ y = \eta (t, x ) $.
Actually $\Phi$ is the unique solution of the elliptic equation
$$
    \Delta \Phi = 0 \  \text{ in }\, {\mathcal D}_\eta, \quad
     \Phi(t,x,\eta(t,x)) = \psi(t,x) \, ,  \quad
   \Phi_y(t,x,y) \to 0 \ \text{ as }y\to - \infty \, . 
$$
The time evolution of the fluid is determined by two boundary conditions at the free surface. 
The first is that the fluid particles  remain, along the evolution, on the free surface   (kinematic
boundary condition), and the second one is that the pressure of the fluid  
is equal, at the free surface, to the constant atmospheric pressure  (dynamic boundary condition). Then, as shown by Zakharov \cite{Zak1} and Craig-Sulem \cite{CS}, 
the time evolution of the fluid is determined by the 
following equations for the unknowns $ (\eta (t,x), \psi (t,x)) $,  
\begin{equation}\label{WWeq}
 \eta_t  = G(\eta)\psi \, , \quad 
  \psi_t  =  
- g \eta - \dfrac{\psi_x^2}{2} + \dfrac{1}{2(1+\eta_x^2)} \big( G(\eta) \psi + \eta_x \psi_x \big)^2 \, , 
\end{equation}
where $g > 0 $ is the gravity constant and $G(\eta)$ denotes 
the Dirichlet-Neumann operator $
 [G(\eta)\psi](x) := \Phi_y(x,\eta(x)) -  \Phi_x(x,\eta(x)) \eta _x(x)$. 
It results that $ G(\eta) [\psi] $ has zero average. %   for any $ \eta, \psi $.  
%The equation \eqref{WWeq} 
% is space invariant and  
%  $ \tau_\theta G(\eta)\psi = G( \tau_\theta \eta)[ \tau_\theta \psi] $, $ 
% \tau_\theta u (x) := u (x + \theta )$. 
 %Moreover in infinite depth  
 %$ G(\eta + m, \psi) = G(\eta) $, for any $ m \in \mathbb{R} $.
 With no loss of generality we set the gravity constant $g=1$.
The equations \eqref{WWeq} are the Hamiltonian system
\begin{equation}\label{PoissonTensor}
 \pa_t \vet{\eta}{\psi} = \cJ \vet{\nabla_\eta \mathcal{H}}{\nabla_\psi \mathcal{H}}, \quad \quad \cJ:=\begin{bmatrix} 0 & \uno \\ -\uno & 0 \end{bmatrix},
 \end{equation}
 where $ \nabla $ denote the $ L^2$-gradient, and the Hamiltonian
$  \mathcal{H}(\eta,\psi) :=  \frac12 \int_{\mathbb{T}} \left( \psi \,G(\eta)\psi +\eta^2 \right) \de x
$
is the sum of the kinetic and potential energy of the fluid. 
The associated symplectic $ 2$-form is 
\begin{equation} \label{sympl-form-st}
{\cal W}  
\left( \begin{pmatrix}
\eta_1 \\
\psi_1
\end{pmatrix}, 
\begin{pmatrix}
\eta_2 \\
\psi_2
\end{pmatrix}
\right) 
= (  -  \psi_1 , \eta_2 )_{L^2} + (\eta_1  , \psi_2 )_{L^2}  \, . 
\end{equation}
In addition of being Hamiltonian, the water waves 
system \eqref{WWeq} possesses other important symmetries.
First of all it is time reversible with respect to the involution 
\begin{equation}\label{revrho}
\rho\vet{\eta(x)}{\psi(x)} := \vet{\eta(-x)}{-\psi(-x)}, \quad \text{i.e. }
\mathcal{H} \circ \rho = \mathcal{H} \, ,
\end{equation}
or equivalently the water waves vector field $ X(\eta, \psi) $ 
anticommutes with $ \rho $,  i.e. $ X \circ \rho = - \rho \circ X $.
This property follows noting that the Dirichlet-Neumann operator satisfies 
(see e.g. \cite{BFM}) 
\begin{equation}\label{DN-rev}
G(  \eta^\vee ) [ \psi^\vee ] = \left( G(\eta) [\psi ] \right)^\vee  \quad
\text{where} \quad  f^\vee (x) := f (-  x)  \, . 
\end{equation}
Noteworthy solutions of \eqref{WWeq} are the so-called traveling Stokes waves, namely 
solutions of the form $\eta(t,x)=\breve \eta(x-ct)$ and $\psi(t,x)=\breve \psi(x-ct)$ for some real  $c$   and  $2\pi$-periodic functions  $(\breve \eta (x), \breve \psi (x)) $.
In a reference frame in translational motion with constant speed $c$,  the water waves equations \eqref{WWeq} then become
\begin{equation}\label{travelingWW}
\eta_t  = c\eta_x+G(\eta)\psi \, , \quad 
 \psi_t  = c\psi_x - g \eta - \dfrac{\psi_x^2}{2} + \dfrac{1}{2(1+\eta_x^2)} \big( G(\eta) \psi + \eta_x \psi_x \big)^2  
\end{equation}
and the Stokes waves $(\breve \eta, \breve \psi)$ are  equilibrium 
steady solutions 
of \eqref{travelingWW}. 

The rigorous existence proof of the bifurcation of small amplitude Stokes waves for pure 
  gravity water waves goes back to the works of Levi-Civita \cite{LC}, 
 Nekrasov \cite{Nek},  and Struik \cite{Struik}.
 We denote by $B(r):= \{ x \in \bR \colon \  |x| < r\}$ the real ball with center 0 and radius $r$. 
 \begin{teo}\label{LeviCivita}
{\bf (Stokes waves)} There exist $\e_0 >0$ and a unique family  of real analytic 
 solutions $(\eta_\e(x), \psi_\e(x), c_\e)$, parameterized by the amplitude $|\e| \leq \e_0$, of 
\begin{equation}\label{travelingWWstokes}
c \, \eta_x+G(\eta)\psi = 0 \, , \quad 
c \, \psi_x - g \eta - \dfrac{\psi_x^2}{2} + 
\dfrac{1}{2(1+\eta_x^2)} \big( G(\eta) \psi + \eta_x \psi_x \big)^2  = 0 \, , 
\end{equation}
  such that
 $ \eta_\e (x), \psi_\e (x) $ are $2\pi$-periodic;  $\eta_\e (x) $ is even
and $\psi_\e (x) $ is odd. They have  the expansion 
 \begin{equation}\label{exp:Sto}
 \begin{aligned}
  \eta_\e (x) &=  \e \cos (x) + \frac{\e^2}{2} \cos (2x) + 
  \cO(\e^3)\, , \quad 
  \psi_\e (x) = \e \sin (x) + \frac{\e^2}{2} \sin (2x)   
  +\cO(\e^3) \, ,  \\
  c_\e &= 1 + \frac{1}{2} \e^2+\cO(\e^3) \, .
 \end{aligned}
 \end{equation}
More precisely for any  $ \sigma \geq  0 $ and $ s > \frac52 $, there exists $ \e_0>0 $ such that
the map $\e \mapsto (\eta_\e, \psi_\e, c_\e)$ is analytic from $B(\e_0) \to H^{\sigma,s} (\bT)\times H^{\sigma,s}(\bT)\times \bR$, where 
$ H^{\sigma,s}(\bT) $ is the space of $ 2 \pi $-periodic analytic functions 
$ u(x) = \sum_{k \in \mathbb{Z}} u_k e^{\im k x} $
with $ \| u \|_{\sigma,s}^2 := \sum_{k \in \mathbb{Z}} |u_k|^2 \langle k \rangle^{2s} 
e^{2 \sigma |k|} < + \infty$.
\end{teo}
The existence of solutions of \eqref{travelingWWstokes}
can  nowadays
be   deduced by the  analytic 
Crandall-Rabinowitz 
bifurcation theorem from a simple eigenvalue, see e.g. \cite{BuT}.  Since Lewy \cite{Lewy} it is known that $C^1$
 traveling waves are actually real analytic, see also Nicholls-Reitich \cite{NR}.
The expansion \eqref{exp:Sto} is given for example in \cite[Proposition 2.2]{NS}.
The analyticity result of Theorem \ref{LeviCivita} is explicitely proved in \cite{BMV}. 
We also mention that   more general time quasi-periodic traveling Stokes waves have been recently proved 
for \eqref{WWeq} in \cite{BFM2} in finite depth (actually for any constant vorticity), in \cite{FG} in infinite depth,
and in \cite{BFM} for gravity-capillary water waves with constant vorticity in any depth.  
\\[1mm]{\bf Linearization at the Stokes waves.}\label{goodunknown}
In order to determine the stability/instability of the Stokes  waves given by Theorem \ref{LeviCivita}, 
we linearize  the water waves equations \eqref{travelingWW} with $ c = c_\e $ at  $(\eta_\epsilon(x), \psi_\epsilon(x))$. 
In the sequel we  follow closely   \cite{NS}, but, as in \cite{BFM2,BBHM}, we emphasize the 
Hamiltonian and reversible structures of the linearized equations, 
since these properties  play a crucial role in our proof of the instability result.

By using the shape derivative formula
for the differential $ \de_\eta G(\eta)[\hat \eta ]$ of the Dirichlet-Neumann operator (see e.g. formula (3.4) in  \cite{NS}), one obtains the autonomous real linear system
\begin{equation}\label{linearWW}
 \vet{\hat \eta_t}{\hat \psi_t}
 = \begin{bmatrix} -G(\eta_\e)B-\pa_x \circ (V-c_\e) & G(\eta_\e) \\ -g+B(V-c_\e)\pa_x - B \pa_x \circ (V-c_\e) - BG(\eta_\e)\circ B  & - (V-c_\e)\pa_x + BG(\eta_\e) \end{bmatrix}\vet{\hat \eta}{\hat \psi}
 \end{equation}
 where
$$
V := V(x) := -B (\eta_\e)_x + (\psi_\e)_x \, , \ \  
 B := B(x) := \frac{G(\eta_\e)\psi_\e + (\psi_\e)_x (\eta_\e)_x}{1+(\eta_\e)_x^2}
 = \frac{ (\psi_\e)_x- c_\e}{1+(\eta_\e)_x^2}(\eta_\e)_x
  \, . 
$$
The functions $(V,B)$ are the horizontal and vertical components of the velocity field
$ (\Phi_x, \Phi_y) $ at the free surface. 
Moreover $\e \mapsto (V,B)$ is  analytic as a map 
$B(\e_0) \to H^{\sigma, s-1}(\bT)\times H^{\sigma,s-1}(\bT)$.

The real system \eqref{linearWW} is Hamiltonian, i.e. of the form $ \cJ \mathcal A $
for a symmetric operator  $ \mathcal A = \mathcal A^\top $,
where $\mathcal A^\top$ is the transposed operator with respect the standard real  scalar product of $L^2(\bT, \bR)\times L^2(\bT, \bR)$.

Moreover, since $ \eta_\e $ is
even in $x$ and $ \psi_\e $ is 
odd in $x$, then  the functions $ (V, B) $ are respectively 
even and odd in $ x $. 
Using  also \eqref{DN-rev},  the 
linear operator  in \eqref{linearWW} is   
 reversible, i.e. it anti-commutes with the involution $ \rho $ in \eqref{revrho}. 

Under the time-independent  
``good unknown of Alinhac" linear transformation
\begin{equation}\label{Alin}
 \vet{\hat \eta}{\hat \psi} := Z \vet{u}{v} \, , \qquad  Z = \begin{bmatrix}  1 & 0 \\ B & 1\end{bmatrix}, \quad Z^{-1} = \begin{bmatrix}  1 & 0 \\ -B & 1\end{bmatrix},
\end{equation}
the system \eqref{linearWW} assumes the 
simpler form 
\begin{equation}\label{linearWW2}
 \vet{u_t}{v_t} =\begin{bmatrix} -\pa_x\circ (V-c_\e) & G(\eta_\e) \\ -g - ((V-c_\e) B_x)  & - (V-c_\e)\pa_x \end{bmatrix}\vet{u}{v}.
\end{equation}
Note that, since the transformation $ Z $ is symplectic, i.e.
%\footnote{$ Z^\top $ is the transposed matrix with respect to the real 
%$ L^2 ({\mathbb T}, {\mathbb R}^2)$ scalar product.} 
$ Z^\top \cJ Z = \cJ $, 
and reversibility
preserving, i.e. $ Z \circ \rho = \rho \circ Z $, the linear system \eqref{linearWW2}
is Hamiltonian and reversible as \eqref{linearWW}. 

Next, following Levi-Civita \cite{LC}, 
we perform a conformal change of variables to flatten 
the water surface. By \cite[Prop. 3.3]{NS}, 
or \cite[section 2.4]{BM},  there exists a diffeomorphism of $\mathbb{T}$,
 $ x\mapsto x+\mathfrak{p}(x)$, with a small $2\pi$-periodic  function $\mathfrak{p}(x)$, such that, by defining the associated composition operator $ (\mathfrak{P}u)(x) := u(x+\mathfrak{p}(x))$, the Dirichlet-Neumann operator writes as
$$
 G(\eta) = \pa_x \circ \mathfrak{P}^{-1} \circ {\mathcal H} \circ \mathfrak{P} \, , 
$$
where $ {\mathcal H} $ is the Hilbert transform. 
The function $\mathfrak p(x)$ is  determined as a fixed point  of  
$\mathfrak{p}  =  \cH[\eta_\e\circ(\text{Id} + \mathfrak{p})]$, see e.g. \cite[Proposition 3.3.]{NS} or \cite[formula (2.125)]{BM}.
By the analyticity of the map $\e \to \eta_\e \in H^{\sigma,s}$, $\sigma >0$, $s > 1/2$,  
the 
analytic implicit function theorem\footnote{We use that the composition operator  $ p  \mapsto \eta (x + p(x) ) $ induced by  
an analytic 
function $ \eta (x) $ is analytic on $ H^s (\mathbb T )$ for $s>1/2$.} implies the existence of a solution   $\e \mapsto \mathfrak{p}(x):=\mathfrak{p}_\e(x) $  analytic
as a map 
$B(\e_0) \to H^{s}(\bT)$.
Moreover, since $\eta_\e$ is even,  the function $\mathfrak p(x)$ is odd.

Under the symplectic  and reversibility-preserving map 
\begin{equation}\label{LC}
 \mathcal{P} := \begin{bmatrix}(1+\mathfrak{p}_x)\mathfrak{P} & 0 \\ 0 & \mathfrak{P} \end{bmatrix} \, , 
\end{equation}
($ \mathcal{P} $ preserves the symplectic 2-form % $ \cal W $
 in \eqref{sympl-form-st}
by inspection, and  commutes with $ \rho $ being $ \mathfrak{p}(x) $ odd), 
the system \eqref{linearWW2} transforms into
the linear system $ h_t = \cL_\e h $
 where  $ \cL_\e $ is the Hamiltonian and reversible real operator
\begin{align}\label{cLepsilon}
\cL_\e= \begin{bmatrix} \pa_x \circ (1+p_\e(x)) & |D| \\ - (1+a_\e(x)) &   (1+p_\e(x))\pa_x \end{bmatrix}= \cJ \begin{bmatrix}   1+a_\e(x) &   -(1+p_\e)(x)\pa_x \\ \pa_x \circ (1+p_\e(x)) & |D| \end{bmatrix} 
\end{align}
where 
\begin{equation}\label{def:pa}
 1+p_\e(x) :=  \displaystyle{\frac{ c_\e-V(x+\mathfrak{p}(x))}{ 1+\mathfrak{p}_x(x)}} \, , \quad 1+a_\e(x):=   \displaystyle{\frac{1+ (V(x + \mathfrak{p}(x)) - c_\e)
 B_x(x + \mathfrak{p}(x))  }{1+\mathfrak{p}_x(x)}} \, .
\end{equation}
The functions $p_\e (x) $ and $a_\e (x) $  are even in $ x $ and, 
by the expansion \eqref{exp:Sto} of the Stokes wave,  it results  \cite[Lemma 3.7]{NS}
\begin{align}
p_\e (x) & = - 2 \e \cos (x) + \e^2  \big( \frac32 - 2 \cos (2x) \big) + \cO(\e^3)
= \e p_1 (x) + \e^2 p_2 (x)  + \cO(\e^3) \, ,  \label{SN1}\\
a_\e (x) & = - 2 \e \cos (x) +  \e^2  \big( 2 - 2 \cos (2x) \big) + \cO(\e^3)
= \e a_1(x) +\e^2 a_2 (x) + \cO(\e^3) \, . \label{SN2}
\end{align}  
In addition, by the analiticity results of the functions $ V, B, \mathfrak{p}(x) $ given above,  
the functions   $p_\e$ and $a_\e$   are analytic in $\e$ as maps $B(\e_0)\to H^{s} (\mathbb T)$.
\\[1mm]{\bf Bloch-Floquet expansion.}\label{Bloch-Floquet}
The operator $\cL_\e$ in \eqref{cLepsilon} has $2\pi$-periodic coefficients, so its spectrum on $L^2(\bR, \bC^2)$ is most conveniently described by Bloch-Floquet theory (see e.g. \cite{J} and references therein).
This theory guarantees that  
$$
\sigma_{L^2(\bR)} (\cL_\e ) = \bigcup_{\mu\in [- \frac12, \frac12)} \sigma_{L^2(\bT)}  (\cL_{\mu, \e}) \, ,  
\qquad \cL_{\mu,\e}:= e^{- \im \mu x} \, \cL_\e \, e^{\im \mu x}  \ . 
$$
This 
reduces the problem to study the spectrum of  $\cL_{\mu, \e}$ acting on $L^2(\bT, \bC^2)$ for different values of $\mu$. 
% The advantage is that  this operator has  pure point spectrum.
%and therefore   standard perturbation theory becomes available.	
In particular, if $\lambda$ is an eigenvalue of $\cL_{\mu,\e}$ with eigenvector $v(x)$, then  $h (t,x) = e^{\lambda t} e^{\im \mu x} v(x)$ solves $h_t = \cL_{\e} h$. We remark that:
\\[1mm]
\indent 1.
If $A = \mathrm{Op}(a) $  
is a pseudo-differential  operator with   
 symbol $ a(x, \xi ) $, which is $2\pi$ periodic in the $x$-variable, 
then 
$  A_\mu := e^{- \im \mu x}A  e^{ \im \mu x}  =  \mathrm{Op} (a(x, \xi + \mu )) $
is a pseudo-differential  operator with   symbol $ a(x, \xi + \mu ) $
(which can be proved  e.g. following   Lemma 3.5 of \cite{NS}). 
\\[1mm]
\indent 2.
If $ A$ is a real operator then 
$ \overline{ A_\mu} = A_{- \mu } $. 
As a consequence the spectrum  
\begin{equation}\label{propmu-mu}
 \sigma (A_{-\mu}) = \overline{  \sigma (A_{\mu}) } \, .
\end{equation}
Then we can study $  \sigma (A_{\mu}) $ just for $ \mu > 0 $.
% and we obtain $  \sigma (A_{-\mu}) $   by taking the complex conjugated of the spectrum   $ \sigma (A_{\mu}) $.
Furthermore $\sigma(A_{\mu})$ is a 1-periodic set with respect to $\mu$, so   one can restrict  to  $\mu \in [0, \frac12)$.

By the previous remarks the Floquet operator  associated with the real operator $\cL_\e$ in \eqref{cLepsilon} is  the complex  \emph{Hamiltonian} and \emph{reversible} operator 
(see Definition \ref{def:compl.ham} below) 
\begin{align}\label{WW}
 \cL_{\mu,\e} :&= \begin{bmatrix} (\pa_x+\im\mu)\circ (1+p_\e(x)) & |D+\mu| \\ -(1+a_\e(x)) & (1+p_\e(x))(\pa_x+\im \mu) \end{bmatrix} \\ 
 &= \underbrace{\begin{bmatrix} 0 & \uno\\ -\uno & 0 \end{bmatrix}}_{\displaystyle{=\cJ}} \underbrace{\begin{bmatrix} 1+a_\e(x) & -(1+p_\e(x))(\pa_x+\im \mu) \\ (\pa_x+\im\mu)\circ (1+p_\e(x)) & |D+\mu| \end{bmatrix}}_{\displaystyle{=:\cB_{\mu,\e}}} \, . \notag 
\end{align}
We regard $  \cL_{\mu,\e} $ as an operator with 
domain $H^1(\bT):= H^1(\mathbb{T},\bC^2)$ 
and range $L^2(\bT):=L^2(\mathbb{T},\bC^2)$, equipped with  
the complex scalar product % of $L^2(\mathbb{T})$  defined by
\begin{equation}\label{scalar}
(f,g) := \frac{1}{2\pi} \int_{0}^{2\pi} \left( f_1 \bar{g_1} + f_2 \bar{g_2} \right) \, \text{d} x  \, , 
\quad
\forall f= \vet{f_1}{f_2}, \ \  g= \vet{g_1}{g_2} \in  L^2(\bT, \bC^2) \, .
\end{equation} We also  denote $ \| f \|^2 = (f,f) $.

The complex operator $\cL_{\mu,\e}$ in \eqref{WW}  is \emph{Hamiltonian} 
and \emph{Reversible},  according to the 
following definition.

\begin{sia} {\bf (Complex Hamiltonian/Reversible operator)} \label{def:compl.ham}
A complex operator $\cL : H^1(\bT,\bC^2) \to L^2(\bT,\bC^2) $ 
is 
\\[1mm] 
($i$) \emph{Hamiltonian}, if % it can be written as
$\cL = \cJ \cB $ where $ \cB  $ is a self-adjoint operator, namely 
$ \cB = \cB^*   $, 
where  $\cB^*$ (with domain $H^1(\bT)$)  is the adjoint with respect to the complex scalar product \eqref{scalar} of $L^2(\mathbb{T})$.
\\[1mm] 
($ii$) \emph{Reversible},  if
\begin{equation}\label{Reversible}
 \cL \circ \bro =- \bro \circ \cL \, ,
\end{equation}
where $\bro$ is the complex involution (cfr. \eqref{revrho})
\begin{equation}\label{reversibilityappears}
 \bro \vet{\eta(x)}{\psi(x)} := \vet{\bar\eta(-x)}{-\bar\psi(-x)} \, .
\end{equation}
\end{sia}
The property \eqref{Reversible} for $ \cL_{\mu,\epsilon} $ 
follows  because 
$ \cL_\e $ is a real  operator 
which is reversible with respect to the involution $ \rho $ in \eqref{revrho}.
Equivalently, since $\cJ \circ \bro = -\bro \circ \cJ$, a complex Hamiltonian operator
$ \cL = \cJ \cB $ is reversible, 
if the self-adjoint operator $\cB $ is \emph{reversibility-preserving},  i.e. 
\begin{equation}
\label{B.rev.pres}
\cB \circ \bro = \bro \circ \cB \, .
\end{equation}
We shall deeply exploit these algebraic properties in the proof of Theorem \ref{TeoremoneFinale}.

\smallskip

In addition $(\mu, \e) \to \cL_{\mu,\e} \in \cL(H^1(\bT), L^2(\bT))$ is analytic, 
since  the functions $\e \mapsto a_\e$, $p_\e$  defined in \eqref{SN1}, \eqref{SN2} are analytic as maps $B(\e_0) \to H^1(\bT)$ 
and  ${\mathcal L}_{\mu,\e}$ is linear in $\mu$.
Indeed 
the Fourier multiplier operator $|D+\mu|  $   can be written,  for  any $ \mu \in [-\frac12, \frac12) $,  
as $|D+\mu| = |D| + \mu \sgn(D)+ |\mu | \, \Pi_0 $ and thus  
(see \cite[Section 5.1]{NS})
\begin{equation}\label{|D+mu|}
|D+\mu| =  |D| + \mu (\sgn(D)+\Pi_0) \, , \quad \forall \mu > 0 \, , 
\end{equation}
where $\sgn(D)$ is the Fourier multiplier operator, 
acting on $2\pi$-periodic functions,  
 with symbol 
\begin{equation}\label{def:segno}
 \sgn(k) := 1\  \forall k > 0 \, , \quad \sgn(0):=0 \, ,\quad \sgn(k) := -1 \ \forall k < 0 \, , 
\end{equation}
and $\Pi_0$ is the projector operator on the zero mode, 
$\Pi_0f(x) := \frac{1}{2\pi} \int_\bT f(x)\de x. $

\smallskip

Our aim is to prove the existence of eigenvalues of $  \cL_{\mu,\e}  $ 
with non zero real part. 
We remark that the Hamiltonian structure of $\cL_{\mu,\e}$ implies that eigenvalues with non zero real part may arise only from multiple
  eigenvalues of $\cL_{\mu,0}$, because   if $\lambda$ is an eigenvalue of $\cL_{\mu,\e}$ then  also $-\bar \lambda$ is. 
In particular simple purely imaginary eigenvalues of $\cL_{\mu,0}$ remain on the imaginary axis under perturbation.
%For this reason 
We now carefully describe the spectrum of $\cL_{\mu,0}$.
\\[1mm]{\bf The spectrum of $\cL_{\mu,0}$.}\label{initialspectrum}
The spectrum of the Fourier multiplier matrix operator 
\begin{equation}\label{cLmu}
 \cL_{\mu,0} = \begin{bmatrix} \pa_x+\im\mu & |D+\mu| \\ -1 & \pa_x+\im \mu \end{bmatrix}
\end{equation}
consists of the purely imaginary eigenvalues $\{\lambda_k^\pm(\mu)\;,\; k\in \bZ \} $, where
\begin{align} \label{omeghino}
\lambda_k^\pm(\mu):= \im \big( {\pm} k+\mu \mp \sqrt{|k{\pm}\mu|} \big)  \,  .  
\end{align}
It is easily verified (see e.g. \cite{Nic})
that the eigenvalues  $\lambda_k^\pm(\mu)$ in \eqref{omeghino} may ``collide" only for $\mu=0$ or $\mu=\frac14$. For $\mu=0$ the {real} operator $\cL_{0,0}$ possesses the  eigenvalue $0$ with algebraic multiplicity $4$, 
$$
\lambda_0^+(0) = \lambda_0^-(0) = \lambda_1^+(0) = \lambda_{{1}}^-(0)=0 \, ,
$$
 and geometric multiplicity $3$. A real  basis of the 
 Kernel of $\cL_{0,0}$ is
\begin{align}\label{basestart}
 f_1^+ := \vet{\cos(x)}{\sin(x)}, \quad f_1^{-} := \vet{-\sin(x)}{\cos(x)},\qquad f_0^-:=\vet{0}{1} \, ,
\end{align}
 together with the generalized eigenvector 
\begin{align}\label{basestartadd} 
f_0^+:=\vet{1}{0} , \qquad \cL_{0,0}f_0^+ =-f_0^- \, . 
\end{align} 
Furthermore $0$ is an isolated eigenvalue for $\cL_{0,0}$, namely the 
 spectrum   $\sigma\left(\cL_{0,0}\right)  $ decomposes in two separated parts
\begin{equation}
\label{spettrodiviso0}
\sigma\left(\cL_{0,0}\right) = \sigma'\left(\cL_{0,0}\right) \cup \sigma''\left(\cL_{0,0}\right)
\quad \text{where} \quad  \sigma'(\cL_{0,0}):=\{0\} 
\end{equation}
and 
$$
 \sigma''(\cL_{0,0}):= \big\{ \lambda_k^\sigma(0),\ {k\neq 0,1 ,\ \sigma=\pm} \big\} ,
 % \quad \text{ with } \Sigma := \big\{ (+,1),\,(-,-1),\,(+,0),\,(-,0) \big\} \
$$
Note that  $ \sigma''(\cL_{0,0})$  
is contained in $\{ \lambda \in \im \bR \, : \, |\lambda| \geq 2-\sqrt{2}\}$.

We shall also use that, as proved in Theorem 4.1 in \cite{NS},
the operator $ {\mathcal L}_{0,\e} $ 
possesses, for any sufficiently small $\e \neq 0$, 
 the eigenvalue $ 0 $
with a four 
dimensional 
generalized  Kernel,   
%  \cW_\e := \text{span} \{ U_1, \tilde U_2, U_3, U_4 \} $, 
spanned 
by  $ \e $-dependent 
vectors $ U_1, \tilde U_2, U_3, U_4 $ satisfying, for some real constant $ \alpha_\e $, 
 \begin{equation}\label{genespace} 
 {\mathcal L}_{0,\e} U_1 = 0 \, , \  \ 
 {\mathcal L}_{0,\e} \tilde U_2 = 0 \, ,  \  \  {\mathcal L}_{0,\e}  U_3 =  \alpha_\e \,  \tilde U_2 \, , \ 
 \  {\mathcal L}_{0,\e}  U_4 =  - U_1 \, , \quad U_1 = \vet{0}{1} \, . 
 \end{equation}
 By Kato's perturbation theory (see  Lemma \ref{lem:Kato1} below)
for any $\mu, \e \neq 0$  sufficiently small, the perturbed spectrum
$\sigma\left(\cL_{\mu,\e}\right) $ admits a disjoint decomposition as 
\begin{equation}\label{SSE}
\sigma\left(\cL_{\mu,\e}\right) = \sigma'\left(\cL_{\mu,\e}\right) \cup \sigma''\left(\cL_{\mu,\e}\right) \, ,
\end{equation}
where $ \sigma'\left(\cL_{\mu,\e}\right)$  consists of 4 eigenvalues close to 0. 
We denote by $\cV_{\mu, \e}$   the spectral subspace associated with  $\sigma'\left(\cL_{\mu,\e}\right) $, which   has  dimension 4 and it is  invariant by $\cL_{\mu, \e}$.
Our goal is to prove that,  for $ \e $ small, for values of the Floquet
exponent $ \mu $ in an interval of order $ \e $,  the $4\times 4$ matrix 
which represents
 the operator  $ \cL_{\mu,\e} : \mathcal{V}_{\mu,\e} \to  \mathcal{V}_{\mu,\e} $
 possesses a pair of eigenvalues close to zero 
with opposite non zero real parts.

Before stating our main result, let us introduce a notation we shall use through all the paper:

\begin{itemize}\label{notation}
\item[$\bullet$] {\bf  Notation:}
we denote by  $\cO(\mu^{m_1}\e^{n_1},\dots,\mu^{m_p}\e^{n_p})$, $ m_j, n_j \in \bN  $, 
analytic functions of $(\mu,\e)$ with values in a Banach space $X$ which satisfy, for some $ C > 0 $, the bound
 $\|\cO(\mu^{m_j}\e^{n_j})\|_X \leq C \sum_{j = 1}^p |\mu|^{m_j}|\e|^{n_j}$ 
 for small values of $(\mu, \e)$. 
We denote $r_k (\mu^{m_1}\e^{n_1},\dots,\mu^{m_p}\e^{n_p}) $
scalar  functions  $\cO(\mu^{m_1}\e^{n_1},\dots,\mu^{m_p}\e^{n_p})$ which are  also {\it real} analytic.
\end{itemize}

Our complete  spectral result is the following: 
\begin{teo}\label{TeoremoneFinale}
{\bf (Complete Benjamin-Feir spectrum)}
There exist $ \e_0, \mu_0>0 $  such that, 
for any  $ 0\leq \mu < \mu_0 $ and $ 0\leq \e < \e_0 $, 
 the operator $ \cL_{\mu,\e} : \mathcal{V}_{\mu,\e} \to  \mathcal{V}_{\mu,\e} $ 
 can be represented by a $4\times 4$ matrix of the form 
 \begin{equation} \label{matricefinae}
  \begin{pmatrix} \mathtt{U} & \vline & 0 \\ \hline 0 & \vline & \mathtt{S} \end{pmatrix},
 \end{equation}
 where $ \mathtt{U} $ and $ \mathtt{S} $ are  $ 2 \times 2 $
matrices of the form 
 \begin{align}\label{UU}
& \mathtt{U} := 
\begin{pmatrix} 
\im \big( \frac12\mu+ r(\mu\e^2,\mu^2\e,\mu^3) \big) & -\frac{\mu^2}{8}(1+r_5(\e,\mu))
  \\
 \frac{\mu^2}{8}(1+r_1(\e,\mu)) -\e^2(1+r_1'(\e,\mu\e^2)) & \im \big( \frac12\mu+ r(\mu\e^2,\mu^2\e,\mu^3) \big)\end{pmatrix}, \\
 &\label{S}  \mathtt{S} := 
 \begin{pmatrix} 
  \im \mu \big( 1+ r_9(\e^2,\mu\e,\mu^2)\big) & \mu+ r_{10}(\mu^2\e,\mu^3) \\
  -1- r_8(\e^2,\mu^2\e,\mu^3) &  \im\mu  \big( 1+ r_9(\e^2,\mu\e,\mu^2)\big) 
 \end{pmatrix},
\end{align}
where in each of the two matrices the diagonal entries are identical. 
The eigenvalues of the matrix $ \mathtt{U} $ are  given by
$$
\lambda_1^\pm(\mu,\e) =   \frac12\im\mu+\im r(\mu\e^2,\mu^2\e,\mu^3)\pm \frac{\mu}{8}\sqrt{8\e^2\big(1+r_0(\e,\mu)\big)-\mu^2\big(1+r_0'(\e,\mu)\big)}  \ . 
$$
Note that if $8\e^2 (1+r_0(\e,\mu))-\mu^2 (1+r_0'(\e,\mu)) > 0 $, respectively $<0$, the 
 eigenvalues $\lambda^\pm_1(\mu,\e)$ have a nontrivial real part, respectively are purely imaginary. 
 
The eigenvalues of the matrix $ \mathtt{S} $ are a pair of purely imaginary eigenvalues of the form 
$$
 \lambda_0^\pm (\mu, \e) = {\mp} \im\sqrt{\mu}\big(1+ r'(\e^2,\mu \e,\mu^2 )\big)+ \im\mu \big(1+r_9(\e^2,\mu\e,\mu^2)\big)\, . 
$$
For $ \e = 0$ the eigenvalues $ \lambda_1^\pm(\mu,0),  \lambda_0^\pm (\mu,0) $ 
coincide with those in \eqref{omeghino}.
\end{teo}

We conclude this section describing in detail our  approach.
\\[1mm]
{\bf Ideas  and scheme of proof.} 
 We first write the opetator $\cL_{\mu,\e} = \im \mu + {\sL}_{\mu,\e} $ as in \eqref{calL} 
and we aim to  construct a  basis of  $\cV_{\mu,\e}$
 to represent ${\sL}_{\mu,\e}\vert_{\cV_{\mu,\e}}$ as a convenient  $ 4\times 4$ matrix. 
The unperturbed operator $ {\sL}_{0,0}\vert_{\cV_{0,0}}$ 
possesses $ 0 $ as  isolated  eigenvalue with algebraic multiplicity 4  and  generalized kernel 
$\cV_{0,0}$  spanned  by the vectors $\{f_1^\pm, f_0^\pm\}$  in \eqref{basestart}, \eqref{basestartadd}. 

Exploiting  Kato's theory of  similarity transformations for separated 
eigenvalues we prolong 
the unperturbed symplectic basis
$\{f_1^\pm, f_0^\pm\}$ of $\cV_{0,0}$ 
into a symplectic basis of $\cV_{\mu,\e}$ (cfr. Definition \ref{def:SR}),  depending 
analytically on $\mu, \e $. In Lemma \ref{lem:Kato1}
we  construct the transformation  operator $U_{\mu,\e}$, see \eqref{OperatorU}, 
which is invertible and analytic in  $\mu,\e$, and maps isomorphically $\cV_{0,0}$ into $\cV_{\mu,\e}$.
Furthermore,  since $ {\sL}_{\mu,\e}$ is Hamiltonian and reversible, we prove in Lemma
\ref{propPU} that  the operator $U_{\mu,\e}$ is symplectic and reversibility preserving.
This implies that the vectors
$ f^\sigma_k(\mu,\e) := U_{\mu,\e}f_k^\sigma$, $ k = 0,1$, $\sigma = \pm$, 
form a symplectic and reversible basis of   $\cV_{\mu,\e} $, 
according to Definition \ref{def:SR}. 

This  construction has the  following interpretation 
in the setting of complex symplectic structures, cfr. \cite{Arn,EM}.
The complex symplectic form
\eqref{ses} restricted to the symplectic subspace  $ \mathcal{V}_{\mu,\e} $ is represented, in   
the $ (\mu, \e )$-dependent  symplectic basis $ f^\sigma_k(\mu,\e)$, 
by the {\it constant} antisymmetric matrix $ \tJ_4 $ defined in \eqref{Lform},
for {\it any} value of $ (\mu, \e )$. In this sense $ U_{\mu,\e} $
is acting as a ``Darboux transformation''.  
Consequently, the 
 Hamiltonian and reversible operator  $ {\sL}_{\mu,\e}\vert_{\cV_{\mu,\e}}$ is represented, in 
the symplectic basis $ f^\sigma_k(\mu,\e)$, 
by a $4\times 4$ matrix 
 of the form $\tJ_4 \tB_{\mu,\e}$ with $\tB_{\mu,\e}$ selfadjoint, see  Lemma \ref{lem:B.mat}.
This  property simplifies considerably the perturbation theory of the spectrum 
(we refer to \cite{Olver} for a discussion, in a different context, 
of the difficulties raised by parameter-dependent 
symplectic forms).

We then modify the  basis $\{ f^\sigma_k(\mu,\e)\} $ to construct a new  symplectic and reversible basis $\{g_k^\sigma(\mu,\e)\} $
of $\cV_{\mu,\e}$, still depending analytically on $\mu,\e$, 
with the additional property that $g_1^-(0,\e)$ has zero space average;
 this  property  plays  a crucial role in the expansion obtained in 
Lemma \ref{lem:B2}, 
necessary to exhibit the Benjamin-Feir instability phenomenon, see Remark \ref{rem:newbasis}.
 By construction,
the eigenvalues of the $4\times 4$ matrix $\tL_{\mu,\e}$, 
representing the action of the operator $ {\sL}_{\mu,\e }$ on the basis $ \{g_k^\sigma(\mu,\e)\} $,  coincide with the portion of the spectrum $\sigma'({\sL}_{\mu,\e})$ close to zero, defined in 
\eqref{SSE}.
In  Proposition \ref{BexpG} we prove that the  $4\times 4$ Hamiltonian and reversible matrix $\tL_{\mu,\e}$  
has the form
\begin{equation}
\label{L.form}
\tL_{\mu,\e}= 
\tJ_4 \begin{pmatrix} 
 E &  F \\ 
 F^* &  G 
\end{pmatrix}
=
\begin{pmatrix} 
\tJ_2 E & \tJ_2 F \\ 
\tJ_2 F^* &\tJ_2 G 
\end{pmatrix} ,
\end{equation}
where $\tJ_2 = \begin{psmallmatrix} 0 & 1 \\ -1 & 0 \end{psmallmatrix}$ and $E = E^*$, $G = G^*$ and $F$ are $2 \times 2$ matrices having the expansions
\eqref{BinG1}-\eqref{BinG3}.
To compute these  expansions --from which the Benjamin-Feir instability will emerge-- we use two ingredients.
First we Taylor expand   $(\mu, \e) \mapsto U_{\mu,\e}$ in  Lemma \ref{lem:U.expansion}. 
The  Taylor expansion of $U_{\mu,\e}$  is
not a symplectic operator, but this is no longer important to
compute the expansions \eqref{BinG1}-\eqref{BinG3} of the  
matrix $\tL_{\mu,\e}$. 
We  used that $U_{\mu,\e}$  is symplectic 
to prove the Hamiltonian structure  \eqref{L.form} of $\tL_{\mu,\e}$.
The second ingredient is a careful  analysis of $ \tL_{0,\e}$  and 
$\partial_\mu  \tL_{\mu,\e}\vert_{\mu = 0}$.
In particular we prove that 
the $(2,2)$-entry of the matrix $E$ in \eqref{BinG1} does not have any term 
 $\cO ( \e^m )$ nor $ \cO ( \mu \e^m ) $ {\em for any} $ m \in \bN$.
These terms would be dangerous because they might change the sign of the 
entry $(2,2)$ of the matrix $E$ in  \eqref{BinG1} which instead is always negative.
This is crucial to prove the Benjamin-Feir instability, as we explain below.
We show the absence of  terms  $\cO(\e^m)$, $ m \in \bN$,  
 fully exploiting  the  structural information \eqref{genespace}
concerning the four dimensional generalized Kernel of 
 the operator $\cL_{0,\e}$ for any $\e >0$, see Lemma \ref{lem:B1}.
The absence of terms   $\cO(\mu \e^m)$, $ m \in \bN $, is due to the properties of the   basis $\{ g_k^\sigma(\mu,\e)\}$ (see Remark \ref{rem:newbasis}) and it is the motivation for modifying the original basis $\{ f^\sigma_k(\mu,\e)\}$.

Thanks to this analysis,  the  $2 \times 2$ matrix 
\begin{equation}\label{J2E}
\tJ_2 E =  \begin{pmatrix} 
- \im \big( \frac{\mu}{2}+ r_2(\mu\e^2,\mu^2\e,\mu^3) \big)  & -\frac{\mu^2}{8}(1+r_5(\e,\mu))\\
- \e^2(1+r_1'(\e,\mu\e^2))+\frac{\mu^2}{8}(1+r_1''(\e,\mu))  & - \im \big( \frac{\mu}{2}+ r_2(\mu\e^2,\mu^2\e,\mu^3) \big)   \\
 \end{pmatrix}
\end{equation}
possesses two eigenvalues with non-zero real part
--we say that it exhibits the Benjamin-Feir phenomenon-- as long as the two off-diagonal elements have the same sign,
which happens for $ 0 < \mu < \bar \mu (\e)  $ with $ \bar \mu (\e) \sim 2 \sqrt{2} \e $.  
On the other hand the  $ 2 \times 2 $  matrix  $\tJ_2 G$ 
 has  purely imaginary eigenvalues for $ \mu > 0 $ of order $\cO(\sqrt{\mu})$.
In order to prove that the complete $ 4 \times 4 $ matrix  
$ \tL_{\mu,\e} $ in \eqref{L.form} possesses 
 Benjamin-Feir unstable eigenvalues as well, our aim is to eliminate 
 the  coupling term $  \tJ_2 F $.
This is done in Section \ref{sec:block} by a  block diagonalization procedure, inspired by KAM theory. 
This is a singular perturbation problem because the 
spectrum of the matrices $\tJ_2 E $ and $\tJ_2 G$ tends to $ 0 $
as $ \mu \to 0 $. 
We construct a symplectic and reversibility preserving
 block-diagonalization transformation in three steps: 
\\[1mm]
\indent 1. \emph{First step of block-diagonalization 
(Section \ref{sec:omue})}.
Note that the  spectral gap between the $ 2 $ block 
matrices $ \tJ_2 E $ and $ \tJ_2 G $ is of order $ \cO (\sqrt{\mu} )$,
whereas the   entry  $ F_{11} $ of the matrix $ F $  has  size $ \cO(\e^3) $.
In Section  \ref{sec:omue}  we perform a symplectic and reversibility-preserving change of  coordinates  removing $F_{11}$ and  conjugating  $ \tL_{\mu,\e}$ to a new Hamiltonian and reversible matrix  $\tL^{(1)}_{\mu,\e}$ whose block-off-diagonal matrix $\tJ_2 F^{(1)}$ has size $\cO(\mu \e, \mu^3)$ and $\tJ_2 E^{(1)} $ has  the same form \eqref{J2E}, and therefore
possesses  Benjamin-Feir unstable eigenvalues as well. 
This transformation is inspired   by the Jordan normal form of  $ \tL_{0,\e}$. 
\\[1mm]
\indent 2. \emph{Second step of block-diagonalization
 (Section \ref{sec:5.2})}.
We next perform a step of block-diagonalization to decrease further the size of the off-diagonal blocks: by applying a procedure inspired by KAM theory we obtain (at least) a  $ \cO( \mu^2 ) $ factor  in each entries of 
$ F^{(2)} $ in \eqref{Bsylvy}  (by contrast note the presence of $ \cO(\mu \e) $ entries   in $F^{(1)}$).
To achieve this, we construct 
 a linear change of variables  that conjugates
 the matrix $\tL^{(1)}_{\mu,\e}$ to the new 
Hamiltonian and reversible
matrix $ \tL_{\mu,\e}^{(2)}  $ in \eqref{sylvydec}, 
where the new off-diagonal matrix $\tJ_2 F^{(2)}$  is much smaller  than $\tJ_2 F^{(1)}$.
The delicate point, for which we
perform  Step 2 separately than Step 3, is to estimate the new block-diagonal matrices after
the conjugation, and prove that
$\tJ_2 E^{(2)}$ 
has still the form \eqref{J2E} -- thus possessing Benjamin-Feir unstable eigenvalues. Let us elaborate on this.
In order to reduce the size of $\tJ_2 F^{(1)} $, 
we   conjugate  $\tL_{\mu,\e}^{(1)}$ by 
the symplectic  matrix $\exp(S^{(1)})$, where $S^{(1)}$  is a Hamiltonian matrix
with  the same form  of 
$ \tJ_2 F^{(1)} $, see \eqref{formaS}.  
The transformed  matrix $\tL_{\mu,\e}^{(2)} = 
\exp(S^{(1)}) \tL_{\mu,\e}^{(1)}\exp(-S^{(1)})
$  has the 
Lie expansion\footnote{recall that   $\exp(S) L \exp(-S) = \sum_{n \geq 0} \frac{1}{n!} \textup{ad}_S^n(L)$, where
$\textup{ad}_S^0(L) := L$,   $\textup{ad}_S^n(L) = [S,  \textup{ad}_S^{n-1}(L)]$ for $n \geq 1$.}
\begin{equation}\label{L2.lie.expl}
\begin{aligned}
\tL_{\mu,\e}^{(2)} 
& =
\begin{psmallmatrix} 
\tJ_2 E^{(1)} & 0 \\ 
0 &\tJ_2  G^{(1)}
\end{psmallmatrix} \\
& \quad + 
\begin{psmallmatrix} 
0 & \tJ_2 F^{(1)} \\ 
\tJ_2 [F^{(1)}]^* & 0 
\end{psmallmatrix}
 +\lie{S^{(1)}}{\begin{psmallmatrix} 
\tJ_2 E^{(1)} & 0 \\ 
0 &\tJ_2  G^{(1)}
\end{psmallmatrix} }\\
& \quad 
+ \frac12 \Big[ S^{(1)}, \Big[ S^{(1)}, \begin{psmallmatrix} 
\tJ_2 E^{(1)} & 0 \\ 
0 &\tJ_2  G^{(1)}
\end{psmallmatrix} \Big] \Big] 
   + \Big[ S^{(1)}, \begin{psmallmatrix} 
0 & \tJ_2 F^{(1)} \\ 
\tJ_2 [F^{(1)}]^* & 0 
\end{psmallmatrix} \Big]  + \mbox{h.o.t.}
\end{aligned}
\end{equation}
The first line in the right hand side of \eqref{L2.lie.expl} is the original 
block-diagonal matrix, the second line of \eqref{L2.lie.expl} is a purely off-diagonal matrix and the third line is the sum of two  block-diagonal matrices  and ``h.o.t."  collects terms of much smaller size.
We determine $S^{(1)}$ in such a way that 
the second line of \eqref{L2.lie.expl} vanishes 
 (this equation would be referred to as the ``homological equation'' in the context of KAM theory).  
In this way the  remaining off-diagonal matrices (appearing in the h.o.t. remainder) are much smaller in size. 
We then compute 
the block-diagonal corrections in the third line of \eqref{L2.lie.expl} and show that the new block-diagonal matrix $
\tJ_2 E^{(2)} $  has again the form \eqref{J2E} (clearly with different remainders, but of the same order) and thus displays Benjamin-Feir instability.
This last step is delicate because  $S^{(1)} = \cO(\e, \mu^2)$ and $\tJ_2 F^{(1)} = \cO ( \mu \epsilon, \mu^3 )$ and so the matrix in the third line
of \eqref{L2.lie.expl} could a priori have 
elements of size  $\cO(\mu \e^2)$. 
Adding a term of size $\cO(\mu \e^2)$ 
to the (1,2)-entry of the matrix $\tJ_2 E^{(1)}$, which has 
the form $ -\frac{\mu^2}{8}(1+r_5(\e,\mu)) $ 
as 
in \eqref{J2E},  could make it positive. In such a case 
the  eigenvalues of $\tJ_2 E^{(2)}$ would be purely imaginary, 
and the Benjamin-Feir instability would disappear. 
Actually,  estimating individually
each components, we
show that  no contribution of size $\cO(\mu \e^2)$ appears in the (1,2)-entry.

One further comment is needed. We solve the required homological equation without 
diagonalizing $\tJ_2 E^{(1)}$ and  $\tJ_2 G^{(1)}$ (as done typically in  KAM theory).
Note that diagonalization is not even possible at $\mu \sim 2 \sqrt{2}\e$ 
where %they collide at the top of the ``8'' figure and 
$\tJ_2 E^{(1)}$ becomes a Jordan block (here its eigenvalues fail to be analytic).
We use a direct linear algebra argument that  enables    
to preserve the analyticity in $\mu, \e$ of 
 the  transformed 
$4\times 4$ matrix $\tL^{(2)}_{\mu,\e}$.
\\[1mm]
\indent 3. \emph{Complete block-diagonalization 
 (Section \ref{section34})}.
As a last step in Lemma \ref{ultimate}
we perform, by means of a standard implicit function theorem,  
a symplectic and reversibility preserving transformation
that block-diagonalize $\tL^{(2)}_{\mu,\e}$ completely. 
The invertibility properties and estimates required to apply the implicit function theorem
rely on the solution of the homological equation obtained in Step 2.
 The off-diagonal matrix $\tJ_2 F^{(2)}$ is small enough 
to directly 
prove that the block-diagonal matrix  $
\tJ_2 E^{(3)} $ has the same form of $
\tJ_2 E^{(2)} $, thus possesses Benjamin-Feir unstable eigenvalues
(without distinguishing the entries as we do in Step 2).

In conclusion, the original matrix $\tL_{\mu,\e}$ in \eqref{L.form} has been conjugated to the Hamiltonian and reversible matrix \eqref{matricefinae}.
 This proves 
Theorem \ref{TeoremoneFinale} and Theorem \ref{thm:simpler}.

\section{Perturbative approach to the separated eigenvalues}\label{Katoapp}

In this section we apply  Kato's similarity transformation 
theory \cite[I-\textsection 4-6, II-\textsection 4]{Kato}  to study the splitting of the eigenvalues of 
$ \cL_{\mu,\e} $ close to $ 0 $ for small values of $ \mu $ and $ \e $.
First of all it is convenient to decompose the operator $ \cL_{\mu,\e}$ in \eqref{WW} as 
  \begin{equation}\label{calL}
 \cL_{\mu,\e}  = \im \mu + {\sL}_{\mu,\e} \, , \qquad \mu > 0 \, ,  
\end{equation}
where, using also \eqref{|D+mu|}, 
\begin{equation}\label{calL2}
{\sL}_{\mu,\e}:= 
\begin{bmatrix} \pa_x\circ (1+p_\e(x)) + \im \mu \,  p_\e(x) &
 |D| + \mu (\sgn(D)+\Pi_0) \\ -(1+a_\e(x)) & (1+p_\e(x))\pa_x+\im \mu\,  p_\e(x) \end{bmatrix}  \, . 
\end{equation}
 The operator ${\sL}_{\mu,\e}$ is still  Hamiltonian, having the form 
\begin{equation}\label{calL.ham}
{\sL}_{\mu,\e} = 
 \cJ \, {\cal B}_{\mu, \e} \, , \quad
{\cal B}_{\mu, \e}
:= \begin{bmatrix} 1+a_\e(x) & -((1+p_\e(x))\pa_x-\im \mu\,  p_\e(x)  \\
 \pa_x\circ (1+p_\e(x)) + \im \mu \,  p_\e(x) &  |D| + \mu (\sgn(D)+\Pi_0)\end{bmatrix}
\end{equation}
with ${\cal B}_{\mu, \e}$ selfadjoint, and 
 it is also  reversible, namely it satisfies, by \eqref{Reversible}, 
 \begin{equation}\label{calL.rev}
{\sL}_{\mu,\e}\circ \bro =- \bro \circ {\sL}_{\mu,\e} \, , 
\qquad
\bro \mbox{ defined in }  \eqref{reversibilityappears} \, , 
 \end{equation}
 whereas ${\cal B}_{\mu,\e}$ is reversibility-preserving, i.e. fulfills \eqref{B.rev.pres}. Note also that ${\cal B}_{0,\e}$ is a real operator.

 The scalar operator $ \im \mu \equiv \im \mu \, \text{Id}$ 
just translates the spectrum of $ {\sL}_{\mu,\e}$ 
along the imaginary  axis of  the quantity $ \im \mu $, that is, in view of \eqref{calL}, 
$$ 
\sigma ({\mathcal L}_{\mu,\e}) = \im \mu + \sigma ({\sL}_{\mu,\e}) \, . 
$$ 
Thus in the sequel we focus on studying the spectrum of 
$ {\sL}_{\mu,\e}$.

Note also that ${\sL}_{0,\e} = \cL_{0,\e}$ for any $\e \geq 0$.
In particular ${\sL}_{0,0}$ has zero as isolated eigenvalue with algebraic multiplicity 4, geometric multiplicity 3 and generalized kernel spanned by the vectors  $\{f^+_1, f^-_1, f^+_0, f^-_0\}$ in \eqref{basestart}, \eqref{basestartadd}.
Furthermore its spectrum is separated as in \eqref{spettrodiviso0}.
For  any $\e \neq 0$ small, ${\sL}_{0,\e}$ has  zero as  isolated eigenvalue 
with geometric multiplicity $2$, and two generalized eigenvectors satisfying 
\eqref{genespace}. % algebraic multiplicity 4 and .

We also remark that, in view of \eqref{|D+mu|}, the operator
 ${\sL}_{\mu,\e}$ is linear in $\mu$.
We remind that $ {\sL}_{\mu,\e}  : Y \subset X \to X $   
has domain $Y:=H^1(\mathbb{T}):=H^1(\mathbb{T},\bC^2)$ and range $X:=L^2(\mathbb{T}):=L^2(\mathbb{T},\bC^2)$.

In the next lemma  we construct the  transformation operators which map isomorphically the  unperturbed spectral subspace into the perturbed ones.

\begin{lem}\label{lem:Kato1}
 Let $\Gamma$ be a closed, counterclockwise-oriented curve around $0$ in the complex plane separating $\sigma'\left({\sL}_{0,0}\right)=\{0\}$
  and the other part of the spectrum $\sigma''\left({\sL}_{0,0}\right)$ in \eqref{spettrodiviso0}.
There exist $\e_0, \mu_0>0$  such that for any $(\mu, \e) \in B(\mu_0)\times B(\e_0)$  the following statements hold:
\begin{enumerate}
\item  
The curve $\Gamma$ belongs to the resolvent set of 
the operator ${\sL}_{\mu,\e} : Y \subset X \to X $ defined in \eqref{calL2}.
\item  
The operators
\begin{equation}\label{Pproj}
 P_{\mu,\e} := -\frac{1}{2\pi\im}\oint_\Gamma ({\sL}_{\mu,\e}-\lambda)^{-1} \de\lambda : X \to Y 
\end{equation}  
are well defined projectors commuting  with ${\sL}_{\mu,\e}$,  i.e. 
\begin{align}\label{Pprop}
 P_{\mu,\e}^2 = P_{\mu,\e} \, ,
 \quad P_{\mu,\e}{\sL}_{\mu,\e} = {\sL}_{\mu,\e} P_{\mu,\e} \, . 
\end{align}
The map $(\mu, \epsilon)\mapsto P_{\mu,\epsilon}$ is  analytic from 
$B({\mu_0})\times B({\epsilon_0})$
 to $ \cL(X, Y)$.
\item \label{invitem}
The domain $Y$  of the operator ${\sL}_{\mu,\e}$ decomposes as  the direct sum
\begin{align}\label{Y.dec}
Y= \mathcal{V}_{\mu,\e} \oplus \text{Ker}(P_{\mu,\e}) \, , \quad \mathcal{V}_{\mu,\e}:=\text{Rg}(P_{\mu,\e})=\text{Ker}(\uno-P_{\mu,\e}) \, ,
\end{align}
of the  closed  subspaces $\mathcal{V}_{\mu,\e} $, $ \text{Ker}(P_{\mu,\e}) $  of $ Y $, 
which are invariant under ${\sL}_{\mu,\e}$,
$$
{\sL}_{\mu,\e} : \mathcal{V}_{\mu,\e} \to \mathcal{V}_{\mu,\e} \, , \qquad 
{\sL}_{\mu,\e} : \text{Ker}(P_{\mu,\e}) \to \text{Ker}(P_{\mu,\e}) \, . 
$$ 
Moreover 
\begin{equation}\label{dec.spectrum}
\begin{aligned}
&\sigma({\sL}_{\mu,\e})\cap \{ z \in \bC \mbox{ inside } \Gamma \} = \sigma({\sL}_{\mu,\e}\vert_{{\mathcal V}_{\mu,\e}} )  = \sigma'({\sL}_{\mu, \e}) , \\
&\sigma({\sL}_{\mu,\e})\cap \{ z \in \bC \mbox{ outside } \Gamma \} = \sigma({\sL}_{\mu,\e}\vert_{Ker(P_{\mu,\e})} )  = \sigma''( {\sL}_{\mu, \e}) \ ,
\end{aligned}
\end{equation}
proving the  ``semicontinuity property" \eqref{SSE} of separated parts of the spectrum.
\item \label{tutto} The projectors $P_{\mu,\e}$ 
are similar one to each other: the  transformation operators\footnote{
 The operator $(\uno-R)^{-\frac12} $ is defined, for any 
operator $ R $ satisfying $\|R\|_{{\cL}(Y)}<1 $,  by the power series
\begin{align}\label{rootexp}
 (\uno - R)^{-\frac12} :=  \sum_{k=0}^\infty {-1/2 \choose k}(-R)^k = \uno + \frac{1}{2}R + \frac{3}{8}R^2+\cO(R^3) \, .
\end{align}
}
\begin{equation} \label{OperatorU} 
U_{\mu,\e}   := 
\big( \uno-(P_{\mu,\e}-P_{0,0})^2 \big)^{-1/2} \big[ 
P_{\mu,\e}P_{0,0} + (\uno - P_{\mu,\e})(\uno-P_{0,0}) \big] 
\end{equation}
are bounded and  invertible in $ Y $ and in $ X $, with inverse
\begin{equation}
 \label{Uinv}
U_{\mu,\e}^{-1}  = 
 \big[ 
P_{0,0} P_{\mu,\e}+(\uno-P_{0,0}) (\uno - P_{\mu,\e}) \big] \big( \uno-(P_{\mu,\e}-P_{0,0})^2 \big)^{-1/2} \, , 
\end{equation}
 and 
\begin{equation}\label{U.conj}
U_{\mu,\e} P_{0,0}U_{\mu,\e}^{-1} =  P_{\mu,\e}   \, , 
\qquad U_{\mu,\e}^{-1} P_{\mu,\e}  U_{\mu,\e} = P_{0,0} \, .
\end{equation}
The map $(\mu, \epsilon)\mapsto  U_{\mu,\e}$ is analytic from  $B(\mu_0)\times B(\epsilon_0)$ to $\cL(Y)$.
\item \label{facile} The subspaces $\mathcal{V}_{\mu,\e}=\text{Rg}(P_{\mu,\e})$ are isomorphic one to each other: 
$
\mathcal{V}_{\mu,\e}=  U_{\mu,\e}\mathcal{V}_{0,0}.
$
 In particular $\dim \mathcal{V}_{\mu,\e} = \dim \mathcal{V}_{0,0}=4 $, for any 
 $(\mu, \e) \in B(\mu_0)\times B(\e_0)$.
\end{enumerate}
\end{lem}

\begin{proof}
1. For any $ \lambda \in \bC $ we decompose 
${\sL}_{\mu,\e}-\lambda= {\sL}_{0,0}-\lambda + {\cal R}_{\mu,\e} $
where ${\sL}_{0,0} = \begin{bmatrix} \pa_x & |D| \\ -1 & \pa_x \end{bmatrix}$ and 
\begin{equation}\label{restmue}
{\cal R}_{\mu,\e}:={\sL}_{\mu,\e}-{\sL}_{0,0} = \begin{bmatrix}  (\pa_x +\im \mu) p_\e(x) & \mu g(D) \\ -a_\e(x) &  p_\e(x)(\pa_x  + \im \mu) \end{bmatrix}: Y \to X \, ,
\end{equation}
having used also  \eqref{|D+mu|} and setting 
 $g(D) := \sgn(D) + \Pi_0$. For any $\lambda \in \Gamma$, 
the operator  ${\sL}_{0,0}-\lambda$ is invertible and its inverse is the Fourier multiplier 
matrix operator  
$$
({\sL}_{0,0}-\lambda)^{-1} = \text{Op}\left( \frac{1}{(\im k-\lambda )^2 + |k|}
\begin{bmatrix} \im k - \lambda & -|k| \\ 1 & \im k - \lambda  \end{bmatrix} \right): X \to Y \, .
$$
Hence, for $|\e|<\e_0$ and  $|\mu|<\mu_0$ small enough, uniformly on the compact set $\Gamma$, the operator $({\sL}_{0,0}-\lambda)^{-1}{\cal R}_{\mu,\e}:Y\to Y$ is bounded,  with small operatorial norm. Then ${\sL}_{\mu,\e}-\lambda$ is invertible by Neumann series and
\begin{align}\label{L-lambda}
({\sL}_{\mu,\e}-\lambda)^{-1}  = \big( \uno + ({\sL}_{0,0}-\lambda )^{-1}
{\cal R}_{\mu,\e} \big)^{-1} ({\sL}_{0,0}-\lambda)^{-1}: X\to Y \, .
\end{align}
This proves that $\Gamma$ belongs to the resolvent set of ${\sL}_{\mu,\e}$.\\
2. By the previous point the operator   
$ P_{\mu,\e} $  is  well defined and  bounded  $ X \to Y $. 
It  clearly commutes with ${\sL}_{\mu,\e}$.
 The projection property $P_{\mu,\e}^2= P_{\mu,\e}$ is a classical result based on complex integration, see \cite{Kato}, and we omit it. 
The map $(\mu, \e) \to ({\sL}_{0,0}- \lambda)^{-1} {\cal R}_{\mu,\e} \in \cL(Y)$ is analytic.
 Since the map $T \mapsto (\text{Id} + T)^{-1}$ is analytic in $\cL(Y)$ (for $\| T \|_{\cL(Y)} < 1$) the operators 
$({\sL}_{\mu,\e}-\lambda)^{-1} $ 
in \eqref{L-lambda} 
 and $P_{\mu,\e}$ in $ \cL (X,Y) $
are analytic as well with  
 respect to  $(\mu,\e)$. \\
 3. The decomposition
 \eqref{Y.dec} 
 is a consequence of  $P_{\mu,\e}$ being a continuous projector in $\cL(Y)$.  The invariance  of the subspaces follows since $P_{\mu,\e}$ and ${\sL}_{\mu,\e}$ commute.
To prove  \eqref{dec.spectrum} define 
for an arbitrary $\lambda_0 \not \in \Gamma$  the operator
$$
R_{\mu,\e}(\lambda_0) := - \frac{1}{2\pi \im} \oint_\Gamma \frac{1}{\lambda - \lambda_0} \left( \cL_{\mu,\e} - \lambda \right)^{-1} \, \de \lambda  \ \colon \ X \to Y  \ . 
$$
If $\lambda_0$ is outside $\Gamma$, one has 
$R_{\mu,\e}(\lambda_0) ( {\sL}_{\mu,\e} - \lambda_0) = ({\sL}_{\mu,\e}- \lambda_0)R_{\mu,\e}(\lambda_0) = P_{\mu,\e}$ and thus $\lambda_0 \not\in \sigma({\sL}_{\mu,\e}\vert_{\cV_{\mu,\e}})$.
For  $\lambda_0$  inside $\Gamma$, 
$R_{\mu,\e}(\lambda_0) ( {\sL}_{\mu,\e} - \lambda_0) = ({\sL}_{\mu,\e}- \lambda_0)R_{\mu,\e}(\lambda_0) =  P_{\mu,\e}- \text{Id}$ and thus $\lambda_0 \not\in \sigma({\sL}_{\mu,\e}\vert_{Ker(P_{\mu,\e})})$. Then  \eqref{dec.spectrum} follows.
\\
4. By \eqref{Pproj}, the resolvent identity
 $ A^{-1} - B^{-1} = A^{-1} (B-A) B^{-1} $ and \eqref{restmue}, we write
 $$ 
 P_{\mu,\e} - P_{0,0} = 
 \frac{1}{2\pi\im}\oint_\Gamma ({\sL}_{\mu,\e}-\lambda)^{-1} {\mathcal R}_{\mu,\e} 
 ({\sL}_{0,0}-\lambda)^{-1} \de \lambda \, .
 $$
Then $ \| P_{\mu,\e} - P_{0,0} \|_{{\cL}(Y)}<1 $ for $ |\e| < \e_0 $, 
 $ |\mu| < \mu_0 $ small enough and the operators $ U_{\mu,\e } $ in \eqref{OperatorU}
are well defined in  $ \cL(Y)$ (actually $ U_{\mu,\e }  $ are also in $ \cL(X)$).  
The invertibility of $U_{\mu,\e}$ and  formula 
\eqref{U.conj} are  proved in \cite{Kato}, Chapter I, Section 4.6, for the pairs of projectors
$ Q = P_{\mu,\e} $ and $ P =  P_{0,0} $.
The analyticity of $(\mu,\e) \mapsto U_{\mu,\e}\in \cL(Y)$
follows by the analyticity  $(\mu,\e) \mapsto P_{\mu,\e} \in \cL(Y)$ and of the map $T \mapsto (\text{Id} - T)^{-\frac12}$ in $\cL(Y)$ for $\|T\|_{\cL(Y)} < 1$.   \\
5. It  follows from the conjugation formula \eqref{U.conj}.
\end{proof}

The Hamiltonian and reversible nature of
the operator $ {\sL}_{\mu,\e} $, see \eqref{calL.ham} and \eqref{calL.rev}, imply
additional   algebraic properties 
for   spectral projectors $P_{\mu,\e}$ and the transformation operators $U_{\mu,\e} $.

\begin{lem}\label{propPU}
 For any $(\mu, \epsilon)  \in B(\mu_0)\times  B(\epsilon_0)$, the following holds true:
\begin{itemize}
\item[(i)]
The projectors $P_{\mu,\e}$ defined in \eqref{Pproj} are  (complex) skew-Hamiltonian, 
namely $ \cJ P_{\mu,\e} $ are skew-Hermitian 
\begin{equation}\label{Pskew}
\cJ P_{\mu,\e}=P_{\mu,\e}^*\cJ \, , 
\end{equation}
%The projectors $P_{\mu,\e}$ defined in \eqref{Pproj} are skew-Hamiltonian, namely $\cJ P_{\mu,\e}=P_{\mu,\e}^*\cJ $,  
and reversibility preserving, i.e. 
$ \bro P_{\mu,\e} = P_{\mu,\e}  \bro $.
\item[(ii)]  The transformation operators 
$ U_{\mu,\e} $ in \eqref{OperatorU} are symplectic, namely 
$$ 
U_{\mu,\e}^* \cJ U_{\mu,\e}= \cJ \, , 
$$ 
and reversibility preserving.
 \item[(iii)]  $P_{0,\e}$ and $U_{0,\e}$ are real operators, i.e. $\bar{P_{0,\e}}=P_{0,\e}$ and $\bar{U_{0,\e}}=U_{0,\e}$.
\end{itemize}
\end{lem}
\begin{rmk}
The term  (complex) skew-Hamiltonian is used in  \cite[Section 6]{FMMX} for matrices.
\end{rmk}

\begin{proof}
Let  $\gamma\colon [0,1] \to \bC$ be a counter-clockwise oriented  parametrization of   $\Gamma$. \\
$(i)$  Since ${\sL}_{\mu,\e}$ is Hamiltonian, it results  $ {\sL}_{\mu,\e} \cJ = - \cJ {\sL}_{\mu,\e}^* $ on $Y$. 
Then, for any scalar $ \lambda $ in the resolvent set of $ {\sL}_{\mu,\e} $, the number
$ - \lambda $ belongs to the resolvent of ${\sL}_{\mu,\e}^* $ and 
\begin{equation}\label{dallaHam}
\cJ ({\sL}_{\mu,\e} -\lambda)^{-1} = - ({\sL}_{\mu,\e}^*+\lambda)^{-1} \cJ \, . 
\end{equation}
Taking the adjoint of \eqref{Pproj}, we have 
\begin{equation}\label{propmue}
 P_{\mu,\e}^*  = \frac{1}{2\pi\im} \int_0^1 \left(\cL_{\mu,\e}^* -\bar\gamma(t)\right)^{-1} \dot{\bar{\gamma}}(t)\de t  = \frac{1}{2\pi\im} \oint_\Gamma \left(\cL_{\mu,\e}^* +\lambda\right)^{-1} \de \lambda \, ,
\end{equation}
because the path $-\bar\gamma (t) $ winds around  the origin clockwise. 
We conclude that 
$$
 {\cJ}P_{\mu,\e}\stackrel{\eqref{Pproj}}{=} -\frac{1}{2\pi \im}  \oint_\Gamma \cJ \left({\sL}_{\mu,\e} -\lambda\right)^{-1} \de \lambda 
 \stackrel{\eqref{dallaHam}}{=} \frac{1}{2\pi \im}  \oint_\Gamma \left({{\sL}}_{\mu,\e}^* +\lambda\right)^{-1}\cJ \de \lambda\ 
 \stackrel{\eqref{propmue}}{=} P_{\mu,\e}^* \cJ \, .
$$
Let us now prove that $P_{\mu,\e}$ is reversibility preserving.
By \eqref{calL.rev} one has 
$ ({\sL}_{\mu,\e} - \lambda) \bar\rho = \bar\rho  ( - {\sL}_{\mu,\e}  - \bar \lambda )$ 
and, for any scalar $ \lambda $ in the resolvent set of $  {\sL}_{\mu,\e}  $, we have 
$ \bar\rho ({\sL}_{\mu,\e} - \lambda)^{-1}  =  - ( {\sL}_{\mu,\e}  + \bar \lambda )^{-1} \bar\rho  $,  using also that $ (\bar\rho)^{-1} = \bar\rho $.
Thus, recalling \eqref{Pproj} and \eqref{reversibilityappears}, we have 
$$
\bar \rho P_{\mu,\e} = \frac{1}{2\pi \im} \int_0^1 -\left({\sL}_{\mu,\e} +\bar \gamma(t)\right)^{-1}\dot{\bar\gamma}(t) \de t \, \bar \rho = -\frac{1}{2\pi \im} \oint_\Gamma ({\sL}_{\mu,\e} - \lambda)^{-1}\de \lambda \, \bar \rho = P_{\mu,\e} \bar \rho \, ,
$$
because the path $-\bar\gamma (t) $ winds around  the origin clockwise. \\
$(ii)$ 
If an operator $A$ is skew-Hamiltonian then $A^k$, $k\in\bN$, is skew-Hamiltonian as well.  As a consequence, being the projectors $P_{\mu,\e}$, $P_{0,0}$ and their difference skew-Hamiltonian, 
the operator  $\big( \uno-(P_{\mu,\e}-P_{0,0})^2 \big)^{-1/2}$ defined as in \eqref{rootexp} is skew Hamiltonian as well. Hence, by \eqref{OperatorU}  we get 
\begin{align*}
\cJ U_{\mu, \e} 
& =\left[ \big( \uno-(P_{\mu,\e}-P_{0,0})^2 \big)^{-1/2} \right]^* \ 
 \big[ 
P_{0,0}P_{\mu,\e} +(\uno-P_{0,0})  (\uno - P_{\mu,\e})\big]^* \ \cJ  
\ \ \stackrel{\eqref{Uinv}}{=}\ \  U_{\mu,\e}^{-*} \cJ  
\end{align*}
and therefore $ U_{\mu,\e}^{*} \cJ U_{\mu, \e}  = \cJ $. 
Finally the operator $U_{\mu,\e}$ defined in  \eqref{OperatorU} 
is reversibility-preserving just as $\bro$ commutes with  $P_{\mu,\e}$ and $P_{0,0}$. 
\\
$(iii)$  By \eqref{Pproj} and since ${\sL}_{0,\e}$ is a real operator, we have
$$
\bar{P_{0,\e}} 
= \frac{1}{2\pi \im} \int_0^1 \left({\sL}_{0,\e} -\bar \gamma(t)\right)^{-1}\dot{\bar\gamma}(t) \de t = 
- \frac{1}{2\pi \im} \oint_\Gamma \left({\sL}_{0,\e} - \lambda  \right)^{-1}
 \de \lambda  = P_{0,\e} 
$$
because the path $\bar\gamma (t) $ winds around  the origin clockwise, proving 
that the operator $P_{0,\e}$ is real.   Then the operator $U_{0,\e}$ defined in \eqref{OperatorU} is real
as well.
\end{proof}

By the previous lemma, the linear involution $\bro $ 
commutes with the spectral projectors $P_{\mu,\e}$ and then 
$\bro $ leaves invariant the subspaces 
$ \mathcal{V}_{\mu,\e} = \text{Rg}(P_{\mu,\e}) $.  

\smallskip

Let us discuss  the implications of the previous lemma in the setting of
complex symplectic structures, presented for example in \cite{Arn,EM}. 
The infinite dimensional complex space $ L^2 (\bT, \bC^2) $, with scalar product \eqref{scalar}, is equipped with the {\it complex symplectic form}  
\begin{equation}\label{ses}
{\cal W}_c  \, \colon L^2 (\bT, \bC^2) \times L^2 (\bT, \bC^2) \to \bC \, , \quad  {\cal W}_c(f, g) := (\cJ f,g) \, ,
\end{equation}
which is sesquilinear,  skew-Hermitian and non-degenerate, 
cfr. Definition 1 in \cite{EM}. 
The skew-Hamiltonian property \eqref{Pskew} of the projector $ P_{\mu,\e} $ implies the following lemma. 

\begin{lem}\label{lem:simpl}
For any $ (\mu,\e)$, the  linear subspace $ \cV_{\mu,\e} = \text{Rg}(P_{\mu,\e}) $ 
is a  \text{complex symplectic}  subspace of $L^2 (\bT, \bC^2) $, 
namely  the 
 symplectic form $ {\cal W}_c$ in \eqref{ses}, restricted to 
 $ \cV_{\mu,\e} $, is non-degenerate. 
 \end{lem}
 
 \begin{proof}
Let 
 $ \tilde f \in \cV_{\mu,\e} $, thus $\tilde f = P_{\mu,\e} \tilde f$.
 Suppose that $ {\cal W}_c(\tilde f, \tilde g )  = 0 $
for any $ \tilde g = P_{\mu,\e} g \in \cV_{\mu,\e}  $,  
 $ g \in L^2 (\bT, \bC^2) $. Thus
$$
0 =  {\cal W}_c(\tilde f , \tilde g ) = (\cJ \tilde f , P_{\mu,\e} g )  = 
(P_{\mu,\e}^* \cJ \tilde f ,  g ) \stackrel{\eqref{Pskew} } =   ( \cJ P_{\mu,\e}\tilde  f ,  g ) =
( \cJ \tilde f ,  g )  \, . 
 $$
We deduce that $ \cJ \tilde f = 0 $ and then  $ \tilde f = 0 $.   
\end{proof}

\begin{rmk}\label{rem:U.W}
In view of Lemma \ref{propPU}-($ii$) 
the transformation operator $ U_{\mu,\e} $ is symplectic, namely preserves 
the symplectic form \eqref{ses}, i.e. $ {\cal W}_c(U_{\mu,\e} f,  U_{\mu,\e} g ) =  {\cal W}_c( f,  g )  $, for any 
$ f, g \in L^2 (\bT, \bC^2) $.
\end{rmk}

\noindent
{\bf Symplectic and reversible basis of $\mathcal{V}_{\mu,\e}$.}
%The space $L^2(\bT, \bC^2)$ is a complex symplectic space endowed with the complex symplectic form (according to Definition 1.1 of \cite{EM})
%$\omega(f, g) := ( \cJ f , g ) , \quad \forall f, g \in L^2(\bT, \bC^2)$.
It is convenient to represent the Hamiltonian and reversible operator
$ {\sL}_{\mu,\e} : \mathcal{V}_{\mu,\e} \to \mathcal{V}_{\mu,\e} $ in a 
basis which is symplectic  and reversible, according to the following definition. 
\begin{sia}\label{def:SR} 
{\bf (Symplectic and reversible basis)}
 A basis $\mathtt{F}:=\{\mathtt{f}^+_1,\,\mathtt{f}^-_1,\,\mathtt{f}^+_0,\,\mathtt{f}^-_0 \}$ of $\mathcal{V}_{\mu,\e}$ is 
 \begin{itemize}
 \item 
 \emph{symplectic} if, for any $ k, k' = 0,1 $,  
 \begin{equation}\label{symplecticbasis}
 \molt{\cJ \tf_k^-}{\tf_k^+} = 1 \, ,  \ \  
 \big( \cJ \tf_k^\sigma, \tf_k^\sigma \big) = 0 \, , \  \forall \sigma = \pm \, ;  
  \ \  \text{if} \ k \neq k' \  \text{then} \  
\big( \cJ \tf_k^\sigma, \tf_{k'}^{\sigma'} \big) = 
   0 \, , \ \forall \sigma, \sigma' = \pm  \, . 
\end{equation}
\item 
\emph{reversible} if
\begin{equation}
\label{reversiblebasis}
 \bar \rho \tf^+_1 =  \tf^+_1 , \quad   \bar \rho \tf^-_1 = - \tf^-_1 , \quad 
 \bar \rho \tf^+_0 =  \tf^+_0 , \quad   \bar \rho \tf^-_0 = - \tf^-_0, \quad \text{i.e. } \bro \tf_k^\sigma = \sigma \tf_k^\sigma \, , \ \forall \sigma = \pm, k = 0,1 \, .
\end{equation}
\end{itemize}
\end{sia}

%\begin{rmk}\todo{added}
%In \cite{EM} a canonical basis is defined in (1.13).  
%\end{rmk}

\begin{rmk}\label{rem:U2}
By Remark \ref{rem:U.W}, the operator $U_{\mu,\e}$ maps a symplectic basis in a symplectic basis. %  (cfr. Definition \ref{def:SR}).
\end{rmk}

In the next lemma we outline a property of a reversible basis. We use the following 
notation along the paper:  we denote by $even(x)$ a real $2\pi$-periodic function which is even in $x$, and by $odd(x)$ a real $2\pi$-periodic function which is odd in $x$.
\begin{lem}\label{lem:f.parity}
The real and imaginary parts of the elements of a reversible basis $\mathtt{F}=\{\mathtt{f}^\pm_k \}$, $k=0,1$, enjoy the following parity properties
\begin{equation}\label{reversiblebasisprop}
  \tf_k^+(x) = \vet{even(x)+\im odd(x)}{odd(x)+\im even(x)}, \quad
  \tf_k^-(x) = \vet{odd(x)+\im even(x)}{even(x)+\im odd(x)}.
\end{equation}
\end{lem}
\begin{proof} By the definition of the involution $\bro$ in \eqref{reversibilityappears}, we get 
$$
  \tf_k^+(x) = \vet{a(x)+\im b(x)}{c(x)+\im d(x)} = \bro \tf_k^+(x) = \vet{a(-x)-\im b(-x)}{-c(-x)+\im d(-x)} \implies a, d \text{ even}, \ b,c \text{ odd} \, .
$$
The properties of $\tf_k^-$ follow similarly.
\end{proof}
We now expand  a vector of $ \mathcal{V}_{\mu,\e} $ along a  symplectic basis.

 \begin{lem} \label{lem:simplbas}
 Let $\mathtt{F}= \{ \tf_{1}^+, \tf_{1}^-, \tf_{0}^+, \tf_{0 }^- \} $  be a  
 symplectic basis of  $ \mathcal{V}_{\mu,\e} $.
Then any $\tf $ in $ \mathcal{V}_{\mu,\e}$ has the expansion 
\begin{equation}\label{espansionegenerica}
\tf = - \molt{ \cJ  \tf}{\tf_{1}^-} \tf_{1}^+  
 + \molt{\cJ  \tf}{\tf_{1}^+} \tf_{1}^-
- 
 \molt{\cJ  \tf}{\tf_{0}^-} \tf_{0}^+  
 + \molt{\cJ  \tf}{\tf_{0}^+} \tf_{0}^- \, .  
\end{equation}
\end{lem}
\begin{proof}
We decompose 
$ \tf = \alpha_{1}^+ \tf_{1}^+  
 + \alpha_{1}^- \tf_{1}^-
+\alpha_{0}^+ \tf_{0}^+  
 + \alpha_{0}^- \tf_{0}^- $ for suitable coefficients $ \alpha_k^\sigma \in \bC $. 
By applying $\cJ$, taking the $L^2$ scalar products with the 
vectors $ \{ \tf_k^\sigma \}_{\sigma = \pm, k=0,1}$, 
using \eqref{symplecticbasis} 
 and noting that $ \molt{\cJ \tf_k^+}{\tf_k^-} = - 1  $,  we get 
the expression of the coefficients $\alpha_k^\sigma$ as in \eqref{espansionegenerica}.
\end{proof}

We now represent ${\sL}_{\mu,\e} :\mathcal{V}_{\mu,\e}\to\mathcal{V}_{\mu,\e}  $ with respect to a symplectic and reversible basis.

\begin{lem}\label{lem:B.mat}
The $ 4 \times 4 $ matrix that represents the Hamiltonian and reversible operator 
${\sL}_{\mu,\e}= \cJ {\cal B}_{\mu,\e}:\mathcal{V}_{\mu,\e}\to\mathcal{V}_{\mu,\e} $ with respect to a symplectic and reversible basis $\mathtt{F}=\{\tf_1^+,\tf_1^-,\tf_0^+,\tf_0^-\} $ of $\mathcal{V}_{\mu,\e}$ is 
\begin{align}\label{Lform}
 \tJ_4 \tB_{\mu,\e} \, ,\quad  \tJ_4 := 
 \begin{pmatrix} 
 \tJ_2& \vline & 0 \\
 \hline
0  & \vline & \tJ_2
\end{pmatrix}, \quad \tJ_2 := \begin{pmatrix} 
 0 & 1 \\
-1  & 0
\end{pmatrix}, \quad \text{where } \quad \tB_{\mu,\e}= \tB_{\mu,\e}^* \end{align}
is the self-adjoint matrix
\begin{equation}\label{matrix22} 
\tB_{\mu,\e} = 
\begin{pmatrix}
\BVe{+}{1}{+}{1} & \BVe{-}{1}{+}{1} & \BVe{+}{0}{+}{1} & \BVe{-}{0}{+}{1} \\
\BVe{+}{1}{-}{1} & \BVe{-}{1}{-}{1} & \BVe{+}{0}{-}{1} & \BVe{-}{0}{-}{1} \\
\BVe{+}{1}{+}{0} & \BVe{-}{1}{+}{0} & \BVe{+}{0}{+}{0} & \BVe{-}{0}{+}{0} \\
\BVe{+}{1}{-}{0} & \BVe{-}{1}{-}{0} & \BVe{+}{0}{-}{0} & \BVe{-}{0}{-}{0} \\
	\end{pmatrix}.
\end{equation}
 The entries of the matrix $\tB_{\mu,\e}$ are alternatively  real or purely imaginary: for any $ \sigma = \pm $, $ k = 0, 1 $, 
\begin{equation}\label{revprop}
 \molt{{\cal B}_{\mu,\e}  \, \tf^{\sigma}_{k}}{\tf^{\sigma}_{k'}} \text{ is real},\qquad 
   \molt{{\cal B}_{\mu,\e}  \, \tf^{\sigma}_{k}}{\tf^{-\sigma}_{k'}} \text{ is purely imaginary} \, .
\end{equation}
\end{lem}

\begin{proof}
Lemma \ref{lem:simplbas} implies  that
$$
{\sL}_{\mu,\e} \tf^{\sigma}_{k} = -\sum_{\substack{k'=0,1, \sigma'=\pm}}\sigma' \big(\cJ {\sL}_{\mu,\e} \tf^{\sigma}_{k }, \tf_{k'}^{-\sigma'} \big) 
\tf^{\sigma'}_{k' } = \sum_{\substack{k'=0,1, \sigma'=\pm}}\sigma'
\big( {\cal B}_{\mu,\e} \tf^{\sigma}_{k }, \tf_{k'}^{-\sigma'} \big) \tf^{\sigma'}_{k'} \, .
$$
Then the matrix representing  the operator ${\sL}_{\mu,\e} :\mathcal{V}_{\mu,\e}\to\mathcal{V}_{\mu,\e} $ 
with respect to the basis $\mathtt{F}$ is 
given by $\tJ_4 \tB_{\mu,\e}$ with $\tB_{\mu,\e}$  in 
\eqref{matrix22}. 
The matrix $\tB_{\mu,\e}$ is selfadjoint because ${\cal B}_{\mu,\e}$ is a selfadjoint operator.
We now prove \eqref{revprop}. By  recalling \eqref{reversibilityappears} and \eqref{scalar} it results 
\begin{equation}\label{broprod}
 \molt{f}{g}= \bar{\molt{\bro f}{\bro g}} \, .
 \end{equation} 
 Then, by \eqref{broprod}, since ${\cal B}_{\mu,\e}$ is reversibility-preserving and  \eqref{reversiblebasis}, we get 
$$
\big( {\cal B}_{\mu,\e}  \, \tf^{\sigma}_{k}, \tf^{\sigma'}_{k' } \big) = \bar{\molt{\bar{\rho}{\cal B}_{\mu,\e}  \, \tf^{\sigma}_{k }}{\bar{\rho}\tf^{\sigma'}_{k'}} } =
\bar{\molt{{\cal B}_{\mu,\e} \bar{\rho} \, \tf^{\sigma}_{k}}{\bar{\rho}\tf^{\sigma'}_{k'}}}
 =
 \sigma \sigma' \,  
 \bar{\molt{{\cal B}_{\mu,\e} \, \tf^{\sigma}_{k}}{\tf^{\sigma'}_{k'}}} \, ,
$$
which proves \eqref{revprop}.
\end{proof}

\begin{rmk}
The complex symplectic form ${\cal W}_c$ in
\eqref{ses} 
 restricted to the symplectic subspace $\cV_{\mu,\e}$ is represented, in {\em any} symplectic
  basis (cfr. \eqref{symplecticbasis}), by the matrix $\tJ_4$ in \eqref{Lform}, acting in $\bC^4$ with the standard  complex scalar product.
\end{rmk}

\noindent 
{\bf Hamiltonian and reversible matrices.} 
 It is convenient to give a name to the matrices of the form obtained 
in Lemma \ref{lem:B.mat}.

\begin{sia}
A $ 2n \times 2n $, $ n = 1,2, $ matrix of the form 
$\tL=\tJ_{2n} \tB$ is 
\begin{enumerate}
\item \emph{Hamiltonian} if  $ \tB $ is a self-adjoint matrix, i.e.   $\tB=\tB^*$;
\item \emph{Reversible} if $\tB$ is reversibility-preserving,   i.e. $\rho_{2n}\circ \tB = \tB \circ \rho_{2n} $, where 
\begin{equation}\label{involutionrep}
 \rho_4 := \begin{pmatrix}\rho_2 & 0 \\ 0 & \rho_2\end{pmatrix}, \qquad \rho_2 := \begin{pmatrix} \mathfrak{c}  & 0 \\ 0 & - \mathfrak{c} \end{pmatrix},
\end{equation}
and $\Gc : z \mapsto \bar z $ is the conjugation of the complex plane.
Equivalently, $\rho_{2n} \circ \tL  = -  \tL \circ \rho_{2n}$.
\end{enumerate}
\end{sia}
In the sequel we shall mainly deal with $ 4 \times 4 $ Hamiltonian and reversible matrices.
The transformations 
preserving  the Hamiltonian structure  are called
  \emph{symplectic}, and  satisfy
\begin{align}\label{sympmatrix}
 Y^* \tJ_4 Y = \tJ_4 \, .
 \end{align}
 If $Y$ is symplectic then $Y^*$ and $Y^{-1}$  are symplectic as well. A Hamiltonian matrix $\tL=\tJ_4 \tB$, with $\tB=\tB^*$, is conjugated through $Y$ in the new Hamiltonian matrix
 \begin{equation}
 \tL_1 = Y^{-1}  \tL Y = Y^{-1}  \tJ_4 Y^{-*} Y^*  \tB Y = \tJ_4 \tB_1 \quad \text{where } \quad 
  \tB_1 := Y^* \tB Y = \tB_1^* \, . \label{sympchange}
 \end{equation}
Note that the matrix $ \rho_4 $ in \eqref{involutionrep} represents the action of the 
involution $\bar \rho : {\mathcal V}_{\mu,\e}
\to {\mathcal V}_{\mu,\e}  $ defined in \eqref{reversibilityappears} in a 
reversible basis (cfr. \eqref{reversiblebasis}).
A $ 4\times 4$ matrix $\tB=(\tB_{ij})_{i,j=1,\dots,4}$ is reversibility-preserving if and only if its entries are alternatively real and purely imaginary, namely $\tB_{ij}$ is real when $i+j$ is even and purely imaginary otherwise, as in \eqref{revprop}.
A $4\times 4$ complex matrix $\tL =(\tL_{ij})_{i,j=1, \ldots, 4}$ is reversible if and only if  $\tL_{ij}$ is purely imaginary when $i+j$ is even and real otherwise.
 
In the sequel we shall use that the flow of a Hamiltonian 
reversibility-preserving matrix  is symplectic
and reversibility-preserving. 
\begin{lem}\label{lem:S.conj}
 Let $\Sigma$ be a self-adjoint and reversible matrix, then
 $\exp(\tau \tJ_4 \Sigma)$,
 $ \tau \in \mathbb{R} $,  is a reversibility-preserving symplectic matrix. 
\end{lem}
\begin{proof}
 The flow $\varphi(\tau) := \exp(\tau \tJ_4 \Sigma)$ 
 solves $ \frac{d}{d\tau}  \varphi(\tau) := \tJ_4 \Sigma \varphi(\tau)  $, 
 with 
 $\varphi(0) = \mathrm{Id}$.
 Then  $ \psi(\tau) := \varphi(\tau)^* \tJ_4 \varphi(\tau) -\tJ_4$ satisfies  
 $\psi(0)=0$ and 
 $\frac{d}{d\tau}  \psi (\tau)= \varphi(\tau)^*\tJ_4^*\tJ_4 \varphi(\tau)+ 
 \varphi(\tau)^* \tJ_4\tJ_4 \varphi(\tau) = 0 \, .
 $
Then  $ \psi (\tau) = 0 $ for any $ \tau $ and  
  $\varphi(\tau) $ is symplectic.
  %  in particular $\varphi(1)=\exp(\tJ_4 \Sigma) $. 
The matrix $\exp(\tau \tJ_4 \Sigma) = \sum_{n \geq 0} \frac{1}{n!} ( \tau \tJ_4 \Sigma)^n $  
is reversibility-preserving
  since each $(\tJ_4 \Sigma)^n$, $ n \geq 0 $, is reversibility-preserving.
\end{proof}

\section{Matrix representation of $ {\sL}_{\mu,\e}$ on $ \mathcal{V}_{\mu,\e}$}\label{sec:mr}

In this section we use the transformation operators $U_{\mu,\e}$ obtained  in the previous section to construct  a symplectic and reversible  basis of $\cV_{\mu,\e}$ and,
in Proposition \ref{BexpG}, we compute the $4\times 4$ Hamiltonian and reversible matrix representing  $ {\sL}_{\mu,\e}\colon \cV_{\mu,\e} \to \cV_{\mu,\e}$ on such basis. 
\\[1mm]{\bf First basis of $\mathcal{V}_{\mu,\e}$.}
In view of Lemma \ref{lem:Kato1}, 
the first basis of $\mathcal{V}_{\mu,\e}$ that we consider is 
\begin{equation}\label{basisF}
{\cal F} := \big\{ 
f_{1}^+(\mu,\e), \  f_{1}^- (\mu,\e), \   f_{0}^+(\mu,\e),\   f_{0}^-(\mu,\e) \big\} \, , \quad 
f_{k}^\sigma(\mu,\e) := U_{\mu,\e} f_{k}^\sigma \, , \ \sigma=\pm \, , \,k=0,1 \, , 
\end{equation}
obtained applying the transformation operators  $ U_{\mu,\e} $ in \eqref{OperatorU} 
to the vectors
\begin{equation}
\label{funperturbed}
f_1^+ = \vet{\cos (x)}{\sin (x)} , \quad
f_1^- = \vet{- \sin (x)}{\cos (x)} , \quad 
f_0^+ = \vet{1}{0}, \quad
f_0^- = \vet{0}{1} \, , 
\end{equation}
which form a basis of $ \mathcal{V}_{0,0} =\mathrm{Rg} (P_{0,0}) $, cfr. \eqref{basestart}-\eqref{basestartadd}. 
Note that the  real valued vectors $ \{ f_1^\pm, f_0^\pm \} $  are orthonormal
with respect to the scalar product  \eqref{scalar}, and 
satisfy 
\begin{align}\label{azioneJ}
 \cJ f^+_1 = - f^-_1 , \qquad  \cJ f^-_1 =  f^+_1 , \qquad 
  \cJ f^+_0 = - f^-_0 , \qquad  \cJ f^-_0 =  f^+_0 \, , 
\end{align}
thus  forming a  symplectic and reversible basis for 
$ \mathcal{V}_{0,0} $, according to Definition \ref{def:SR}.

In view of Remarks \ref{rem:U.W} and \ref{rem:U2}, the symplectic operators $ U_{\mu,\e} $ transform, for any $ (\mu, \e ) $ small,  the symplectic basis \eqref{funperturbed} 
of  $ \mathcal{V}_{0,0} $,  into the symplectic basis \eqref{basisF}:

\begin{lem}\label{base1symp}
The  basis $  {\cal F}  $ of $\mathcal{V}_{\mu,\e}$ defined in \eqref{basisF}, 
is symplectic and reversible, i.e. satisfies \eqref{symplecticbasis} 
and \eqref{reversiblebasis}. 
Each map $(\mu, \e) \mapsto f^\sigma_k(\mu, \e)$ is analytic as a map $B(\mu_0)\times B(\epsilon_0) \to H^1(\bT)$. 
\end{lem}
\begin{proof}
Since by Lemma \ref{propPU}-($ii$) the maps $ U_{\mu,\e} $ are symplectic  and reversibility-preserving
the transformed vectors $ f_{1}^+(\mu,\e),\dots,f_{0}^-(\mu,\e) $ are symplectic orthogonals and reversible
as well as the unperturbed ones $ f_1^+,\dots, f_0^- $.
The analyticity of $f^\sigma_k(\mu, \e)$ follows from the analyticity property 
of   $U_{\mu, \e}$ proved in Lemma \ref{lem:Kato1}.
\end{proof}

In the next lemma we  provide a suitable  expansion of the vectors  $ f_k^\sigma(\mu,\e) $ in $ (\mu, \e ) $. 
 We denote by $even_0(x)$ a real, even, $2\pi$-periodic function with zero space average. 
In the sequel 
 $\cO(\mu^{m} \e^{n}) \footnotesize\vet{even(x)}{odd(x)}$ denotes an analytic map in $(\mu, \e)$ with values in  $ H^1(\bT, \bC^2) $,
 whose first component is $even(x)$ and the second one  $odd(x)$;  similar meaning 
for $\cO(\mu^{m} \e^{n}) \footnotesize\vet{odd(x)}{even(x)}$, 
etc...

\begin{lem}\label{expansion1} 
{\bf (Expansion of the basis $ {\cal F}$)} For small values of $(\mu, \e)$ 
the basis $ {\cal F}$ in \eqref{basisF}  has the following expansion
\begin{align}
 f^+_1(\mu, \e) & = \vet{\cos(x)}{\sin(x)} + \im \frac{\mu}{4} \vet{\sin(x)}{\cos(x)} + 
 \epsilon \vet{2 \cos(2x)}{\sin(2x)} \label{exf41}
 \\ &  +
\cO(\mu^2) \vet{even_0(x) + \im odd(x)}{odd(x) + \im even_0(x)}  + \cO(\e^2) \vet{even_0(x)}{odd(x)} + \im\mu \epsilon\vet{odd(x)}{even(x)} + \cO(\mu^2\e,\mu\e^2) \, , \notag \\
 f^-_1(\mu, \e) \label{exf42} &= \vet{-\sin(x)}{\cos(x)} + \im \frac{\mu}{4} \vet{\cos(x)}{-\sin(x)} + \epsilon \vet{-2 \sin(2x)}{\cos(2x)}\\ & + \cO(\mu^2) \vet{odd(x) + \im even_0(x)}{even_0(x) + \im odd(x)} + \cO(\e^2) \vet{odd(x)}{even(x)} 
+ \im\mu \epsilon\vet{even(x)}{odd(x)} + \cO(\mu^2\e,\mu\e^2) \, , \notag \\
 f^+_0(\mu, \e) \label{exf43} &= \vet{1}{0}+ \epsilon \vet{ \cos(x)}{-\sin(x)}
 + \cO(\e^2) \vet{even_0(x)}{odd(x)} + \im\mu \epsilon\vet{odd(x)}{even_0(x)}+ \cO(\mu^2\e,\mu\e^2) \, , \\
 f^-_0(\mu, \e) \label{exf44} & = \vet{0}{1}  + \mu\e \left(\vet{\sin(x)}{\cos(x)}+  \im   \vet{even_0(x)}{odd(x)} \right)+\cO(\mu^2\e,\mu\e^2) \, ,
\end{align}
where 
the remainders $\cO()$ are vectors in $H^1(\bT)$.
 For $\mu=0$ the basis $\{f_k^\pm(0,\e), k=0,1 \}  $ is real  and 
\begin{equation}\label{nonzeroaverage}
f^{+}_1 (0, \e) =  \vet{even_0(x)}{odd(x)},
\ 
f^{-}_1 (0, \e) =  \vet{odd(x)}{even(x)}, 
\  
f^{+}_0 (0, \e) = \vet{1}{0}+ \vet{even_0(x)}{odd(x)} \, , 
\   
f^{-}_0 (0, \e) =  \vet{0}{1} \, .
\end{equation}
\end{lem}

\begin{proof}
The long calculations are given in Appendix \ref{ProofExpansion}. 
\end{proof}

\noindent 
{\bf Second basis of $\mathcal{V}_{\mu,\e}$.}
We now construct from the basis $ {\cal F } $ in \eqref{basisF} 
another symplectic and reversible
basis of $\mathcal{V}_{\mu,\e}$ with an additional property.
Note that  the second component of the vector $f_1^-(0,\e)$ is 
an even function whose space average is not 
 necessarily zero, cfr. \eqref{nonzeroaverage}. 
 Thus
we introduce the new symplectic and reversible basis of $\mathcal{V}_{\mu,\e}$
$$
\mathcal{G}:= \big\{ g_1^+(\mu,\e),\ g_1^-(\mu,\e),\ g_0^+(\mu,\e),\ g_0^-(\mu,\e) \big\} \, , 
$$
defined by
\begin{equation}\label{basisG}
\begin{aligned}
 & g^+_1(\mu,\e) :=   f^+_1(\mu,\e) \, , 
 \qquad 
 g^-_1(\mu,\e) :=  f^-_1(\mu,\e)-n(\mu,\e)f_0^-(\mu,\e) \, , \\ 
 & g^+_0(\mu,\e) := f^+_0(\mu,\e)+n(\mu,\e)f^+_1(\mu,\e) \, ,\qquad  g^-_0(\mu,\e) := f^-_0(\mu,\e) \, , 
 \end{aligned} 
 \end{equation}
with  
\begin{equation}\label{normaliz}
n(\mu,\e) :=  \displaystyle{\frac{\molt{f_1^-(\mu,\e)}{f_0^-(\mu,\e)}}{\|f_0^-(\mu,\e)\|^2}} \, . 
\end{equation}
Note that $n(\mu,\e)$ is real, because, in view of \eqref{broprod} and Lemma \ref{base1symp},
\begin{equation}\label{nuisreal}
n(\mu,\e) :=  \displaystyle{\frac{\bar{\molt{\bro f_1^-(\mu,\e)}{\bro f_0^-(\mu,\e)}}}{\|f_0^-(\mu,\e)\|^2}} = \displaystyle{\frac{\bar{\molt{f_1^-(\mu,\e)}{f_0^-(\mu,\e)}}}{\|f_0^-(\mu,\e)\|^2}} = \bar{n(\mu,\e)} \, .
\end{equation} 
This new basis has the property that  $g_1^-(0,\e)$ has zero average, see \eqref{zeroaverage}. We shall exploit this feature crucially in Lemma \ref{lem:B2}, see remark 
\ref{rem:newbasis}.
\begin{lem}
 The basis $\mathcal{G}$ in \eqref{basisG} is symplectic and reversible, i.e. it satisfies \eqref{symplecticbasis} and \eqref{reversiblebasis}.
Each map $(\mu, \e) \mapsto g^\sigma_k(\mu, \e)$ is analytic as a map $B(\mu_0)\times B(\epsilon_0) \to H^1(\bT, \bC^2)$.
\end{lem}
\begin{proof}
The vectors $g_k^\pm(\mu,\e) $, $k=0,1$ satisfy \eqref{symplecticbasis} and \eqref{reversiblebasis} because $f_k^\pm(\mu,\e) $, $k=0,1$ satisfy  the same properties  as well, and $n(\mu,\e)$ is real. 
The analyticity of $g^\sigma_k(\mu, \e)$ follows from the corresponding property of  the basis ${\cal F}$.
\end{proof}

We now state the main result of this section.

\begin{prop}\label{BexpG}
The matrix that represents the Hamiltonian and reversible operator  
$ {\sL}_{\mu,\e} : \mathcal{V}_{\mu,\e} \to \mathcal{V}_{\mu,\e} $ in the symplectic and reversible basis $\mathcal{G}$ of $\mathcal{V}_{\mu,\e}$ defined in \eqref{basisG}, 
is a Hamiltonian matrix 
 $\tL_{\mu,\e}=\tJ_4 \tB_{\mu,\e}$, where 
 $\tB_{\mu,\e} $ is a self-adjoint and reversibility preserving
(i.e. satisfying \eqref{revprop})
 $ 4 \times 4$  matrix of the form 
  \begin{equation}\label{splitB}
\tB_{\mu,\e}=
\begin{pmatrix} 
E & F \\ 
F^* & G 
\end{pmatrix}, \qquad 
E  = E^* \, , \   \ G = G^* \, , 
\end{equation} 
  where $E, F, G$  are the $ 2 \times 2 $  matrices 
\begin{align}\label{BinG1}
& E := 
\begin{pmatrix} 
  \e^2(1+r_1'(\e,\mu\e^2))-\frac{\mu^2}{8}(1+r_1''(\e,\mu))  & \im \big( \frac12\mu+ r_2(\mu\e^2,\mu^2\e,\mu^3) \big)  \\
- \im \big( \frac12\mu+ r_2(\mu\e^2,\mu^2\e,\mu^3) \big) & -\frac{\mu^2}{8}(1+r_5(\e,\mu))
 \end{pmatrix} \\
 & \label{BinG2} G := 
 \begin{pmatrix} 
1+ r_8(\e^3,\mu^2\e, \mu \e^2, \mu^3) &   - \im r_9(\mu\e^2,\mu^2\e,\mu^3) \\
  \im  r_9(\mu\e^2, \mu^2\e,\mu^3)  & \mu+ r_{10}(\mu^2\e,\mu^3) 
 \end{pmatrix} \\
 & \label{BinG3}
 F = 
 \begin{pmatrix} 
 r_3(\e^3,\mu\e^2,\mu^2\e,\mu^3) & \im  r_4({\mu\e}, \mu^3)  \\
  \im   r_6(\mu\e, \mu^3)    & r_7(\mu^2\e,\mu^3) 
 \end{pmatrix} \, . 
 \end{align}
\end{prop}

The rest of this section is devoted to the proof of 
Proposition \ref{BexpG}.  
The first step is to provide  the following 
expansion in $(\mu,\e)$ of the basis $\mathcal{G}$. 

\begin{lem}\label{espansioneg} {\bf (Expansion of the basis $\mathcal{G}$)}
For  small values of $(\mu,\e) $, the basis $ {\mathcal G}$ defined in \eqref{basisG} 
has the following expansion
\begin{align}
 g^+_1(\mu, \e) \label{exg41}& = \vet{\cos(x)}{\sin(x)} + \im \frac{\mu}{4} \vet{\sin(x)}{\cos(x)} + 
 \epsilon \vet{2 \cos(2x)}{\sin(2x)} \\
&  +
\cO(\mu^2) \vet{even_0(x) + \im odd(x)}{odd(x) + \im even_0(x)}  + \cO(\e^2) \vet{even_0(x)}{odd(x)} + \im \mu \epsilon\vet{odd(x)}{even(x)} + \cO(\mu^2\e,\mu\e^2) \, , \notag \\
 g^-_1(\mu, \e) \label{exg42} &= \vet{-\sin(x)}{\cos(x)} + \im \frac{\mu}{4} \vet{\cos(x)}{-\sin(x)} + \epsilon \vet{-2 \sin(2x)}{\cos(2x)} \\
& + \cO(\mu^2) \vet{odd(x) + \im even_0(x)}{even_0(x) + \im odd(x)} + \cO(\e^2) \vet{odd(x)}{even_0(x)}
+\im \mu \epsilon\vet{even(x)}{odd(x)}  + \cO(\mu^2\e,\mu\e^2) \, , \notag \\
 g^+_0(\mu, \e) &= \label{exg43} \vet{1}{0}+ \epsilon \vet{ \cos(x)}{-\sin(x)}
+ \cO(\e^2) \vet{even_0(x)}{odd(x)} + \im\mu \epsilon\vet{odd(x)}{even_0(x)}  + \cO(\mu^2\e,\mu\e^2) \, , \\
 g^-_0(\mu, \e) & = \label{exg44}\vet{0}{1}  + \mu\e \left(\vet{\sin(x)}{\cos(x)}+ \im   \vet{even_0(x)}{odd(x)} \right)+\cO(\mu^2\e,\mu\e^2) \, .
\end{align}
In particular, at $\mu=0$, the basis $ \{ g^\sigma_k (0,\e), \sigma= \pm, k =0 ,1 \} $ 
is real, 
\begin{equation}\label{muiszero}
g^{+}_1 (0, \e) =  \vet{even_0(x)}{odd(x)}, 
\, 
g^{-}_1 (0, \e) =  \vet{odd(x)}{even_0(x)},
\,
 g^{+}_0 (0, \e) = \vet{1}{0}+\vet{even_0(x)}{odd(x)}, 
 \, 
  g^{-}_0 (0, \e) = \vet{0}{1}, 
\end{equation}
and,  for any $\e$, 
\begin{equation}\label{zeroaverage}
\int_{\bT}g_1^-(0,\e) \, \de x  = 0 \, .
\end{equation}
\end{lem}

\begin{proof}
First note that, 
 by 
\eqref{nonzeroaverage},   $ \footnotesize  f_0^-(0,\e) = \vet{0}{1}  $, and thus 
$g_1^-(0,\e) $ in \eqref{basisG}  reduces to  
\begin{equation*}
g_1^-(0,\e) = f_1^-(0,\e) - \Big( f_1^-(0,\e), \vet{0}{1} \Big)\vet{0}{1} \, ,
\end{equation*} 
which satisfies \eqref{zeroaverage},  recalling also  that  the  first component 
of $  f_1^-(0,\e) $ is odd. 
In order to prove \eqref{exg41}-\eqref{exg44} we 
note that  $n(\mu,\e)$ in \eqref{normaliz} is real  by \eqref{nuisreal}, 
 and satisfies,  
by  \eqref{exf42}, \eqref{exf44},
$$
n(\mu,\e) = \frac{1}{1+r(\mu^2\e, \mu \e^2)}\Big[r(\e^2) +\mu\e \Big(\vet{-\sin(x)}{\cos(x)},\vet{\sin(x)}{\cos(x)} \Big)+r(\mu^2\e,\mu\e^2)\Big]  =  r(\e^2, \mu^2\e,\mu\e^2) \, .
$$
Hence, in view of \eqref{exf41}-\eqref{exf44}, 
 the vectors $ g^\sigma_k (\mu,\e) $
satisfy the expansion \eqref{exg41}-\eqref{exg44}.  
Finally at $\mu=0$ 
the vectors $g^\pm_k(0,\e)$, $k = 0,1$, are real being  real linear combinations of real vectors.
\end{proof}

We start now the proof of Proposition \ref{BexpG}.
It is useful to  decompose $ {\cal B}_{\mu,\e}$ in \eqref{calL.ham} as 
  \begin{equation*}
   {\cal B}_{\mu,\e}  =
 {\cal B}_\e  + {\cal B}^\flat +  {\cal B}^\sharp  \, , 
\end{equation*}
where $ {\cal B}_\e $, $  {\cal B}^\flat $, $  {\cal B}^\sharp  $ 
are the self-adjoint and reversibility preserving operators
\begin{align}
 \label{cBe}
&  {\cal B}_\e  := {\cal B}_{0,\e} := \left[
\begin{array}{cc}
1+a_\e(x) &   - (1+p_\e(x)) \partial_x   \\
\partial_x \circ (1+p_\e(x)) &    |D|
\end{array}
\right],  \\ \label{cBflat}
& {\cal B}^\flat :=   
\mu \begin{bmatrix}
0 & 0 \\
0 & g(D)
\end{bmatrix} , \qquad g(D) = {\sgn}(D) + \Pi_0 \, , \\
& \label{cBsharp} {\cal B}^\sharp :=  
\mu \begin{bmatrix}
0 &  -\im  p_{\e} \\
 \im p_{\e} & 0 
\end{bmatrix} \, .  %  p_{\geq 1}:=\e^{-1} p_\e \, .
\end{align}
Note that the operators $ {\cal B}^\flat $, $  {\cal B}^\sharp  $ are linear in $ \mu $. 
In order to prove \eqref{splitB}-%\eqref{BinG1}-
\eqref{BinG3} we  exploit the representation Lemma  \ref{lem:B.mat} and compute perturbatively the 
$ 4 \times 4 $ matrices,   associated, as in \eqref{matrix22}, to the 
self-adjoint and reversibility preserving operators 
  $ {\cal B}_\e$, $ {\cal B}^\flat$ and $ {\cal B}^\sharp$, 
in the basis $\mathcal{G} $.

\begin{lem}\label{lem:B1} {\bf (Expansion of $\mathtt{B}_\e$)} 
The self-adjoint  and reversibility preserving matrix  
$\tB_\e:= \tB_\e(\mu)$ associated, as in \eqref{matrix22}, with the self-adjoint and reversibility preserving operator $ {\cal B}_\e$, defined in \eqref{cBe}, with respect to the basis 
$\mathcal{G} $ of $ {\mathcal V}_{\mu,\e} $ 
in \eqref{basisG},  expands as
\begin{align}\label{expB1}
 \mathtt{B}_\e = \begin{pmatrix} 
 \e^2+\frac{\mu^2}{8}+r_1(\e^3, \mu\e^4) & \im r_2(\mu\e^3) & \vline & r_3(\e^3, \mu\e^2) & \im r_4(\mu\e^3) \\
 -\im r_2(\mu\e^3) & \frac{\mu^2}{8} & \vline & \im r_6(\mu\e) & 0 \\
 \hline
 r_3(\e^3, \mu\e^2) & -\im r_6(\mu\e) & \vline & 1+r_8(\e^3, \mu\e^2) & \im r_9(\mu\e^2)\\
 -\im r_4(\mu\e^3) & 0 & \vline & -\im r_9(\mu\e^2) & 0 \\
 \end{pmatrix}+\cO(\mu^2\e, \mu^3) \, .
\end{align}
\end{lem}
\begin{proof}
We expand the matrix $ \tB_\e(\mu) $ as 
\begin{equation}\label{TaylorexpBemu}
\tB_\e(\mu) = \tB_\e(0) + \mu (\pa_\mu \tB_\e)(0) + \frac{\mu^2 }{2} (\pa_\mu^2 \tB_0)(0) + \cO(\mu^2\e,\mu^3) \, . 
 \end{equation}
 To simplify notation, during this proof we often identify a matrix  with its matrix elements.
 \\[1mm]
{\bf The matrix $\tB_\e(0)$.} The main result of this long paragraph 
is to prove that  the matrix $\tB_\e(0)$ has the expansion \eqref{Bsoloeps}. 
The matrix $\tB_\e(0)$ is real, because 
the operator $ {\cal B}_{\e}$ is real and the basis $ \{ g_k^\pm(0,\e) \}_{k=0,1}$ is  real.
Consequently, by \eqref{revprop},  its  matrix elements $(\tB_\e(0))_{i,j}$ are real  whenever  $i+j$ is even and vanish for $i+j$ odd.
% It is also selfadjoint. 
In addition $g^-_0(0,\e)  = \footnotesize  \vet{0}{1}$ by \eqref{muiszero}, 
and, by \eqref{cBe}, we get 
$ {\cal B}_{\epsilon} g^-_0(0, \e)  = 0 $, for any $ \epsilon $.  
We deduce that the self-adjoint matrix $ \tB_\e(0) $ has the form 
 \begin{equation}\label{tBe}\tB_\e(0) =
\left( 
{\cal B}_\epsilon  \, g^\sigma_k(0, \e) , \, g^{\sigma '}_{k'}(0, \e)
\right)_{k, k'=0,1, \sigma, \sigma'  = \pm} = 
\begin{pmatrix}
  \begin{matrix}
\ta  &   0 \\
0 &  \tb  \\
  \end{matrix}
  & \vline &  
    \begin{matrix}
\alpha   &  0  \\
 0  &  0   \\
  \end{matrix}
   \\
\hline
    \begin{matrix}
\alpha  & 0  \\
  0  &  0  \\
  \end{matrix}
 & \vline &
  \begin{matrix}
\tc &   0\\
  0 &   0
  \end{matrix}
\end{pmatrix},
\end{equation}
with $\ta$, $\tb$, $\tc$, $\alpha$ real numbers depending on $\e$.
  We claim that
$ \tb = 0 $ for any $ \e $. As a first step we prove that
\begin{equation}\label{Jordancondition}
\text{ either } \ \tb=0 \, ,  \qquad\text{ or } \ \tb\neq 0 \ \text{ and } \ \ta=0=\alpha \, .
\end{equation}
Indeed, by Theorem 4.1 in \cite{NS}, 
the operator $ {\sL}_{0,\e} \equiv  {\mathcal L}_{0,\e}$ 
possesses, for any sufficiently small $\e \neq 0$, 
 the eigenvalue $ 0 $
with a four  dimensional 
generalized  Kernel  $
 \cW_\e := \text{span} \{ U_1, \tilde U_2, U_3, U_4 \} $, spanned 
by  $ \e$-dependent 
vectors $ U_1, \tilde U_2, U_3, U_4 $ satisfying 
	\eqref{genespace}. Note that $ U_1, \tilde U_2 $ are eigenvectors,  and 
	$ U_3, U_4$ generalized eigenvectors, of $ {\sL}_{0,\e}$ with eigenvalue $ 0 $. 
% \begin{equation}\label{genespace} 
% {\cal L}_{0,\e} U_1 = 0 \, , \  \ 
% {\cal L}_{0,\e} \tilde U_2 = 0 \, ,  \  \  {\cal L}_{0,\e}  U_3 =  c \tilde U_2 \, , \ 
% \  {\cal L}_{0,\e}  U_4 =  - U_1 \, . 
% \end{equation}
 By Lemma \ref{lem:Kato1} it 
 results that $ \cW_\e = {\mathcal V}_{0,\e} = \text{Rg}(P_{0,\e} )$
 and by \eqref{genespace} we have $ {\sL}_{0,\e}^2 = 0 $ on $ \mathcal{V}_{0,\e} $. 
 Thus the matrix 
 \begin{equation}\label{formaLep}
\tL_\e(0):=\tJ_4 \tB_\e(0) = 
\begin{pmatrix}
  \begin{matrix}
0  &   \tb \\
-\ta &   0  \\
  \end{matrix}
  & \vline &  
    \begin{matrix}
0  &  0  \\
 -\alpha  &  0   \\
  \end{matrix}
   \\
\hline
    \begin{matrix}
0 & 0  \\
 - \alpha   &  0  \\
  \end{matrix} 
 & \vline &
  \begin{matrix}
0 &   0\\
  -\tc &   0
  \end{matrix}
\end{pmatrix},
\end{equation}
which represents $ {\sL}_{0,\e}:\mathcal{V}_{0,\e}\to\mathcal{V}_{0,\e}$, 
satisfies $ \tL^2_\epsilon(0) = 0 $, namely 
$$ 
\tL^2_\epsilon(0) = \begin{pmatrix}
  \begin{matrix}
-\ta \tb  &   0 \\
0 &   -\ta \tb \\
  \end{matrix}
  & \vline &  
    \begin{matrix}
-\alpha  \tb  &  0  \\
 0 &  0   \\
  \end{matrix}
   \\
\hline
    \begin{matrix}
0 & 0  \\
 0  &  -\alpha \tb  \\
  \end{matrix}
 & \vline &
  \begin{matrix}
0 &   0\\
  0 &   0
  \end{matrix}
\end{pmatrix} =  0 \, .
 $$
This implies  \eqref{Jordancondition}. We now prove that the matrix $\tB_\e(0)$ defined in \eqref{tBe} expands as
\begin{equation}
\label{Bsoloeps}\tB_\e(0)  = 
\begin{pmatrix}
  \begin{matrix}
\ta  &   0 \\
0 &   \tb  \\
  \end{matrix}
  & \vline &  
    \begin{matrix}
\alpha   &  0  \\
 0  &  0   \\
  \end{matrix}
   \\
\hline
    \begin{matrix}
\alpha  & 0  \\
  0  &  0  \\
  \end{matrix}
 & \vline &
  \begin{matrix}
\tc &   0\\
  0 &   0
  \end{matrix}
\end{pmatrix}= \begin{pmatrix}\e^2+ {r(\e^3)} & 0 & \vline & r(\e^3) & 0 \\ 
 0& 0 & \vline & 0 & 0 \\
 \hline 
  r(\e^3)  & 0 & \vline & 1+ r(\e^3) & 0 \\
 0 & 0 & \vline & 0 & 0\end{pmatrix}.
\end{equation}
 We expand  the operator $ {\cal B}_\e$ in \eqref{cBe} as
\begin{equation}\label{pezziB}
{\cal B}_\e= {\cal B}_0+\e {\cal B}_1+ \e^2 {\cal B}_2 + \cO(\e^3) , 
\ 
 {\cal B}_0 := \begin{bmatrix} 1 & -\pa_x \\ \pa_x & |D| \end{bmatrix} \, ,\  {\cal B}_j := \begin{bmatrix} a_j(x) & -p_j(x)\pa_x \\ \pa_x\circ p_j(x) & 0 \end{bmatrix} \,,  \   j = 1,2 \, , 
\end{equation}
where the remainder term  $\cO(\e^3) \in \cL(Y, X)$ and,  by \eqref{SN1}-\eqref{SN2},
\begin{equation}\label{pezziap}
a_1(x) =p_1 (x)  =-2\cos(x) \, , \quad 
a_2(x)  =2-2\cos(2x) \, , \ \  \ p_2(x)  =\frac32-2\cos(2x) \, .
\end{equation}

\noindent {$ \bullet$  \it Expansion of $\ta=\e^2+r(\e^3)$. }
 By \eqref{exg41}  we split the real function $g_1^+(0,\e)$ as 
 \begin{equation}\label{pezzig1p}
g_1^+(0,\e) = f_1^+ + \e g_{1_1}^+ + \e^2 g_{1_2}^+ + \cO(\e^3),\  \ \ f_1^+ = \vet{\cos(x)}{\sin(x)},\  \ g_{1_1}^+ := \vet{2\cos(2x)}{\sin(2x)}, 
\ \ 
g_{1_2}^+ := \vet{even_0(x)}{odd(x)}, 
 \end{equation}
 where both $g_{1_2}^+$ and $\cO(\e^3)$ are vectors in  $H^1(\bT)$.
 Since $ {\cal B}_0f_1^+=\cJ^{-1} {\sL}_{0,0}f_1^+= 0$, and both $ {\cal B}_0$, $ {\cal B}_1$ are self-adjoint real operators, it results 
\begin{align}
 \ta&=\molt{{\cal B}_\e g^+_1(0,\e)}{ g^+_1(0,\e)} \notag \\
 &= \e \molt{{\cal B}_1 f_1^+}{f_1^+} + \e^2  \left[ \molt{{\cal B}_2 f_1^+}{f_1^+}+2\molt{{\cal B}_1 f_1^+}{g_{1_1}^+} + \molt{{\cal B}_0g_{1_1}^+}{g_{1_1}^+} \right]+\cO(\e^3) \, .
 \label{expa0}
\end{align}
By  \eqref{pezziB} one has
\begin{equation}\label{Bon1p}
{\cal B}_1f_1^+ = \vet{0}{2\sin(2x)}, \quad {\cal B}_2f_1^+=\vet{\frac12\cos(x)}{3\sin(3x)-\frac12 \sin(x)}, \quad {\cal B}_0g_{1_1}^+=\vet{0}{-2\sin(2x)}=- {\cal B}_1f_1^+.
\end{equation}
Then the  $\e^2$-term of $\ta$ is  $\molt{{\cal B}_2f_1^+}{f_1^+}+\molt{{\cal B}_1f_1^+}{g_{1_1}^+} $ and, by \eqref{expa0}, \eqref{Bon1p}, \eqref{pezzig1p}, a direct computation gives 
$\ta= \e^2 +r(\e^3)$ as stated in  \eqref{Bsoloeps}.

In particular, for $\e \neq 0$  sufficiently small, one has  $\ta \neq 0$ and the  second alternative in \eqref{Jordancondition} is  ruled out, implying   $\tb = 0$.\\
{ $ \bullet$  \it  Expansion of $\tc=1+r(\e^3)$. }
 By \eqref{exg43}  we split the real-valued function $g_0^+(0,\e)$ as
\begin{equation}\label{pezzig0p}
g_0^+(0,\e) = f_0^+ + \e g_{0_1}^+ + \e^2 g_{0_2}^+ + \cO(\e^3) \, , \  \  
f_0^+ = \vet10,
\   g_{0_1}^+:= \vet{\cos(x)}{-\sin(x)} \, , 
\   g_{0_2}^+:= \vet{even_0(x)}{odd(x)} \, .
\end{equation} 
Since, by \eqref{basestart} and  
\eqref{pezziB}, 
 ${\cal B}_0f_0^+= f_0^+$, and both ${\cal B}_0$, ${\cal B}_1$ are self-adjoint real operators,  
\begin{align}
 \tc&=\molt{{\cal B}_\e g^+_0(0,\e)}{ g^+_0(0,\e)} \notag \ \\
 &=1+ \e \molt{{\cal B}_1 f_0^+}{f_0^+} + \e^2  \left[ \molt{{\cal B}_2 f_0^+}{f_0^+}+2\molt{{\cal B}_1 f_0^+}{g_{0_1}^+} + \molt{{\cal B}_0g_{0_1}^+}{g_{0_1}^+} \right]+r(\e^3) \, ,
 \label{expc0}
\end{align}
where we  also used $\|f_0^+\| = 1$ and $ (f_0^+ , g_{0_1}^+ )= (f_0^+, g_{0_2}^+ ) =0$.
By  \eqref{pezziB}, \eqref{pezziap}  one has
\begin{equation}\label{Bon0p}
{\cal B}_1f_0^+ = 2\vet{-\cos(x)}{\sin(x)}, \quad {\cal B}_2f_0^+=\vet{2-2\cos(2x)}{4\sin(2x)}, \quad {\cal B}_0g_{0_1}^+=2\vet{\cos(x)}{-\sin(x)}=- {\cal B}_1f_0^+.
\end{equation}
Then the  $\e^2$-term of $\tc$ is  $\molt{{\cal B}_2f_0^+}{f_0^+}+\molt{{\cal B}_1f_0^+}{g_{0_1}^+} $ and, by  \eqref{pezzig0p}-\eqref{Bon0p},
we conclude that 
$\tc= 1 +r(\e^3)$ as stated in \eqref{Bsoloeps}.\\
{$ \bullet$  \it  Expansion of $\alpha=\cO(\e^3)$. }
By  \eqref{pezzig1p}, \eqref{pezzig0p} and since
${\cal B}_0, {\cal B}_1$ are self-adjoint and real  we have 
\begin{align*}
 &\alpha =\molt{{\cal B}_\e g^+_1(0,\e)}{ g^+_0(0,\e)}
 = \molt{{\cal B}_0 f_1^+}{f_0^+} + \e \left[ \molt{{\cal B}_1 f_1^+}{f_0^+} + \molt{{\cal B}_0 f_1^+ }{g_{0_1}^+} + \molt{{\cal B}_0 g_{1_1}^+}{ f_0^+ }\right] +  \notag  \\ 
 &  \e^2  \big[ \molt{{\cal B}_2 f_1^+}{f_0^+}+\molt{{\cal B}_1 f_1^+}{g_{0_1}^+} +\molt{{\cal B}_1 f_0^+}{g_{1_1}^+} +\molt{{\cal B}_0 g_{1_2}^+}{f_0^+} + \molt{{\cal B}_0g_{1_1}^+}{g_{0_1}^+} + \molt{{\cal B}_0f_1^+}{g_{0_2}^+} \big]+r(\e^3)  \, . 
 \end{align*}
Recalling that ${\cal B}_0 f_1^+ = 0$ and  ${\cal B}_0 f_0^+ = f_0^+$, we arrive at 
  \begin{align}
 \alpha  & = \e \left[ \molt{{\cal B}_1 f_1^+}{f_0^+} + \molt{g_{1_1}^+}{ f_0^+ }\right] 
 \notag \\ 
 \notag
 & \quad +  \e^2  \big[ \molt{{\cal B}_2 f_1^+}{f_0^+}+\molt{{\cal B}_1 f_1^+}{g_{0_1}^+} +\molt{{\cal B}_1 f_0^+}{g_{1_1}^+} +\molt{g_{1_2}^+}{f_0^+} + \molt{{\cal B}_0g_{1_1}^+}{g_{0_1}^+}  \big]+r(\e^3)=r(\e^3) \, , 
\end{align}
using that,  by   \eqref{pezzig1p}, \eqref{Bon1p}, \eqref{pezzig0p} \eqref{Bon0p}, all the scalar products in the formula vanish.
 
 We have proved  the expansion \eqref{Bsoloeps}. 
\\[1mm]
{\bf Linear terms in $ \mu $.}
We now compute the terms  of $\tB_\e(\mu)$  that are linear in $\mu$. It results 
\begin{equation}\label{MatrixX}
\pa_\mu \tB_\e(0) = X + X^* 
\qquad \text{where} \qquad X :=
 \big( {\cal B}_\e g_k^\sigma(0,\e), (\pa_\mu g^{\sigma'}_{k'})(0,\e) \big)_{k,k'=0,1,  \sigma,\sigma'=\pm} \, .  
 \end{equation}
We now prove that 
\begin{equation}\label{matX}
   X = \begin{pmatrix} 
 \cO(\e^4) & 0 & \vline & \cO(\e^2) & 0 \\ 
 \cO(\e^3)  & 0 & \vline & \cO( \e) & 0 \\
 \hline
 \cO(\e^4) & 0 & \vline & \cO(\e^2) & 0 \\
 \cO(\e^3) & 0 & \vline & \cO(\e^2) & 0
 \end{pmatrix}.
 \end{equation}
The matrix $ \tL_\e (0) $ in \eqref{formaLep} where $\tb=0$, represents the action 
of the operator $\cL_{0,\e}:\mathcal{V}_{0,\e}\to\mathcal{V}_{0,\e}$ in the basis
$ \{ g^{\sigma}_k (0,\e) \} $  and then we deduce that  
 $ \cL_{0,\e} g_1^-(0,\e) = 0 $,  $ \cL_{0,\e} g_0^-(0,\e) = 0 $. 
Thus %  lie in the kernel of $\cL_{0,\e}$ and consequently in the kernel of 
also $ {\cal B}_\e  g_1^-(0,\e) = 0 $, $ {\cal B}_\e  g_0^-(0,\e)  = 0 $, for every $\e$,
and the second and the fourth column of the matrix $X$ in \eqref{matX} are zero. 
In order to compute the other two columns we use the expansion of the derivatives, 
where denoting with a dot the derivative w.r.t. $\mu$,
\begin{align}\label{reuse}
 &\dot g^{+}_{1}(0,\e) = \frac\im4 	\vet{\sin(x)}{\cos(x)}+\im\e \vet{odd(x)}{even(x)}+\cO(\e^2) 	\, , \qquad  \dot g^{+}_{0}(0,\e) = \im\e \vet{odd(x)}{even_0(x)}+\cO(\e^2) \, ,\\  
\notag 
 &\dot g^{-}_{1}(0,\e) = \frac\im4\vet{\cos(x)}{-\sin(x)}+\im\e \vet{even(x)}{odd(x)}+\cO(\e^2) \, ,\quad \dot g^{-}_{0}(0,\e) =\e \Big(\vet{\sin(x)}{\cos(x)}+\im\vet{even_0(x)}{odd(x)} \Big)+\cO(\e^2) 
\end{align}
that follow by  \eqref{exg41}-\eqref{exg44}. 
In view of \eqref{azioneJ}, \eqref{exg41}-\eqref{exg44},  \eqref{formaLep}  and since $ {\cal B}_\e g_k^\sigma(0,\e)=-\cJ \cL_\e g_k^\sigma(0,\e) $, we have
\begin{align}
{\cal B}_\e g_1^+(0,\e) &= \big(\e^2+r(\e^3)\big)\,  \cJ g_1^-(0,\e) + r(\e^3)\, \cJ f_0^- =\e^2 \vet{\cos (x)}{\sin (x)} + r(\e^3)\Big(\vet{1}{0} + \vet{even_0(x)}{odd(x)} \Big) \, , \notag \\
{\cal B}_\e g_0^+(0,\e) &= r(\e^3)\cJ g_1^-(0,\e) +\big(1+r(\e^3)\big)\cJ f_0^-= \vet{1}{0}+r(\e^3)\Big(\vet{1}{0} + \vet{even_0(x)}{odd(x)} \Big) \, .\label{reuse2}
\end{align}
 The other two columns of the matrix $X$  in \eqref{MatrixX} have the expansion  \eqref{matX},   by  \eqref{reuse} and \eqref{reuse2}.
\\[1mm]
{\bf Quadratic terms in $ \mu $.}
% Finally we compute the expansion of $\tB_\e(\mu)$ in $\mu^2$. 
By denoting with a double dot the double derivative w.r.t. $\mu$, we have 
\begin{equation}\label{dersec}
\pa_\mu^2 \tB_0(0) = \molt{{\cal B}_0 f_k^\sigma}{\ddot g_{k'}^{\sigma'}(0,0)}+\molt{\ddot g_{k}^{\sigma}(0,0)}{{\cal B}_0 f_{k}^{\sigma'}}+2\molt{{\cal B}_0\dot g_k^\sigma(0,0)}{\dot g_{k'}^{\sigma'}(0,0)}=:Y+Y^*+2Z \, .
\end{equation}
We claim that $Y = 0 $. Indeed, its first, second and fourth column are zero, since ${\cal B}_0f_k^\sigma=0$ for $f_k^\sigma \in \{ f_1^+,f_1^-,f_0^- \} $.
 The third column is also zero 
by noting  that $ {\cal B}_0 f_0^+ = f_0^+ $ and
$$
\ddot g_{1}^{+}(0,0) = \vet{even_0(x)+\im odd(x)}{odd(x)  +\im  even_0(x)}, \ \ \ddot g_{1}^{-}(0,0) = \vet{odd(x)  +\im  even_0(x)}{even_0(x)+\im odd(x)}, \ \ \ddot g_{0}^{+}(0,0)=\ddot 
g_{0}^{-}(0,0)=0 \, .
$$
We claim that 
\begin{align}\label{matZ}
 Z = \molt{{\cal B}_0\dot g_k^\sigma(0,0)}{\dot g_{k'}^{\sigma'}(0,0)}_{\substack{k,k'=0,1,\\\sigma,\sigma'=\pm}} = \begin{pmatrix} 
 \frac{1}{8} & 0 & \vline & 0 & 0 \\
 0 & \frac{1}{8} & \vline & 0 & 0 \\
 \hline
 0 & 0 & \vline & 0 & 0 \\
 0 & 0 & \vline & 0 & 0 \\
 \end{pmatrix} \, .
\end{align}
Indeed,  by \eqref{reuse}, we have  $\dot g^+_0(0,0)=\dot g^-_0(0,0)= 0$. 
Therefore the last two columns of $Z$, and by self-adjointness  the last two rows, are zero. 
By \eqref{reuse}, 
$\dot g^+_1(0,0) = \frac{\im}{4}\footnotesize \vet{\sin (x)}{\cos (x)}$
and 
$\dot g^-_1(0,0) = \frac{\im}{4} \footnotesize\vet{\cos (x)}{-\sin (x)}$,
so that 
$ \footnotesize {\cal B}_0  \dot g^+_1(0,0) = \frac{\im}{2} \vet{\sin (x)}{\cos (x)} $  and 
$ \footnotesize {\cal B}_0 \dot g^-_1(0,0) = \frac{\im}{2} \vet{\cos (x)}{-\sin (x)} $, 
and we obtain the matrix \eqref{matZ} computing the scalar products.

In conclusion \eqref{TaylorexpBemu}, \eqref{MatrixX}, \eqref{matX}, \eqref{dersec}, the fact that $Y=0$ and \eqref{matZ} imply \eqref{expB1}, using 
also the selfadjointness of 
$\tB_\e$
 and 
\eqref{revprop}.
\end{proof}

We now consider $ {\cal B}^\flat $. 

\begin{lem}\label{lem:B2}
{\bf (Expansion of $\mathtt{B}^\flat$)} 
The self-adjoint and reversibility-preserving matrix 
$\tB^\flat$ associated, as in \eqref{matrix22}, to the self-adjoint 
and reversibility-preserving operator $  {\cal B}^\flat$, defined in \eqref{cBflat}, with respect to the basis $\mathcal{G}$  of $ {\mathcal V}_{\mu,\e} $
in \eqref{basisG}, admits  the expansion
\begin{equation}\label{EBflat}
\mathtt{B}^\flat=  \begin{pmatrix} 
-\frac{\mu^2}{4} & \im (\frac{\mu}{2}+r_2(\mu \e^2)) & \vline & 0 & 0 \\         
-\im (\frac{\mu}{2}+r_2(\mu \e^2)) & -\frac{\mu^2}{4} & \vline & 
\im r_6(\mu \e) & 0 \\ \hline
0 & - \im r_6(\mu \e)   & \vline & 0 & 0 \\
0 & 0 & \vline & 0 & \mu
\end{pmatrix}+\cO(\mu^2\e,\mu^3)\,. 
\end{equation}
\end{lem}
\begin{proof}
We have to compute the 
expansion of the matrix entries $ ( {\cal B}^\flat g^\sigma_k(\mu,\e), g^{\sigma'}_{k'}(\mu,\e)) $. 
The operator $  {\cal B}^\flat $ in 
 \eqref{cBflat} is linear in $\mu$ and by \eqref{exg41}, \eqref{exg42},   \eqref{zeroaverage} 
and the identities 
$ \sgn(D) \sin(kx)  = - \im\cos(kx) $ and  
$ \sgn(D)\cos(kx) = \im \sin(kx) $ for any $ k \in \bN $,  
 we have 
\begin{align*}
 {\cal B}^\flat g^+_1(\mu,\e) &= -\im\mu\vet{0}{\cos(x)}-\frac{\mu^2}{4} \vet{0}{\sin(x)}  - \im\mu\e \vet{0}{\cos(2x)}+\im \cO(\mu\e^2) \vet{0}{even_0(x)} +\cO(\mu^2\e,\mu^3) \, ,\\
 {\cal B}^\flat g^-_1(\mu,\e) &= \im\mu\vet{0}{\sin(x)} -\frac{\mu^2}{4} \vet{0}{\cos(x)}+\im\mu\e \vet{0}{\sin(2x)}+\im  \cO(\mu\e^2) \vet{0}{ odd(x)}+\cO(\mu^2\e,\mu^3) \, .
 \end{align*}
Note that 
 $\mu 
\footnotesize\begin{bmatrix}
0 & 0 \\
0 & \Pi_0
\end{bmatrix}g^-_1(\mu,\e) = \cO(\mu^3\e, \mu^2\e^2)$  thanks to the property \eqref{zeroaverage} of the basis $\mathcal{G}$. \\
 In addition, by  \eqref{exg43}-\eqref{exg44},  we get that 
$$
 {\cal B}^\flat g^+_0(\mu,\e) = \im\mu\e \vet{0}{\cos(x)}+ \im \cO(\mu\e^2) \vet{0}{even_0(x)}+ \cO(\mu^2\e) \,  , \quad  {\cal B}^\flat g^-_0(\mu,\e) = \vet{0}{\mu}+\cO(\mu^2\e) \, .
$$
Taking the scalar products of the above expansions of 
 $ {\cal B}^\flat g^\sigma_k (\mu,\e) $
 with the functions $g^{\sigma'}_{k'}(\mu,\e) $ expanded as in 
 \eqref{exg41}-\eqref{exg44}
 we deduce \eqref{EBflat}. 
\end{proof}

\begin{rmk}\label{rem:newbasis}
The $(2,2)$ entry in the matrix $ \mathtt{B}^\flat $ in \eqref{EBflat} has no terms  $ \cO (\mu\e^k ) $, thanks to property \eqref{zeroaverage}. This property is fundamental in order to 
verify that the $(2,2)$ entry of the matrix $ E $ in \eqref{BinG1} starts with 
$ - \frac{\mu^2}{8} $ and therefore it is negative for $ \mu $ small. 
Such  property does not hold for the first basis ${\cal F}$ defined in 
\eqref{basisF}, and this motivates the use  of the second basis ${\mathcal G}$.
\end{rmk}

Finally we consider ${\cal B}^\sharp$.

\begin{lem}\label{lem:B3}
{\bf (Expansion of  $\mathtt{B}^\sharp$)}
The self-adjoint and reversibility-preserving matrix 
$\tB^\sharp$ associated, as in \eqref{matrix22}, to the self-adjoint 
and reversibility-preserving operators $ {\cal B}^\sharp$, defined in  \eqref{cBsharp}, with respect to the basis $\mathcal{G}$  of $ {\mathcal V}_{\mu,\e} $
in \eqref{basisG}, admits the expansion
\begin{equation}\label{tBsharp}
 \mathtt{B}^\sharp = \begin{pmatrix} 
 0 & \im  r_2(\mu\e^2) & \vline & 0 & \im  r_4(\mu\e) \\ 
 -\im r_2(\mu\e^2)& 0 & \vline & -\im r_6(\mu\e) & 0 \\
 \hline
 0 & \im r_6(\mu\e) & \vline & 0 & - \im r_9(\mu\e^2) \\
 -\im r_4(\mu\e) & 0 & \vline & \im r_9(\mu\e^2) & 0
 \end{pmatrix}+\cO(\mu^2\e) \, .
\end{equation}
\end{lem}
\begin{proof}
Since ${\cal B}^\sharp  = -\im \mu p_\e \cJ$ and  
$p_\e=\cO(\e)$
  by \eqref{SN1},   we have the expansion
\begin{equation}\label{TaylorexpBsharpemu}
 \big( {\cal B}^\sharp g_k^\sigma(\mu,\e), g_{k'}^{\sigma'}(\mu,\e) \big) = 
\big({\cal B}^\sharp g_k^\sigma(0,\e), g_{k'}^{\sigma'}(0,\e) \big) + \cO(\mu^2\e) \, .
\end{equation}
We claim that the matrix entries
$ ( {\cal B}^\sharp g^{\sigma}_k (0, \e) , g^{\sigma}_{k'} (0, \e) )   $, $  k, k' =  0,1 $ are zero. Indeed they are real by \eqref{revprop}, and also % simultaneously 
purely imaginary, since  the operator ${\cal B}^\sharp$
 is purely imaginary\footnote{An operator $\mathcal{A}$ is \emph{purely imaginary} if $\bar{\mathcal{A}}=-\mathcal{A}$. A purely imaginary operator sends real functions into purely imaginary ones.}
  and the basis $ \{ g_k^\pm(0,\e) \}_{k=0,1}$ is  real. 
 Hence $\tB^\sharp$ has the form 
 \begin{equation}\label{formatBsharp}
 \mathtt{B}^\sharp = \begin{pmatrix} 
 0 & \im\beta & \vline & 0 & \im\delta \\ 
 -\im\beta & 0 & \vline & -\im\gamma & 0 \\
 \hline
 0 & \im\gamma & \vline & 0 & \im\eta \\
 -\im\delta & 0 & \vline & -\im\eta & 0
 \end{pmatrix}+\cO(\mu^2\e) 
 \quad \text{where} \quad 
 \left\{\begin{matrix}\molt{{\cal B}^\sharp g_1^-(0,\e)}{g_1^+(0,\e)}=:\im\beta \, , \\   \molt{{\cal B}^\sharp g_1^-(0,\e)}{g_0^+(0,\e)}=:\im\gamma \, , \\
\molt{{\cal B}^\sharp g_0^-(0,\e)}{g_1^+(0,\e)}=:\im\delta \, , 
\\ \molt{{\cal B}^\sharp g_0^-(0,\e)}{g_0^+(0,\e)}=:\im\eta \, ,\end{matrix}\right.
 \end{equation}
and $\alpha $, $ \beta $, $ \gamma $, $ \delta$ are real numbers.
As ${\cal B}^\sharp  = \cO(\mu \e)$ in $\cL(Y)$, we get immediately 
that $ \gamma =r( \mu\e ) $ and 
$   \delta =  r(\mu\e) $. 
Next we  compute the expansion of $\beta$ and $\eta$.
We 
 split the operator $ {\cal B}^\sharp$ in \eqref{cBsharp} as
\begin{equation}\label{pezziBsharp}
{\cal B}^\sharp = \im\mu \e{\cal B}_1^\sharp+  \cO(\mu\e^2) \, ,
\qquad {\cal B}_1^\sharp :=   -p_1(x) \cJ \, , 
\end{equation}
with $p_1(x)$ in  \eqref{pezziap} and $\cO(\mu \e^2) \in \cL(Y)$.
By \eqref{pezziBsharp} and the expansion  \eqref{exg41}-\eqref{exg44},
$g_1^+(0,\e) = f_1^+ + \cO(\e)$, 
$g_1^-(0,\e)=f_1^-+\cO(\e)$,  
$g_0^+(0,\e)=f_0^+  +\cO(\e)$, 
$g_0^-(0,\e) = \footnotesize \vet{0}{1}$
we obtain  
$$
 \beta = \mu\e \molt{{\cal B}_1^\sharp f_1^-}{f_1^+} +r(\mu\e^2) \, ,  \ \ \   
   \eta =  \mu\e \molt{{\cal B}_1^\sharp f_0^-}{f_0^+} +r(\mu\e^2)  \, . 
$$
Computing
$ \footnotesize  {\cal B}_1^\sharp f_1^- =  \vet{1+\cos(2x)}{\sin(2x)} $,
$ \footnotesize  {\cal B}_1^\sharp f_0^- =  \vet{2\cos(x)}{0} $
and the various scalar products with the vectors $f^\sigma_k$ in \eqref{funperturbed},  we  get
$ \beta =r(\mu\e^2) $,  
 $   \eta = r(\mu\e^2) $. 
Using also  \eqref{TaylorexpBsharpemu} and \eqref{formatBsharp},   one gets  \eqref{tBsharp}.
 \end{proof}

Lemmata \ref{lem:B1}, \ref{lem:B2} and  \ref{lem:B3} imply  Proposition \ref{BexpG}.

\section{Block-decoupling}\label{sec:block}

The $ 4 \times 4 $ Hamiltonian and reversible 
matrix $\tL_{\mu,\e} = \tJ_4 \tB_{\mu,\e} $ obtained 
in Proposition \ref{BexpG}, has the form 
\begin{equation}\label{Lepmu}
\tL_{\mu,\e}= 
\tJ_4 \begin{pmatrix} 
 E &  F \\ 
 F^* &  G 
\end{pmatrix}
=
\begin{pmatrix} 
\tJ_2 E & \tJ_2 F \\ 
\tJ_2 F^* &\tJ_2 G 
\end{pmatrix} ,
\end{equation}
where $ E, G, F $ are the $ 2 \times 2 $ matrices in \eqref{BinG1}-\eqref{BinG3}. 
In particular $\tJ_2 E$ has the form 
\begin{equation}\label{J2E.1}
\tJ_2 E =
 \begin{pmatrix} 
- \im \big( \frac{\mu}{2}+ r_2(\mu\e^2,\mu^2\e,\mu^3) \big)  & -\frac{\mu^2}{8}(1+r_5(\e,\mu))\\
- \e^2(1+r_1'(\e,\mu\e^2))+\frac{\mu^2}{8}(1+r_1''(\e,\mu))  & - \im \big( \frac{\mu}{2}+ r_2(\mu\e^2,\mu^2\e,\mu^3) \big)   \\
 \end{pmatrix}
\end{equation}
 and therefore possesses 
 two eigenvalues with non-zero real part (``Benjamin-Feir" eigenvalues), as long as its two off-diagonal entries have the  same sign, see the discussion below \eqref{J2E}.
In order to prove that also the full $ 4 \times 4 $ matrix 
$ \tL_{\mu,\e} $  in \eqref{Lepmu} possesses Benjamin-Feir unstable eigenvalues, 
we aim to eliminate the coupling term $  \tJ_2 F $  by a  change of variables.
More precisely in this section we conjugate the matrix $\tL_{\mu,\e}$ in \eqref{Lepmu} 
 to the  Hamiltonian and reversible
{\it block-diagonal} matrix
$\tL_{\mu,\e}^{(3)}$ in \eqref{L3.fin}, 
$$
\tL_{\mu,\e}^{(3)} = 
\begin{pmatrix} 
\tJ_2 E^{(3)} & 0 \\ 
0 &\tJ_2  G^{(3)}
\end{pmatrix} ,
$$
where $\tJ_2 E^{(3)}$ is a $2 \times 2$ matrix with the same form as
\eqref{J2E.1} (clearly with different remainders, but of the same order).
The {spectrum of the} $ 4 \times 4 $ matrix $ \tL_{\mu,\e}^{(3)}$, which coincides with that of $ \tL_{\mu,\e} $, contains the 
Benjamin-Feir unstable eigenvalues of the $ 2 \times 2 $ matrix $\tJ_2 E^{(3)} $
(it turns out that the two  eigenvalues of $ \tJ_2  G^{(3)}$ are purely imaginary). 
This will prove Theorem \ref{TeoremoneFinale}. 

The block-diagonalization  of  $ \tL_{\mu,\e} $ is achieved in 
three steps, in Lemma \ref{decoupling1}, Lemma \ref{decoupling2}, and finally Lemma \ref{ultimate}. 
Motivations and goals of each step 
were % extensively 
described  at the end of Section \ref{sec:2}.

\subsection{First step of Block-decoupling} \label{sec:omue}
We write  the matrices $ E, F, G $ in \eqref{splitB} as 
\begin{equation}\label{splitEFG}
E = 
\begin{pmatrix} 
E_{11} & \im E_{12} \\ 
- \im E_{12} & E_{22} 
\end{pmatrix}\, , \quad
F = 
\begin{pmatrix} 
F_{11} & \im F_{12} \\ 
\im F_{21} & F_{22} 
\end{pmatrix} \, , \quad
G = 
\begin{pmatrix} 
G_{11} & \im G_{12} \\ 
- \im G_{12} & G_{22} 
\end{pmatrix} 
\end{equation} 
where the real numbers $ E_{ij}, F_{ij}, G_{ij} $, $ i , j = 1,2 $, have the expansion 
given in  \eqref{BinG1}-\eqref{BinG3}. 

\begin{lem}\label{decoupling1}
 Conjugating the Hamiltonian and reversible matrix $\tL_{\mu,\e} = \tJ_4 \tB_{\mu,\e} $ 
obtained in Proposition \ref{BexpG} through the symplectic and reversibility-preserving 
 $ 4 \times 4 $-matrix 
 \begin{align}\label{Ychange}
 Y = \uno_4 + m \begin{pmatrix} 0 & - P \\ Q & 0 \end{pmatrix}
 \  \text{with} \  \ 
 Q:=\begin{pmatrix} 1 & 0 \\ 0 & 0\end{pmatrix}  \, , 
 \   P:=\begin{pmatrix} 0 & 0 \\ 0 & 1\end{pmatrix} \, , 
 \ m := m(\mu,\e):=-\frac{F_{11}(\mu,\e)}{G_{11}(\mu,\e)},  
\end{align} 
where $ m =  r(\e^3, \mu \e^2, \mu^2\e, \mu^3 ) $ is a real number, 
we obtain the Hamiltonian and reversible matrix 
\begin{align}
&\tL_{\mu,\e}^{(1)} := Y^{-1} \tL_{\mu,\e} Y = \tJ_4\tB^{(1)}_{\mu,\e}  =
\begin{pmatrix}   \tJ_2 E^{(1)}  & \tJ_2  F^{(1)}  \\  
\tJ_2 [F^{(1)}]^*  &  \tJ_2 G^{(1)}  \end{pmatrix}  \label{LinH}
%% = 
%\begin{pmatrix}    \tL^{(1),11}_{\mu,\e} &  \tL^{(1),10}_{\mu,\e} \\  \tL_{\mu,\e}^{(1),01} &  \tL_{\mu,\e}^{(1),00} \end{pmatrix}
 \end{align}
where $ \tB_{\mu,\e}^{(1)} $ is a self-adjoint and reversibility-preserving 
 $ 4 \times 4$  matrix 
  \begin{equation}\label{splitB1}
\tB_{\mu,\e}^{(1)} =
\begin{pmatrix} 
E^{(1)} & F^{(1)} \\ 
[F^{(1)}]^* & G^{(1)} 
\end{pmatrix}, \quad 
E^{(1)}  = [E^{(1)}]^* \, , \ G^{(1)} = [G^{(1)}]^* \, , 
\end{equation} 
 where the $ 2 \times 2 $  matrices  $E^{(1)} $, $ G^{(1)} $  have the same expansion
 \eqref{BinG1}-\eqref{BinG2}   of $ E, G $ and 
\begin{equation}\label{BinH}
 F^{(1)} = 
 \begin{pmatrix} 
 0 & \im  r_4(\mu\e,\mu^3)  \\
  \im  r_6(\mu\e, \mu^3)  & r_7(\mu^2\e,\mu^3) 
 \end{pmatrix}\, . 
 \end{equation}
 Note that  the entry $ F^{(1)}_{11} $ is $ 0 $,
the other entries of $ F^{(1)}$ have  the same size as for $ F $ in \eqref{BinG3}.  
\end{lem}

\begin{proof}
The matrix $Y $ is symplectic, i.e. \eqref{sympmatrix} holds,
and since $m$ is real, it is reversibility preserving, i.e. 
satisfies  \eqref{revprop}.  
By \eqref{sympchange},  
\begin{equation}\label{B1forma}
   \tB_{\mu,\e}^{(1)} = Y^* \tB_{\mu,\e} Y = \begin{pmatrix}  E^{(1)} & F^{(1)} \\ [F^{(1)}]^* &  
   G^{(1)}  \end{pmatrix},
\end{equation}
where, by \eqref{Ychange} and  \eqref{splitEFG},  the self-adjoint matrices 
$E^{(1)}, G^{(1)}  $ are 
\begin{equation}\label{E1G1}
\begin{aligned}
& E^{(1)} := E+ m ( QF^* + FQ) +m^2 QGQ = E +
\begin{pmatrix}  2 m F_{11} + m^2 G_{11} & - \im m F_{21} \\ \im m F_{21} &  
0  \end{pmatrix}   \, , \\ 
& G^{(1)} := G-m  (PF +  F^*P) +m^2 PEP = G + 
\begin{pmatrix}  0 &  \im m F_{21}  \\ 
- \im m F_{21}  &   - 2m F_{22} +m^2 E_{22}  \end{pmatrix} \, .
\end{aligned}
\end{equation}
Similarly,  the off-diagonal $ 2 \times 2 $ matrix $F^{(1)} $ is 
\begin{equation}\label{F1}
F^{(1)} := F + m (QG -  EP) - m^2 QF^*P =
 \begin{pmatrix}  0 &  \im (F_{12} + mG_{12} - mE_{12} + m^2 F_{21}) \\ 
 \im F_{21}  &  F_{22} - m E_{22}  \end{pmatrix}  \, , 
 \end{equation}
where we have used that the first entry  of this matrix is $F_{11} + m G_{11} = 0 $, 
by the definition of $ m $ in \eqref{Ychange}. 
By \eqref{B1forma}-\eqref{F1} and \eqref{BinG1}-\eqref{BinG3}
we deduce the expansion of $ \tB_{\mu,\e}^{(1)} $ in \eqref{BinH}, \eqref{splitB1} and consequently that of \eqref{LinH}. 
\end{proof}

\subsection{Second step of Block-decoupling} \label{sec:5.2}

We now perform 
a further step of block decoupling, obtaining the new Hamiltonian and reversible 
matrix $ \tL_{\mu,\e}^{(2)}  $ in \eqref{sylvydec}  
where the $2 \times 2 $ matrix $ \tJ_2 E^{(2)} $
has still  the Benjamin-Feir unstable eigenvalues
and the size of the new coupling matrix $\tJ_2 F^{(2)}$ 
is  much smaller  than $\tJ_2 F^{(1)}$.
In particular note that  the entries of 
$  F^{(2)} $ in \eqref{Bsylvy} have size 
$ \cO (\mu^2 \e^3, \mu^3 \e^2, \mu^5 \e, \mu^7) $ whereas those
of $  F^{(1)} $ in \eqref{BinH} are $ \cO (\mu \e^3, \mu^3)$.

\begin{lem} {\bf (Step of block-decoupling)}\label{decoupling2}
There exists a $2\times 2$ reversibility-preserving matrix $ X $,
analytic in $ (\mu, \e) $, of the form
\begin{equation}\label{Xsylvy}
X=\begin{pmatrix} x_{11} & \im x_{12} \\ \im x_{21} & x_{22} \end{pmatrix}
=
\begin{pmatrix}
   r_{11}(\mu^2, \mu \e) & \im\, r_{12}(\mu^3, \mu \e) \\
  \im r_{21}(\e, \mu^2) &   r_{22}(\mu^3, \mu \e)\end{pmatrix}, 
  \quad x_{11},\, x_{12},\, x_{21},\, x_{22} \in \bR \, , 
\end{equation}
such that, by conjugating the Hamiltonian and reversible matrix 
$\tL_{\mu,\e}^{(1)}$, defined in \eqref{LinH}, with the symplectic and reversibility-preserving $4\times 4$ matrix 
\begin{equation}\label{formaS}
\exp\left(S^{(1)} \right) \, , 
\quad \text{ where } 
\qquad S^{(1)} := \tJ_4 \begin{pmatrix} 0 & \Sigma \\ \Sigma^* & 0 \end{pmatrix} \, , \qquad \Sigma:= \tJ_2 X \, , 
\end{equation}
we get the Hamiltonian and reversible matrix  
\begin{equation}\label{sylvydec}
 \tL_{\mu,\e}^{(2)} := \exp\left(S^{(1)} \right)\tL_{\mu,\e}^{(1)} \exp\left(-S^{(1)} \right)= \tJ_4 \tB_{\mu,\e}^{(2)} =
\begin{pmatrix}   \tJ_2 E^{(2)}  & \tJ_2  F^{(2)}  \\  
\tJ_2 [F^{(2)}]^*  &  \tJ_2 G^{(2)}  \end{pmatrix},
\end{equation}
 where the $ 2 \times 2 $ self-adjoint and reversibility-preserving  matrices  $E^{(2)} $, $ G^{(2)} $  have the same expansion
 of $ E^{(1)} $, $ G^{(1)} $, namely of 
 $ E, G $, given in  \eqref{BinG1}-\eqref{BinG2},  and 
\begin{equation}\label{Bsylvy}
 F^{(2)} = \begin{pmatrix} F^{(2)}_{11} & \im F^{(2)}_{12} \\ \im F^{(2)}_{21} & F^{(2)}_{22}  \end{pmatrix}=
 \begin{pmatrix} 
  r_3(\mu^2 \e^3, \mu^3 \e^2, \mu^5 \e, \mu^7) 
  &
  \im  r_4(\mu^2\e^3, \mu^4 \epsilon^2, \mu^5 \epsilon,  \mu^7)  \\
  \im  r_6(\mu^2 \epsilon^3, \mu^4 \e^2, \mu^5 \epsilon, \mu^7) 
  &
    r_7(\mu^3 \epsilon^3, \mu^4\e^2, \mu^6 \epsilon, \mu^8)
 \end{pmatrix}\, . 
 \end{equation}
\end{lem}

\begin{rmk}
The new matrix $ \tL_{\mu,\e}^{(2)} $ 
in  \eqref{sylvydec} is still analytic  in $ (\mu, \e)$, as $ \tL_{\mu,\e}^{(1)}$.
This is not obvious a priori, since the spectrum of the matrices 
$ \tJ_2 E^{(1)} $ and $ \tJ_2 G^{(1)} $ is shrinking to zero as $ (\mu, \e) \to  0 $. 
\end{rmk}

The rest of the section is devoted to the proof of Lemma \ref{decoupling2}.
We denote for simplicity $ S = S^{(1)} $. 

The matrix  $\text{exp}(S)$ is symplectic and reversibility preserving
because the matrix $ S $ in \eqref{formaS} is Hamiltonian and 
reversibility preserving, cfr. Lemma \ref{lem:S.conj}. 
Note that $ S $  is reversibility preserving  since $X$ 
has  the form  \eqref{Xsylvy}. %  thus reversibility preserving. 

We now expand in Lie series 
the Hamiltonian and reversible matrix $ \tL_{\mu,\e}^{(2)} 
= \exp (S)\tL_{\mu,\e}^{(1)} \exp (-S) $. 

We split $\tL_{\mu,\e}^{(1)}$ 
into its $2\times 2$-diagonal and off-diagonal Hamiltonian and reversible matrices
\begin{align}
& \qquad  \qquad \qquad  \qquad  \qquad \qquad \tL_{\mu,\e}^{(1)} = D^{(1)} + R^{(1)}  \, , \notag \\
& 
D^{(1)} :=\begin{pmatrix} D_1 & 0 \\ 0 & D_0 \end{pmatrix} =  \begin{pmatrix} \tJ_2 E^{(1)} & 0 \\ 0 & \tJ_2 G^{(1)}  \end{pmatrix}, \quad 
R^{(1)} := \begin{pmatrix}  0 & \tJ_2 F^{(1)} \\ \tJ_2 [F^{(1)}]^* & 0 \end{pmatrix} \, . \label{LDR}
\end{align} 
In order to construct a transformation which  eliminates the main part of the 
off-diagonal part $ R^{(1)}  $,  we   conjugate  $\tL_{\mu,\e}^{(1)}$ by 
a symplectic  matrix $\exp(S)$ generated as the flow of a Hamiltonian matrix $S $ 
with  the same form  of 
$  R^{(1)} $.
By a  Lie expansion we obtain
\begin{align}
\label{lieexpansion}
& \tL_{\mu,\e}^{(2)} 
 = \exp(S)\tL_{\mu,\e}^{(1)} \exp(-S) \notag \\
 &  = D^{(1)} +\lie{S}{D^{(1)}}+ \frac12 [S, [S, D^{(1)}]] + 
 R^{(1)}  + [S, R^{(1)}]  \\
 & + 
\frac12 \int_0^1 (1-\tau)^2 \exp(\tau S)  \text{ad}_S^3( D^{(1)} )  \exp(-\tau S) \, \de \tau 
+ \int_0^1 (1-\tau) \, \exp(\tau S) \, \text{ad}_S^2( R^{(1)} ) \, \exp(-\tau S) \, \de \tau \notag 
\end{align} 
where $\text{ad}_A(B) := [A,B] := AB - BA $ denotes the commutator 
between linear operators $ A, B $.

We look for a $ 4 \times 4 $ matrix $S$ as in \eqref{formaS} which solves
the homological equation
$$  
R^{(1)}  +\lie{S}{ D^{(1)} } = 0  
$$
which, recalling \eqref{LDR}, amounts to eliminate the off-diagonal part  
\begin{equation}\label{homoesp}
\begin{pmatrix} 0 & \tJ_2F^{(1)}+\tJ_2\Sigma D_0
- D_1\tJ_2\Sigma \\ 
\tJ_2{[F^{(1)}]}^*+\tJ_2\Sigma^*D_1-D_0\tJ_2\Sigma^* & 0 \end{pmatrix} =0 \, .
\end{equation}
Note that the equation $  \tJ_2F^{(1)}+\tJ_2\Sigma D_0 - D_1\tJ_2\Sigma = 0 $ implies 
also  $  \tJ_2{[F^{(1)}]}^*+\tJ_2\Sigma^*D_1-D_0\tJ_2\Sigma^*  = 0 $ and viceversa. 
Thus, writing  $ \Sigma =\tJ_2  X  $, namely $ X = - \tJ_2  \Sigma $,  
the equation \eqref{homoesp} is equivalent to solve the  ``Sylvester" equation 
\begin{equation}\label{Sylvestereq}
D_1 X - X D_0 = - \tJ_2F^{(1)}   \, .
\end{equation}
Recalling \eqref{LDR}, \eqref{Xsylvy} and \eqref{splitEFG},
it  amounts  to solve the 
 $4\times 4$ real linear system 
\begin{align}\label{Sylvymat}
\footnotesize 
\underbrace{ \begin{pmatrix}  G_{12}^{(1)} - E_{12}^{(1)} &  G_{11}^{(1)} & E_{22}^{(1)} & 0 \\   
G_{22}^{(1)} & G_{12}^{(1)} - E_{12}^{(1)} & 0 & - E_{22}^{(1)} \\
E_{11}^{(1)} & 0 & G_{12}^{(1)} - E_{12}^{(1)}  & -G_{11}^{(1)} \\
 0 &  - E_{11}^{(1)} & -G_{22}^{(1)}  &  G_{12}^{(1)} - E_{12}^{(1)}
 \end{pmatrix}}_{=: {\cal A} }
 \underbrace{ \begin{pmatrix} x_{11} \\ x_{12} \\ x_{21} \\ x_{22} \end{pmatrix}}_{ =: \vec x}
  =
  \underbrace{
    \begin{pmatrix} 
 -F_{21}  \\  F_{22} \\ - F_{11} \\  F_{12}
 \end{pmatrix}
 }_{=: \vec f}.
\end{align}
Recall that, by \eqref{BinH},  $F_{11}=0$.

We solve this system using the following result, verified by a direct calculus.
\begin{lem}\label{LemmaSylvy}
The determinant of the matrix 
\begin{equation}\label{formA}
A := \begin{pmatrix}  a &  b & c & 0 \\  d & a & 0 & - c \\
e & 0 & a  & -b \\
 0 & - e & -d  &  a
 \end{pmatrix} 
\end{equation}
where $ a,b,c, d , e  $ are real numbers,   is 
\begin{equation}\label{Sylvydet}
 \det A = a^4 -2 a^2 (b d + c e)+(b d - c e)^2 \, .
\end{equation}
If $ \det A \neq 0 $ then $ A  $ is invertible and 
\begin{align}\label{Sylvyinv}
  A^{-1}   =  
\footnotesize{\frac{1}{ \det A} \left(
\begin{array}{cccc}
  \! a \left(a^2-b d - c e\right) &  \!  b \left(-a^2+b d - c e\right) & 
-c \left(a^2+b d - 
 c e\right) & \! - 2 a b c \\
  \! d \left(-a^2+b d -  c e\right) &  \! a \left(a^2-b d - c e\right) & 2 a c d & 
   \! - c \left(-a^2-b d + c e\right) \\
  \! - e \left(a^2+b d -  c e\right) &  \! 2 a b e & a \left(a^2-b d - c e\right) & 
 \!  b \left(a^2-b d + c e\right) \\
 \!  - 2 a d e &  \! - e \left(-a^2-b d + c e\right) & d \left(a^2-b d + c e\right) &
  \!  a \left(a^2-b d - c e\right)  
\end{array}
\right)}  \, . 
\end{align}
\end{lem}
As the Sylvester matrix $ \cal A $ in \eqref{Sylvymat} has the  form 
\eqref{formA} with (cfr. \eqref{BinG1}, \eqref{BinG2}) 
\begin{equation} \label{abcde}
\begin{aligned}
& a =  G_{12}^{(1)} - E_{12}^{(1)}  = 
- \frac{\mu}{2} \big(1 +r(\e^2, \mu \e, \mu^2)\big) \, , \quad b = G_{11}^{(1)} 
= 1 +  r(\e^3, \mu \e^2, \mu^2 \e, \mu^3)  \, , \\
&  c = 
E_{22}^{(1)} = - \frac{\mu^2}{8} \big(1 + r(\e,\mu)\big) \, , \quad d = G_{22}^{(1)}
= \mu (1 + r(\mu \e, \mu^2))\, , \quad e = E_{11}^{(1)} = r(\e^2, \mu^2) \, , 
\end{aligned}
\end{equation}
we use  \eqref{Sylvydet} 
 to compute  
\begin{equation}\label{detcalA}
\det {\cal A} =\mu^2(1+r(\mu, \e^3)) \, .
\end{equation}
Moreover, by \eqref{Sylvyinv}, we have 
\begin{align}\label{calA-1}
  {\cal A}^{-1}    =  
\footnotesize{\frac{1}{\mu} \left(
\begin{array}{cccc}
   \frac{\mu}{2}(1+ r(\e,\mu)) &  1+ r(\e,\mu) & 
\frac{\mu^2}{8} (1+ r(\e,\mu)) &  -\frac{\mu^2}{8} (1+ r(\e,\mu))  \\
 \mu (1+ r(\e,\mu))  &  \frac{\mu}{2} (1+ r(\e,\mu))  & \frac{\mu^3}{8} (1+ r(\e,\mu))  & 
  -  \frac{\mu^2}{8} (1+ r(\e,\mu))  \\
r(\e^2, \mu^2)  &  r(\e^2, \mu^2)   &  \frac{\mu}{2}(1+ r(\e,\mu))  & 
 - 1 + r(\e,\mu) \\
 \mu  r(\e^2, \mu^2)  &  r(\e^2, \mu^2)   &  - \mu (1+ r(\e,\mu))  &
 \frac{\mu}{2}(1+ r(\e,\mu)) 
\end{array}
\right)}  \, . 
\end{align}
Therefore, 
for any $\mu\neq 0$, there exists a unique solution $\vec x = {\cal A}^{-1} \vec f  $
of the linear system \eqref{Sylvymat}, namely  a unique matrix $ X $ which solves  
the Sylvester equation \eqref{Sylvestereq}.

\begin{lem}
% For any $ m \neq 0 $
The matrix  solution $X $ of the Sylvester equation \eqref{Sylvestereq} 
is analytic in $(\mu, \e) $ and admits an expansion as in \eqref{Xsylvy}.
\end{lem}

\begin{proof}
% $ x_{ij} = [{\cal A}^{-1} \vec f]_{ij} $.  
 %so that $X$ is the unique solution of the Sylvester equation \eqref{Sylvestereq}.
The expansion    \eqref{Xsylvy} of the coefficients
$ x_{ij} = [{\cal A}^{-1} \vec f]_{ij} $  
follows, for any $\mu \neq 0 $ small,  by 
\eqref{calA-1} 
% \eqref{Sylvyinv}, \eqref{detcalA}, the 
%expansions of $ a,b,c,d,e $ in \eqref{abcde} obtained  by \eqref{BinG1}, \eqref{BinG2},
and the expansions of $ F_{ij} $ in \eqref{BinH}.
In particular each $x_{ij}$ admits an analytic extension at $\mu = 0$ and the resulting matrix $X$   still solves 
\eqref{Sylvestereq} at $\mu = 0$ (note that, for $\mu = 0$, one has  $F^{(1)}= 0$
and the Sylvester equation does not have a unique solution).
\end{proof}
 Since the matrix $ S $ solves the homological equation $\lie{S}{ D^{(1)} }+  R^{(1)} =0$ we deduce 
 by \eqref{lieexpansion} that 
\begin{equation}
\label{Lie2}
\tL_{\mu,\e}^{(2)}
 =   D^{(1)}  +\frac12\lie{S}{ R^{(1)} }+
\frac12 \int_0^1 (1-\tau^2) \, \exp(\tau S) \, \text{ad}_S^2( R^{(1)} ) \, \exp(-\tau S) \de \tau \, . 
\end{equation}
 The matrix $\frac12 \lie{S}{ R^{(1)} }$ is, by \eqref{formaS}, 
\eqref{LDR},   the block-diagonal Hamiltonian and reversible matrix
\begin{equation}\label{Lieeq2}
\frac12 \lie{S}{ R^{(1)} } = \begin{pmatrix}  \frac12 \tJ_2 ( \Sigma \tJ_2 [F^{(1)}]^*- F^{(1)} \tJ_2 \Sigma^*) & 0 \\ 0 & \!\! \!\! \!\!  \frac12 \tJ_2 ( \Sigma^* \tJ_2 F^{(1)}- [F^{(1)}]^* \tJ_2 \Sigma) \end{pmatrix} = \begin{pmatrix} \tJ_2 \tilde E & 0 \\ 0 &\tJ_2 \tilde G \end{pmatrix},
\end{equation}
where, since $ \Sigma = \tJ_2 X $,  
\begin{equation}\label{EGtilde}
\tilde E := \textup{Sym} \big( \tJ_2 X \tJ_2  [F^{(1)}]^* \big) 
%+ (\tJ_2 X \tJ_2  [F^{(1)}]^*)^* \Big)  
\, , \qquad 
\tilde G :=  \textup{Sym} \big(  X^* F^{(1)} \big) \, , 
\end{equation}
denoting $ \textup{Sym}(A) := \frac12 (A+ A^* )$. 
\begin{lem}\label{lem:EGtilde}
The  self-adjoint and reversibility-preserving matrices 
 $ \tilde E, \tilde G  $ in \eqref{EGtilde} have the form
\begin{equation}\label{tilde.E.G}
\tilde E = \begin{pmatrix}
 r_1(\mu\e^2,\mu^3\e, \mu^5)  &  \!\!  \!\!\im r_2(\mu^2\e^2,\mu^3\e,\mu^5) \\ 
- \im r_2(\mu^2\e^2,\mu^3\e,\mu^5) & \!\!  \!\! r_5(\mu^2\e^2,\mu^4\e,\mu^5) \end{pmatrix},  
 \tilde G =
  \begin{pmatrix} 
r_8(\mu \e^2, \mu^3 \epsilon, \mu^5)  & \!\!  \!\!  \im  r_9(\mu^3 \epsilon, \mu^2 \epsilon^2, \mu^5)
  \\
\im  r_9(\mu^3 \epsilon, \mu^2 \epsilon^2, \mu^5) & 
\!\!  \!\!  r_{10}(\mu^4\e, \mu^2\e^2,\mu^6) 
  \end{pmatrix}.
\end{equation}
\end{lem}
\begin{proof}
For simplicity set $F=F^{(1)}$.  By   \eqref{Xsylvy}, \eqref{BinH} and since $F_{11} = 0$ (cfr. \eqref{BinH}), 
one has 
\begin{align*}
\tJ_2 X \tJ_2  F^* &= 
 \begin{pmatrix}   x_{21}F_{12} &  \im ( x_{22}F_{21} + x_{21}F_{22}) 
 \\ 
 \im x_{11}F_{12} & x_{12}F_{21}- x_{11}F_{22}  \end{pmatrix} 
= 
\begin{pmatrix} 
 r(\mu\e^2,\mu^3\e, \mu^5) &  \im  r(\mu^2\e^2,\mu^3\e,\mu^5) \\ 
 \im  r(\mu^2\e^2,\mu^3\e,\mu^5) &  r(\mu^2\e^2,\mu^4\e,\mu^5)  \end{pmatrix} 
\end{align*}
and, adding its symmetric (cfr. \eqref{EGtilde}), 
the expansion of $\tilde E$ in \eqref{tilde.E.G} follows. 
For $\tilde G$ one has
\begin{align*}
 X^*   F &=  \begin{pmatrix}
 x_{21}F_{21} &  \im ( x_{11}F_{12} - x_{21}F_{22}) \\ \im  x_{22}F_{21} & x_{22}F_{22}+ x_{12}F_{12}  \end{pmatrix} 
= \begin{pmatrix}  r(\mu \e^2, \mu^3 \epsilon, \mu^5) &  \im r(\mu^3 \epsilon, \mu^2 \epsilon^2, \mu^5)   \\
\im r(\mu^4 \epsilon, \mu^2 \epsilon^2, \mu^6)  & r(\mu^4 \epsilon,
\mu^2 \e^2,  \mu^6)   
\end{pmatrix} 
\end{align*}
and the expansion of $\tilde G $ in \eqref{tilde.E.G} follows
by symmetrizing. 
\end{proof}

We now show that the last term in \eqref{Lie2} is  very small.

\begin{lem}\label{lem:series}
The $ 4 \times 4 $  Hamiltonian and reversibility  matrix 
\begin{equation}\label{series}
\frac12 \int_0^1 (1-\tau^2) \, \exp(\tau S) \, \textup{ad}_S^2( R^{(1)} ) \, \exp(-\tau S) \, \de \tau
= \begin{pmatrix}
\tJ_2 \widehat E & \tJ_2 F^{(2)}\\
\tJ_2 [ F^{(2)}]^* & \tJ_2 \widehat G
\end{pmatrix}
\end{equation}
where the $ 2 \times 2 $ self-adjoint and reversible  matrices  $\widehat E = \footnotesize\begin{pmatrix} \widehat E_{11} & \im \widehat E_{12} \\ -\im \widehat E_{12} & \widehat E_{22}\end{pmatrix}$, $ \widehat G  = \footnotesize \begin{pmatrix} \widehat G_{11} & \im \widehat G_{12} \\ -\im \widehat G_{12} & \widehat G_{22}\end{pmatrix}$ have   entries 
\begin{equation}\label{E2G2.0}
\widehat E_{ij} \ , \widehat G_{ij}   =  \mu^2 r(\e^3, \mu \e^2, \mu^3 \e, \mu^5 ) \, , \quad 
i,j = 1,2 \, , 
\end{equation}
and   the $2\times 2$ reversible matrix $ F^{(2)}$ admits an expansion as in 
\eqref{Bsylvy}.
\end{lem}
\begin{proof}
Since $S $ and $  R^{(1)}  $ are Hamiltonian and reversibility-preserving
then $ \textup{ad}_S  R^{(1)}  = [S, R^{(1)} ] $ is Hamiltonian   and reversibility-preserving as well.
Thus 
each $ \exp(\tau S) \, \textup{ad}_S^2( R^{(1)} ) \, \exp(-\tau S)$ 
 is Hamiltonian   and reversibility-preserving, and formula  \eqref{series} holds. 
In order to estimate  its entries 
we first compute $\textup{ad}_S^2( R^{(1)} )$. 
Using the form of  $ S $ in \eqref{formaS} and  $[S,  R^{(1)} ]$ in \eqref{Lieeq2} one gets
\begin{equation}\label{tildeF}
\textup{ad}_S^2(R^{(1)})  =  \begin{pmatrix} 0 & \tJ_2\tilde F \\ \tJ_2 \tilde F^* & 0\end{pmatrix}\qquad \text{where} \qquad \tilde F:= 
2\left(  \Sigma \tJ_2 \tilde G - \tilde E \tJ_2 \Sigma \right)
\end{equation}
and  $\tilde E$, $\tilde G$ are defined in  \eqref{EGtilde}.
In order to estimate  $\tilde F$,  
we write $ \footnotesize  \tilde G = \begin{pmatrix} \tilde G_{11} & \im \tilde G_{12} \\ -\im \tilde G_{12} & \tilde G_{22}\end{pmatrix}$, 
$ \footnotesize  \tilde E = \begin{pmatrix} \tilde E_{11} & \im \tilde E_{12} \\ -\im \tilde E_{12} & \tilde E_{22}\end{pmatrix}$ and, by 
\eqref{tilde.E.G}, \eqref{Xsylvy} and $\Sigma = \tJ_2 X$, we obtain 
$$
\Sigma \tJ_2 \tilde G = \begin{pmatrix}
 x_{21} \tilde G_{12} - x_{22} \tilde G_{11} 
 & \!\! \! \!  \im ( x_{21} \tilde G_{22} - x_{22} \tilde G_{12}) \\
 \im ( x_{11} \tilde G_{12} + x_{12} \tilde G_{11}) & \!\! \!\! 
- x_{11} \tilde G_{22} - x_{12} \tilde G_{12}
\end{pmatrix}
=
\begin{pmatrix}
   r(\mu^2 \e^3, \mu^3 \epsilon^2, \mu^5\epsilon,  \mu^7) 
& \!\! \! \! \im r( \mu^2\e^3, \mu^4 \epsilon^2, \mu^5 \epsilon,  \mu^7) \\
\im r( \mu^2\e^3,  \mu^4 \epsilon^2, \mu^5 \epsilon, \mu^7) &
\!\! \! \!  r(\mu^3\e^3, \mu^4 \e^2, \mu^6 \epsilon,   \mu^8)
\end{pmatrix} \, , 
$$
$$
\tilde E \tJ_2 \Sigma = 
\begin{pmatrix}
\tilde E_{12} x_{21} - \tilde E_{11} x_{11}
 &  \!\! \! \!  - \im (\tilde E_{11} x_{12} + \tilde E_{12} x_{22}) \\
\im (\tilde E_{12} x_{11} - \tilde E_{22} x_{21}) 
& \!\! \! \!  -\tilde E_{12} x_{12} - \tilde E_{22} x_{22}
\end{pmatrix}
=
\begin{pmatrix}
r(\mu^2 \epsilon^3, \mu^3 \e^2, \mu^5\e,\mu^7)
 &
 \!\! \! \!   \im r(\mu^2 \epsilon^3, \mu^4 \e^2, \mu^6 \e, \mu^8) \\
 \im r(\mu^2 \epsilon^3, \mu^4 \e^2, \mu^5 \epsilon, \mu^7) 
 & \!\! \! \!  r(\mu^3 \epsilon^3, \mu^4\e^2, \mu^6 \epsilon, \mu^8)
\end{pmatrix} \, . 
$$
Thus  the matrix $\tilde F$ in \eqref{tildeF} has an expansion as in \eqref{Bsylvy}.
Then, for any $ \tau \in [0,1]$, the matrix 
$\exp(\tau S) \, \textup{ad}_S^2( R^{(1)} ) \, \exp(-\tau S) =  \textup{ad}_S^2( R^{(1)} ) (1 + \cO(\mu,\e))$. 
In particular the matrix $F^{(2)}$ in \eqref{series} has the same expansion of $\tilde F$, whereas the matrices $\widehat E$, $\widehat G$ have entries at least 
as in \eqref{E2G2.0}.
\end{proof}

\begin{proof}[Proof of Lemma \ref{decoupling2}.]
It follows by Lemmata  \ref{lem:EGtilde} and 
\ref{lem:series}. The matrix  $E^{(2)} := E^{(1)} + \tilde E + \widehat{ E}$ has the same expansion of $E^{(1)}$ in \eqref{BinG1}.
The same holds for $G^{(2)}$.
\end{proof}

\subsection{Complete block-decoupling and proof of the main results}\label{section34}

We now  block-diagonalize the  $ 4\times 4$ Hamiltonian and reversible 
matrix $\tL_{\mu,\e}^{(2)}$    in \eqref{sylvydec}. 
First we split it 
into its $2\times 2$-diagonal and off-diagonal Hamiltonian and reversible matrices
\begin{align}
& \qquad  \qquad \qquad  \qquad  \qquad \qquad \tL_{\mu,\e}^{(2)} = D^{(2)} + R^{(2)} \, , 
\notag \\
& 
D^{(2)}:=\begin{pmatrix} D_1^{(2)} & 0 \\ 0 & D_0^{(2)} \end{pmatrix} =  \begin{pmatrix} \tJ_2 E^{(2)} & 0 \\ 0 & \tJ_2 G^{(2)}  \end{pmatrix}, \quad 
R^{(2)}:= \begin{pmatrix}  0 & \tJ_2 F^{(2)} \\ \tJ_2 [F^{(2)}]^* & 0 \end{pmatrix} . \label{LDR2}
\end{align}

\begin{lem}\label{ultimate}
There exist a  $4\times 4$ reversibility-preserving Hamiltonian  matrix $S^{(2)}:=S^{(2)}(\mu,\e)$ of the form \eqref{formaS}, analytic in $(\mu, \e)$, of size $\cO(\e^3, \mu \e^2, \mu^3\e,\mu^5)$, and a $4\times 4$ block-diagonal reversible Hamiltonian matrix $P:=P(\mu,\e)$, analytic in $(\mu, \e)$, of size $\mu^2\cO(\e^4,\mu^4\e^3,\mu^6\e^2,\mu^8\e,\mu^{10}) $,   such that 
\begin{equation}\label{ultdec}
\tL_{\mu,\e}^{(3)}:= 
\exp(\mu S^{(2)}) \, \tL_{\mu,\e}^{(2)} \,  \exp(-\mu S^{(2)}) = D^{(2)}+P \ . 
\end{equation}
In particular
\begin{equation}
\label{L3.fin}
\tL_{\mu,\e}^{(3)}  = 
\begin{pmatrix}
\tJ_2 E^{(3)} & 0 \\
0 & \tJ_2 G^{(3)} 
\end{pmatrix}
\end{equation}
where $E^{(3)}$ and $G^{(3)}$ are selfadjoint and reversibility-preserving matrices 
of the form   \eqref{BinG1}-\eqref{BinG2}.
\end{lem}
\begin{proof}
We set for brevity $ S = S^{(2)} $. 
The equation \eqref{ultdec} is equivalent
% writing $ P = \mu [S, R^{(2)}]  + \mathcal{P} $, 
to  the system
\begin{equation}\label{equazionisplittate}
\begin{cases}  \Pi_{D}\big( e^{\mu S} \big(D^{(2)}+R^{(2)}   \big) e^{-\mu S} \big ) - D^{(2)}  = P   \\
\Pi_{\off}\big( e^{\mu S} \big(D^{(2)}+R^{(2)}   \big) e^{-\mu S}\big)  = 0 \, , 
\end{cases}
\end{equation}
where $\Pi_D$ is the projector onto the block-diagonal matrices and $\Pi_\off$ onto 
the block-off-diagonal ones.
The second equation  in \eqref{equazionisplittate} is equivalent, by a Lie expansion, 
and since $ [S, R^{(2)}] $ is block-diagonal,  to 
\begin{equation}\label{nonlinhomo}
R^{(2)} + \mu \lie{S}{D^{(2)}} + \mu^2 \underbrace{\Pi_\off \int_0^1 (1-\tau) e^{\mu\tau S} \text{ad}_S^2\big(D^{(2)}+R^{(2)} \big)e^{-\mu\tau S} \de \tau}_{=: \mathcal{R}(S)} = 0 \, . 
\end{equation}
 The ``nonlinear homological equation" \eqref{nonlinhomo}, i.e. 
$ [S,D^{(2)}] = -\frac1\mu R^{(2)} -\mu \mathcal{R}(S)$, 
is equivalent to solve the $4\times 4$ real linear system
\begin{equation}\label{sistAx2}
{\cal A} \vec{x} =  \vec{f}(\mu,\e,\vec{x}) \, ,\quad \vec{f}(\mu,\e,\vec{x}) = \mu \vec{v}(\mu,\e)+\mu^2 \vec{g}(\mu,\e,\vec{x})
\end{equation}
associated, as in \eqref{Sylvymat}, to \eqref{nonlinhomo}. 
The vector $ \mu \vec{v}(\mu,\e) $ is associated with $ - \frac{1}{\mu} R^{(2)} $
with $R^{(2)}  $ in \eqref{LDR2}.
The vector $ \mu^2 \vec{g}(\mu,\e,\vec{x})  $ is associated with the matrix
$  -\mu \mathcal{R}(S) $, 
% Note that $ \mathcal{R}(S) $ in \eqref{nonlinhomo} 
which is a Hamiltonian and reversible block-off-diagonal 
matrix (i.e of the form \eqref{LDR}), of size 
$\mathcal{R}(S)=\cO(\mu)$ since $\Pi_\off \text{ad}^2_S(D^{(2)}) = 0$.
The function $ \vec{g}(\mu,\e,\vec{x})  $ is quadratic in $ \vec{x} $. 
In view of \eqref{Bsylvy} one has 
\begin{equation}\label{sizev}
\mu^2 \vec{v}(\mu,\e):= (-F^{(2)}_{21},F^{(2)}_{22},-F^{(2)}_{11},F^{(2)}_{12})^\top, \quad F^{(2)}_{ij} = \mu^2 r(\e^3,\mu\e^2,\mu^3\e,\mu^5) \, .
\end{equation}
System \eqref{sistAx2}  is equivalent to $ \vec{x}  = {\cal A}^{-1}  \vec{f}(\mu,\e,\vec{x}) $ and,
%solve the desingularized equation, 
writing  ${\cal A}^{-1} = \frac1\mu {\cal B} (\mu,\e) $ (cfr. \eqref{calA-1}), to 
$$
\vec{x} = {\cal B}(\mu,\e) \vec{v}(\mu,\e) + \mu {\cal B}(\mu,\e) \vec{g}(\mu, \e, \vec{x}) \, . 
$$
By the implicit function theorem
this equation admits a unique small solution $\vec{x}=\vec{x}(\mu,\e)$, analytic in 
$ (\mu, \e ) $,  
 with size $\cO (\e^3,\mu\e^2,\mu^3\e,\mu^5) $ as  $ \vec{v} $ in \eqref{sizev}.
 The claimed estimate of $ P $ follows by the 
the first equation of \eqref{equazionisplittate} and the estimate for 
$ S$ and of $ R^{(2)} $ obtained  by \eqref{Bsylvy}.  
\end{proof}
 
 \noindent
{\sc Proof of Theorems \ref{TeoremoneFinale} and \ref{thm:simpler}. }
By Lemma \ref{ultimate} and recalling  \eqref{calL}
the operator $ \cL_{\mu,\e} : \mathcal{V}_{\mu,\e} \to  \mathcal{V}_{\mu,\e} $ 
is represented by the $4\times 4$ Hamiltonian and reversible matrix 
$$
\im \mu + \exp(\mu S^{(2)})\tL_{\mu,\e}^{(2)} \exp(-\mu S^{(2)}) = \im \mu + 
\begin{pmatrix} \tJ_2E^{(3)} & 0 \\ 0 &  \tJ_2G^{(3)} \end{pmatrix} =: \begin{pmatrix} \mathtt{U} & 0 \\ 0 & \mathtt{S} \end{pmatrix} \, , 
$$ 
where the matrices $E^{(3)}$ and $G^{(3)}$ expand 
as in \eqref{BinG1}-\eqref{BinG2}.
Consequently the  matrices $\mathtt{U}$ and $\mathtt{S}$ 
have an expansion as in \eqref{UU}, \eqref{S}.
Theorem \ref{TeoremoneFinale} is proved. 
The unstable eigenvalues in
 Theorem \ref{thm:simpler} arise from the block $\mathtt{U}$. Its bottom-left entry vanishes for $
\frac{\mu^2}{8} (1+r'_1(\mu,\e)) = \e^2 (1+r''_1(\mu,\e)) $, which, by taking square roots, amounts to solve $ \mu = 2\sqrt{2} \e (1+r(\mu,\e)) $. 
By the implicit function theorem, it admits a unique analytic solution $\underline{\mu}(\e) = 2\sqrt{2} \e(1+  r(\e)). $
% The expansion of the eigenvalues of $\mathtt{U}$ in \eqref{eigelemu} follows and 
The proof of Theorem \ref{thm:simpler} is complete.

\appendix

\section{Proof of Lemma \ref{expansion1} }\label{ProofExpansion}
We  provide the expansion of the basis $f_k^\pm(\mu,\e) = U_{\mu,\e}f_k^\pm $, $k=0,1$, in \eqref{basisF}, where $f_k^\pm$ defined in \eqref{funperturbed}  belong to the subspace $\mathcal{V}_{0,0}:=\text{Rg}(P_{0,0})$. We first Taylor-expand the transformation operators $U_{\mu,\e}$ defined in \eqref{OperatorU}.  We denote  $\pa_\e$ with an  apex and  $\pa_\mu$ with a dot. 
\begin{lem}\label{lem:U.expansion}
The first jets of $U_{\mu,\e}P_{0,0}$ are 
 \begin{align}
  U_{0,0}P_{0,0}&=P_{0,0} \, , \quad U_{0,0}'P_{0,0}=P_{0,0}'P_{0,0} \, , \quad \dot U_{0,0}P_{0,0}=\dot P_{0,0}P_{0,0} \, , \label{Ufirstorder}\\
\dot U_{0,0}'P_{0,0}&=
\big(\dot P_{0,0}'- \frac12 P_{0,0}\dot P_{0,0}' \big)P_{0,0} \, , \label{Umix} 
 \end{align}
where
\begin{align}\label{Pdereps}
 P_{0,0}' &= \frac{1}{2\pi\im} \oint_\Gamma ({\sL}_{0,0}-\lambda)^{-1} {\sL}_{0,0}' ({\sL}_{0,0}-\lambda)^{-1} \de\lambda \, ,  \\ 
\dot P_{0,0}  \label{Pdermu} &= \frac{1}{2\pi\im} \oint_\Gamma ({\sL}_{0,0}-\lambda)^{-1} \dot {\sL}_{0,0} ({\sL}_{0,0}-\lambda)^{-1} \de\lambda \, ,
\end{align}
and
\begin{subequations}
\begin{align}
\dot P_{0,0}' &= -\frac{1}{2\pi\im} \oint_\Gamma ({\sL}_{0,0}-\lambda)^{-1} \dot {\sL}_{0,0} ({\sL}_{0,0}-\lambda)^{-1}  {\sL}_{0,0}' ({\sL}_{0,0}-\lambda)^{-1} \de\lambda  \label{Pmisto1}\\
&\qquad   -\frac{1}{2\pi\im} \oint_\Gamma ({\sL}_{0,0}-\lambda)^{-1} {\sL}_{0,0}' ({\sL}_{0,0}-\lambda)^{-1} \dot {\sL}_{0,0} ({\sL}_{0,0}-\lambda)^{-1} \de\lambda \label{Pmisto2} \\ 
&\qquad   + \frac{1}{2\pi\im} \oint_\Gamma ({\sL}_{0,0}-\lambda)^{-1} \dot {\sL}_{0,0}' ({\sL}_{0,0}-\lambda)^{-1}  \de\lambda \label{Pmisto3} \, .
\end{align}
\end{subequations}
The operators ${\sL}_{0,0}'$ and $\dot {\sL}_{0,0}$ are
\begin{equation}
{\sL}_{0,0}' = \begin{bmatrix} \pa_x \circ p_1(x) & 0 \\ -a_1(x) & p_1(x)\circ \pa_x \end{bmatrix}, \quad \dot {\sL}_{0,0} = \begin{bmatrix} 0 & \sgn(D)+\Pi_0 \\ 0 & 0 \end{bmatrix}, \label{cLfirstorder}
\end{equation}
with $a_1(x)=p_1(x)=-2\cos(x)$, cfr. \eqref{SN1}-\eqref{SN2}.
The operator $\dot {\sL}_{0,0}'$ is
\begin{equation}\label{cLmisto}
\dot {\sL}_{0,0}' = \begin{bmatrix} \im p_1(x) & 0 \\ 0 & \im p_1(x)\end{bmatrix}\,.
\end{equation}
\end{lem}
\begin{proof}
 By \eqref{OperatorU} and \eqref{rootexp} one has the Taylor expansion  in $\cL(Y)$
$$
   U_{\mu,\e}P_{0,0}  = P_{\mu,\e}P_{0,0} + \frac{1}{2}(P_{\mu,\e}-P_{0,0})^2P_{\mu,\e}P_{0,0} +\cO(P_{\mu,\e}-P_{0,0})^4   \, ,
  $$
  where  $\cO(P_{\mu,\e}-P_{0,0})^4 = \cO(\e^4,\e^3\mu,\e^2\mu^2,\e\mu^3,\mu^4) \in \cL(Y)$.
Consequently one derives \eqref{Ufirstorder}, \eqref{Umix},
using also the identity
$\dot P_{0,0} P_{0,0}' P_{0,0} + P_{0,0}' \dot P_{0,0} P_{0,0} = - P_{0,0} \dot P_{0,0}' P_{0,0}$, 
which follows  differentiating $P_{\mu,\e}^2 = P_{\mu,\e}$.  
Differentiating  \eqref{Pproj}  one gets  \eqref{Pdereps}--\eqref{Pmisto3}. Formulas
\eqref{cLfirstorder}-\eqref{cLmisto} follow by \eqref{calL2}.
\end{proof}
By the previous lemma we have the Taylor expansion
\begin{equation}\label{ordinibase}
f_k^\sigma(\mu,\e) = f_k^\sigma + \e P_{0,0}' f_k^\sigma +\mu \dot P_{0,0} f_k^\sigma + \mu\e  \big(\dot P_{0,0}'- \frac12 P_{0,0}\dot P_{0,0}' \big) f_k^\sigma + \cO(\mu^2,\e^2) \, .
\end{equation}
In order to compute the vectors
 $P_{0,0}' f_k^\sigma$ and $\dot P_{0,0} f_k^\sigma$ using 
 \eqref{Pdereps} and \eqref{Pdermu}, it is useful to know the action of  $({\sL}_{0,0} - \lambda)^{-1}$ on the vectors 
\begin{equation}
\label{fksigma}
f_k^+:=\vet{\cos(kx)}{\sin(kx)},
\quad f_k^- :=\vet{-\sin(kx)}{\cos(kx)},
\quad f_{-k}^+ :=\vet{\cos(kx)}{-\sin(kx)},
\quad
f_{-k}^- :=\vet{\sin(kx)}{\cos(kx)} , \quad k \in \bN \, .  
\end{equation}
\begin{lem}\label{lem:VUW}
The space $ H^1(\bT) $ decomposes as 
$
H^1(\bT) =  \cV_{0,0} \oplus \cU \oplus \cW_{H^1} $, with $\cW_{H^1}:= \overline{\bigoplus\limits_{k=2}^\infty \cW_k}^{H^1}\!\!\!\!\!\!\!
$, where the subspaces $\cV_{0,0}, \cU $ and $ \cW_k $, defined below, are 
invariant  under   ${\sL}_{0,0} $ and  the following properties hold:
\begin{itemize}
\item[(i)] $ \cV_{0,0} = \text{span} \{ f^+_1, f^-_1, f^+_0, f^-_0\}$  is the generalized kernel of ${\sL}_{0,0}$. For any $ \lambda \neq 0 $ the operator 
$ {\sL}_{0,0}-\lambda :  \cV_{0,0} \to \cV_{0,0} $ is invertible and  
 \begin{align}\label{primainversione1}
& ({\sL}_{0,0}-\lambda)^{-1}f_1^+ = -\frac1\lambda f_1^+ \, ,
\quad 
({\sL}_{0,0}-\lambda)^{-1}f_1^- = -\frac1\lambda f_1^-,
\quad  ({\sL}_{0,0}-\lambda)^{-1}f_0^- = -\frac1\lambda f_0^- \, ,  \\
& \label{primainversione2}
({\sL}_{0,0}-\lambda)^{-1}f_0^+ = -\frac1\lambda f_0^+ + \frac{1}{\lambda^2} f_0^- \, .
\end{align} 
\item[(ii)] $\cU := \text{span}\left\{ f_{-1}^+, f_{-1}^-  \right\}$.   For any 
$ \lambda \neq \pm 2 \im $ the operator 
$ {\sL}_{0,0}-\lambda :  \cU \to \cU $ is invertible and
\begin{equation}
\label{primainversione3}
 ({\sL}_{0,0}-\lambda)^{-1} f_{-1}^+ = \frac{1}{\lambda^2+4}\left(-\lambda f_{-1}^+ + 2 f_{-1}^-\right), \quad ({\sL}_{0,0}-\lambda)^{-1} f_{-1}^- = \frac{1}{\lambda^2+4}\left(-2 f_{-1}^+ - \lambda f_{-1}^-\right) \, .
\end{equation}
\item[(iii)] 
Each
subspace $\cW_k:= \text{span}\left\{f_k^+, \ f_k^-, f_{-k}^+, \ f_{-k}^- \right\}$ is  invariant under $ {\sL}_{0,0} $.  Let $\cW_{L^2}:=\overline{\bigoplus\limits_{k=2}^\infty \cW_k}^{L^2}\!\!\!\!\!\!$. For any
$|\lambda| < \frac12$, the operator 
$ {\sL}_{0,0}-\lambda :  \cW_{H^1} \to \cW_{L^2} $ is invertible and, 
 for any $f \in \cW_{L^2} $, 
\begin{equation}
\label{primainversione4}
 ({\sL}_{0,0}-\lambda)^{-1} f  =  (\pa_x^2 + |D|)^{-1} \begin{bmatrix} \partial_x & - |D| \\ 1 & \partial_x\end{bmatrix} f + \lambda \varphi_f(\lambda, x) \, ,
\end{equation}
for some analytic  function  $\lambda \mapsto \varphi_f(\lambda, \cdot) \in H^1(\bT, \bC^2)$.
\end{itemize}
\end{lem}
\begin{proof}
By inspection the spaces $\cV_{0,0}$, $\cU$ and $ \cW_k$ are invariant under $ {\sL}_{0,0}$ 
and, by Fourier series, they decompose $H^1(\bT, \bC^2)$. \\
$(i)$ Formulas  \eqref{primainversione1}-\eqref{primainversione2} follow using that 
$f_1^+, f_1^-, f_0^-$ are in the kernel of ${\sL}_{0,0}$, and ${\sL}_{0,0}f_0^+ =-f_0^- $.\\
$(ii)$ Formula \eqref{primainversione3} follows using that  ${\sL}_{0,0} f^+_{-1} = -2 f^{-}_{-1}$ and  ${\sL}_{0,0} f^-_{-1} = 2 f^{+}_{-1}$.\\
$(iii)$ Let $\cW := \cW_{H^1}$.  The operator 
$ \restr{({\sL}_{0,0}-\lambda\uno)}{\cW} $ is invertible for any $ \lambda \notin \{ \pm \im 
\sqrt{|k|} \pm \im k, k \geq 2, k \in {\mathbb N}  \}$ and 
$ \footnotesize   (\restr{{\sL}_{0,0}}{\cW})^{-1} = \left( \pa_x^2 + |D|\right)^{-1} \begin{bmatrix} \pa_x & -|D| \\ 1 & \pa_x\end{bmatrix}_{|\cW}  $.
In particular, by Neumann series, for any  $ \lambda $ such that 
 $ |\lambda | \| (\restr{{\sL}_{0,0}}{\cW})^{-1}\|_{\cL(\cW_{L^2},H^1(\bT))} < 1 $,  e.g. for any $ |\lambda | < 1/ 2 $, 
$$
  (\restr{{\sL}_{0,0}}{\cW}-\lambda)^{-1} = 
(\restr{{\sL}_{0,0}}{\cW})^{-1}  
\big( \uno - \lambda (\restr{{\sL}_{0,0}}{\cW})^{-1} \big)^{-1} =  
   (\restr{{\sL}_{0,0}}{\cW})^{-1} \sum_{k \geq 0} ((\restr{{\sL}_{0,0}}{\cW})^{-1}\lambda)^k 
  \, .
$$
 Formula 
  \eqref{primainversione4} follows  with 
 $\varphi_f(\lambda, x):= (\restr{{\sL}_{0,0}}{\cW})^{-1} 
 \sum_{k \geq 1} \lambda^{k-1} [(\restr{{\sL}_{0,0}}{\cW})^{-1}]^k f $.
\end{proof}
We shall also use the following formulas, obtained  by 
\eqref{cLfirstorder} and \eqref{funperturbed}:
\begin{equation}\label{derivoeps}
\begin{aligned}
&{\sL}_{0,0}'f_1^+ = 2\vet{\sin(2x)}{0} \, , \quad  
{\sL}_{0,0}'f_1^- = 2\vet{\cos(2x)}{0} \, ,  \quad
{\sL}_{0,0}'f_0^+ = 2\vet{\sin(x)}{\cos(x)} \, , \quad
{\sL}_{0,0}'f_0^- = 0 \, , \\
& \dot{\sL}_{0,0}f_1^+ = -\im \vet{\cos(x)}{0}\, , \quad \dot{\sL}_{0,0}f_1^-= \im \vet{\sin(x)}{0} \, , \quad 
\dot{\sL}_{0,0}f_0^+ = 0, \quad
 \dot{\sL}_{0,0}f_0^- =  f_0^+ \, .
\end{aligned}
\end{equation}
We  finally compute $P_{0,0}' f_k^\sigma$ and $\dot P_{0,0}f_k^\sigma$.
\begin{lem}
One has
\begin{equation}\label{tuttederivate}
\begin{aligned}
 &P_{0,0}'f^+_1 =\vet{2\cos(2x)}{\sin(2x)} \, ,\ \ \ 
  P_{0,0}'f^-_1 =\vet{-2\sin(2x)}{\cos(2x)} \, , \ \ \ 
   P_{0,0}'f^+_0 = f^+_{-1} \, , \ \ \  
   P_{0,0}'f^-_0 =0 \, ,   \\
 &\dot P_{0,0} f_1^+ = \frac{\im}{4} f^{-}_{-1} \, ,
 \quad
  \dot P_{0,0} f_1^- = \frac{\im}{4} f^+_{-1}\, ,\quad \dot P_{0,0} f_0^+=0 \, ,\quad  
  \dot P_{0,0} f_0^-=0 \, .
\end{aligned}
\end{equation}
\end{lem}
\begin{proof}
We first compute $P_{0,0}'f_1^+$. By \eqref{Pdereps}, \eqref{primainversione1}  and \eqref{derivoeps}  we deduce
$$
P_{0,0}'f_1^+ = -\frac{1}{2\pi\im} \oint_\Gamma \frac{1}{\lambda}({\sL}_{0,0}-\lambda)^{-1} \vet{2\sin(2x)}{0} \de\lambda \,  .
$$
We note that   $ \footnotesize  \vet{2\sin(2x)}{0} $ belongs to $ \cW$, being equal to $ f_{-2}^- - f_2^- $ (recall \eqref{fksigma}). 
By  \eqref{primainversione4}  there is an analytic function $\lambda \mapsto \varphi(\lambda, \cdot) \in H^1(\bT, \bC^2)$ so that  
$$
P_{0,0}'f_1^+ = -\frac{1}{2\pi\im} \oint_\Gamma \frac{1}{\lambda} \Big( \vet{-2\cos(2x)}{-\sin(2x)}  + \lambda \varphi(\lambda) \Big) \, \de\lambda = \vet{2\cos(2x)}{\sin(2x)} \, ,
$$
using the  residue Theorem.
Similarly one computes $P_{0,0}'f_1^-$. 
 By \eqref{Pdereps}, \eqref{primainversione1}  and \eqref{derivoeps}, one has $P_{0,0}'f_0^-=0$. 
Next we compute $P_{0,0}'f_0^+$. 
 By \eqref{Pdereps}, \eqref{primainversione1},  \eqref{primainversione2}  and \eqref{derivoeps} we get
$$
P_{0,0}'f_0^+  
= -\frac{2}{2\pi\im} \oint_\Gamma \frac{1}{\lambda}({\sL}_{0,0}-\lambda)^{-1}  f^{-}_{-1} \de\lambda 
\stackrel{ \eqref{primainversione3}}{=} 
-\frac{1}{2\pi\im} \oint_\Gamma \Big(-\frac{4 }{\lambda(\lambda^2+4)}f_{-1}^+- \frac{2}{\lambda^2+4} f_{-1}^- \Big)  \de\lambda  =f^+_{-1} 
  \, ,  
$$
where in the last step we used the residue theorem.
We compute now $\dot P_{0,0} f^+_1$. 
First we have
$ \dot P_{0,0}f_1^+   =\ \frac{\im}{2\pi\im} \oint_\Gamma \frac{1}{\lambda}({\sL}_{0,0}-\lambda)^{-1} \footnotesize \vet{\cos(x)}{0} \de\lambda  $
and then, writing $ \footnotesize  \vet{\cos(x)}{0} =\frac{1}{2} (f_1^+ + f_{-1}^+ )$ and using \eqref{primainversione3}, we conclude
$$
\dot P_{0,0} f_1^+
=  \frac{\im}{2}\frac{1}{2\pi\im} \oint_\Gamma \Big(  - \frac{1}{\lambda^2} f^+_1 - \frac{1}{\lambda^2+4}f_{-1}^+ +\frac{2}{\lambda (\lambda^2+4)}  f_{-1}^- \Big) \de\lambda
 = \frac{\im}{4} f_{-1}^-  
$$
using again the residue theorem. The computations of $\dot P_{0,0}f^-_1$, $\dot P_{0,0} f_0^+$, $\dot P_{0,0} f_0^-$ are analogous.
\end{proof}
So far we have obtained the linear terms of the expansions  \eqref{exf41}, \eqref{exf42}, \eqref{exf43}, \eqref{exf44}. 
We now provide further information about the expansion of the basis at $\mu=0$. 
 \begin{lem}
 The basis $\{f_k^\sigma(0,\e), \, k = 0,1 , \sigma = \pm\}$ is real.
 For any $\e $  it results $f_0^-(0,\e) \equiv f_0^- $. The property \eqref{nonzeroaverage} holds.
 \end{lem}
 \begin{proof}
 The reality of the basis $f_k^\sigma(0,\e)$ is a consequence of Lemma \ref{propPU}-$(iii)$.  
 Since,  recalling \eqref{calL2}, 
 $ {\sL}_{0,\e} f_0^{-} =  0 $ for any $ \e $ (cfr. \eqref{genespace}), 
 we deduce $({\sL}_{0,\e}-\lambda)^{-1}f_0^-= -\frac{1}{\lambda}f_0^- $ and then, using also   the residue theorem, 
$$
  P_{0,\e}f_0^-= - \frac{1}{2\pi \im} \oint_\Gamma ({\sL}_{0,\e}-\lambda)^{-1}f_0^- \de\lambda =   f_0^- \ . 
$$
In particular $P_{0,\e}f_0^- = P_{0,0} f_0^-$, for any 
$  \e$ and  we get, by \eqref{OperatorU},  
$f_0^-(0,\e) = U_{0,\e} f_0^- = f_0^-$, for any  $ \e$. 

Let us prove property \eqref{nonzeroaverage}. 
In view of \eqref{reversiblebasisprop} and since the basis is real,  we know that
$ \footnotesize  f_k^+ (0,\e) =\vet{even(x)}{odd(x)} $, $ \footnotesize  f_k^- (0,\e) =\vet{odd(x)}{even(x)} $,
 for any  $ k=0,1 $. 
  By Lemma \ref{base1symp} the basis $\{f_k^\sigma(0,\e)\}$ is symplectic (cfr. \eqref{symplecticbasis}) and, since 
  $\cJ f_0^-(0,\e)  = \cJ f_0^- = \footnotesize \vet{1}{0}$, for any $\epsilon $, we get  
$$
 0 = \molt{\cJ f_0^-(0,\e)}{f_1^+(0,\e)} = \Big( \vet{1}{0}, f_1^+(0,\e) \Big) \, , \quad 
 1 = \big( \cJ f_0^-(0,\e), f_0^+(0,\e) \big) = \Big( \vet{1}{0}, f_0^+(0,\e) \Big) \, . 
$$
Thus the first component of both $f_1^+(0,\e)$ and $f_0^+(0,\e)- \footnotesize \vet{1}{0}$ has zero average,
 proving 
 \eqref{nonzeroaverage}.
 \end{proof}

We now provide further information about the expansion of the basis at $\e=0$. 
\begin{lem}
For any small $\mu$, we have $f_0^+(\mu,0) \equiv f_0^+ $
and $f_0^-(\mu,0) \equiv f_0^- $. Moreover the vectors $f_1^+(\mu,0)$ and $f_1^-(\mu,0)$ have both components with zero space average.
\end{lem}
\begin{proof}
The operator ${\sL}_{\mu,0} =  
\footnotesize 
\begin{bmatrix} \pa_x & |D+\mu|\\ -1 & \pa_x \end{bmatrix}$  leaves invariant  the subspace $\mathcal{Z}:=\text{span}\{f_0^+,\,f_0^-\}$ since
$ {\sL}_{\mu,0} f_0^+ = -f_0^-  $ and $ {\sL}_{\mu,0} f_0^- = \mu f_0^+ $. 
The operator $\restr{{\sL}_{\mu,0}}{\mathcal{Z}}$ has the two eigenvalues $\pm\im \sqrt{\mu}$, which, for small $\mu$, lie inside the loop $\Gamma$ around $0$ in \eqref{Pproj}.  
Then, by \eqref{dec.spectrum},  we have 
 $\mathcal{Z} \subseteq {\cV}_{\mu,0} = \text{Rg}(P_{\mu,0}) $ and 
$$
P_{\mu,0} f_0^\pm = f_0^\pm, \quad f_0^\pm(\mu,0) = U_{\mu,0} f_0^\pm = f_0^\pm, \text{ for any }\mu\text{ small} \, .
$$
The basis $\{f_k^\sigma(\mu,0) \}$ is symplectic. Then, since  $ \footnotesize 
\cJ f_0^+ = \vet{0}{-1} $ and  $\footnotesize \cJ f_0^- = \vet{1}{0} $, we have 
$$
 0 = \molt{\cJ f_0^+(\mu,0)}{f_1^\sigma(\mu,0)} = 
 \Big({\footnotesize \vet0{-1}}, f_1^\sigma(\mu,0) \Big) \, , \ \   
 0 = \Big( \cJ f_0^-(\mu,0), f_1^\sigma(\mu,0) \Big) = 
 \Big({\footnotesize \vet10}, f_1^\sigma(\mu,0) \Big) \, , 
$$
namely both the components of $f_1^\pm(\mu,0)$ have zero average.
\end{proof}
We finally consider  the $\mu\e$ term in the expansion \eqref{ordinibase} of the vectors $f_k^\sigma(\mu,\e) $, $ k = 0,1 $, $ \sigma = \pm $. 
\begin{lem}
The derivatives
$ (\pa_{\mu} \pa_\e f_k^\sigma)(0,0) = 
 \left(\dot P_{0,0}'- \frac12 P_{0,0}\dot P_{0,0}' \right)f_k^\sigma $
 satisfy 
 \begin{equation}\label{struttura2}
\begin{aligned}
&  (\pa_{\mu} \pa_\e f_1^+)(0,0)  = \im \vet{odd(x)}{even(x)},  \qquad 
(\pa_{\mu} \pa_\e f_1^-)(0,0) - = \im \vet{even(x)}{odd(x)},\\
& (\pa_{\mu} \pa_\e f_0^+)(0,0) = \im \vet{odd(x)}{even_0(x)},
\qquad 
(\pa_{\mu} \pa_\e f_0^-)(0,0) = \frac{1}{2} \vet{\sin(x)}{\cos(x)}
+ \im   \vet{even_0(x)}{odd(x)}  \ . 
\end{aligned}
\end{equation}
\end{lem}

\begin{proof}
We decompose the Fourier multiplier operator 
$\dot{\sL}_{0,0}$ in \eqref{cLfirstorder} as
$$
\dot {\sL}_{0,0} = \dot {\sL}_{0,0}^{(I)} + \dot {\sL}_{0,0}^{(II)} \, , \qquad
\dot {\sL}_{0,0}^{(I)} := 
\begin{bmatrix} 0 & \sgn(D) \\ 0 & 0 \end{bmatrix} \, , \qquad
\dot {\sL}_{0,0}^{(II)} :=  \begin{bmatrix} 0 & \Pi_0 \\ 0 & 0 \end{bmatrix} \, ,
$$
and, accordingly, we write 
$\dot P_{0,0}' = \eqref{Pmisto1}^{(I)} + \eqref{Pmisto1}^{(II)} +
\eqref{Pmisto2}^{(I)} + \eqref{Pmisto2}^{(II)} + \eqref{Pmisto3} $ 
defining 
\begin{align}
\eqref{Pmisto1}^{(I)} &:= -\frac{1}{2\pi\im} \oint_\Gamma ({\sL}_{0,0}-\lambda)^{-1} \dot {\sL}_{0,0}^{(I)} ({\sL}_{0,0}-\lambda)^{-1}  {\sL}_{0,0}' ({\sL}_{0,0}-\lambda)^{-1} \de\lambda \, , \\  
\eqref{Pmisto1}^{(II)} &:= -\frac{1}{2\pi\im} \oint_\Gamma ({\sL}_{0,0}-\lambda)^{-1} \dot {\sL}_{0,0}^{(II)} ({\sL}_{0,0}-\lambda)^{-1}  {\sL}_{0,0}' ({\sL}_{0,0}-\lambda)^{-1} \de\lambda \, , \\  
\eqref{Pmisto2}^{(I)} &:= -\frac{1}{2\pi\im} \oint_\Gamma ({\sL}_{0,0}-\lambda)^{-1}  {\sL}_{0,0}' ({\sL}_{0,0}-\lambda)^{-1} \dot {\sL}_{0,0}^{(I)} ({\sL}_{0,0}-\lambda)^{-1}  \de\lambda \, , \\  
\eqref{Pmisto2}^{(II)} &:= -\frac{1}{2\pi\im} \oint_\Gamma ({\sL}_{0,0}-\lambda)^{-1}   {\sL}_{0,0}' ({\sL}_{0,0}-\lambda)^{-1} \dot {\sL}_{0,0}^{(II)} ({\sL}_{0,0}-\lambda)^{-1} \de\lambda \, .
\end{align}
Note that the operators $\eqref{Pmisto1}^{(I)}$,  $\eqref{Pmisto2}^{(I)}$ and  $\eqref{Pmisto3}$ are purely imaginary 
 because $\dot {\sL}_{0,0}^{(I)}$ is purely imaginary, 
  ${\sL}_{0,0}' $  in \eqref{cLfirstorder}  is real
and $\dot {\sL}_{0,0}'$ in \eqref{cLmisto} is purely imaginary 
(argue as in  Lemma \ref{propPU}-$(iii)$).
Then, applied to the real vectors $f^\sigma_k$, $k = 0,1$, $\sigma = \pm$, give purely imaginary vectors. 
\\
We first  compute  $(\pa_{\mu} \pa_\e f_1^+)(0,0)$.
Using  \eqref{primainversione1}  and \eqref{derivoeps} we get
\begin{align*}
\eqref{Pmisto1}^{(II)}f_1^+ =  \frac{2}{2\pi\im} \oint_\Gamma \frac{1}{\lambda} 
({\sL}_{0,0}-\lambda)^{-1} \dot {\sL}_{0,0}^{(II)} ({\sL}_{0,0}-\lambda)^{-1}   \vet{\sin(2x)}{0} \de\lambda  = 0
\end{align*}
because, by Lemma 
\ref{lem:VUW}, 
$ ({\sL}_{0,0}-\lambda)^{-1} \footnotesize \vet{\sin(2x)}{0} \in \cW$ and therefore it is a vector with zero average, so in the kernel of 
$\dot {\sL}_{0,0}^{(II)}$. 
In addition 
 $\eqref{Pmisto2}^{(II)}f_1^+ = 0$ since
$ \dot {\sL}_{0,0}^{(II)} ({\sL}_{0,0}-\lambda)^{-1} f_1^+ = 0$.
All together
$\dot P_{0,0}' f_1^+ $ is a purely imaginary vector. 
Since $P_{0,0}$ is a real operator, also 
$ (\dot P_{0,0}'- \frac12 P_{0,0}\dot P_{0,0}' )  f_1^+ $
is purely imaginary, and 
  Lemma
\ref{lem:f.parity} implies that  $(\pa_{\mu} \pa_\e f_1^+)(0,0) $ has the claimed structure
in \eqref{struttura2}.
In the same way one proves the structure for 
 $(\pa_{\mu} \pa_\e f_1^-)(0,0)$.

Next we prove that $(\pa_{\mu} \pa_\e f_0^+)(0,0)$,
in addition to being purely imaginary,  has zero average.
We have, by \eqref{primainversione2}  and \eqref{derivoeps}
\begin{align*}
\eqref{Pmisto1}^{(I)} f_0^+:= \frac{2}{2\pi\im} \oint_\Gamma ({\sL}_{0,0}-\lambda)^{-1} \dot {\sL}_{0,0}^{(I)} ({\sL}_{0,0}-\lambda)^{-1} \frac{1}{\lambda}   \vet{\sin(x)}{\cos(x)} \, \de\lambda
\end{align*}
and since the operators $({\sL}_{0,0}-\lambda)^{-1}$  and $\dot {\sL}_{0,0}^{(I)}$ are  both Fourier multipliers, hence they preserve the absence of average of the vectors, then  $\eqref{Pmisto1}^{(I)} f_0^+$ has zero average.
In addition
$\eqref{Pmisto1}^{(II)} f_0^+ = 0$ 
as  $\dot {\sL}_{0,0}^{(II)} ({\sL}_{0,0}-\lambda)^{-1} \footnotesize  \vet{\sin(x)}{\cos(x)}  = 0$.
Next     $ \eqref{Pmisto2}^{(I)} f_0^+ = 0$ since 
$\dot {\sL}_{0,0}^{(I)} f_0^\pm = 0$, cfr. \eqref{def:segno}. 
Using  also that $ \dot {\sL}_{0,0}^{(II)} f_0^+ = 0$
and $ \dot {\sL}_{0,0}^{(II)} f_0^- = f_0^+ $, 
\begin{align*}
\eqref{Pmisto2}^{(II)} f_0^+
& \stackrel{\eqref{primainversione2}}{=} -\frac{1}{2\pi\im} \oint_\Gamma ({\sL}_{0,0}-\lambda)^{-1}  {\sL}_{0,0}' ({\sL}_{0,0}-\lambda)^{-1} \frac{1}{\lambda^2} f_0^+    \de\lambda   \\
& \stackrel{\eqref{primainversione2}, \eqref{derivoeps}}{=} 
   \frac{2}{2\pi\im} \oint_\Gamma 
\frac{1}{\lambda^3}   
   ({\sL}_{0,0}-\lambda)^{-1}   \vet{\sin(x)}{\cos(x)}    \de\lambda = 0 
\end{align*}
using \eqref{primainversione3} and the residue theorem. 
Finally, by  \eqref{primainversione2} and \eqref{cLmisto} where $ p_1 (x) = - 2 \cos (x) $, 
$$
\eqref{Pmisto3} f_0^+ = 
- \frac{\im 2}{2\pi\im} \oint_\Gamma ({\sL}_{0,0}-\lambda)^{-1}
  \Big( 
- \frac{1}{\lambda} \vet{ \cos (x)}{0} + \frac{1}{\lambda^2} \vet{0}{ \cos (x)} \Big) \, 
  \de\lambda  
$$
is a vector with zero average. 
We conclude that $\dot P_{0,0}' f_0^+$ is an imaginary vector with zero average, as well as  $(\pa_\mu \pa_\e f_0^+)(0,0)$ since $P_{0,0}$ sends zero average functions in zero average functions. 
Finally,  by Lemma \ref{lem:f.parity},  $(\pa_\mu \pa_\e f_0^+)(0,0)$ has the claimed structure in \eqref{struttura2}.

We  finally consider $(\pa_{\mu} \pa_\e f_0^-)(0,0)$.
By \eqref{primainversione1} and ${\sL}_{0,0}'f_0^-=0$ (cfr. \eqref{derivoeps}), it results
$$
\eqref{Pmisto1}^{(M)}f_0^- = - \frac{1}{2\pi\im} \oint_\Gamma \frac{({\sL}_{0,0}-\lambda)^{-1}}{\lambda} \dot {\sL}_{0,0}^{(M)} ({\sL}_{0,0}-\lambda)^{-1} {\sL}_{0,0}' f_0^- \de\lambda=0 \, , 
\quad M = I, II \ . 
$$
Next by \eqref{primainversione1} and 
$\dot {\sL}_{0,0}^{(I)} f_0^- = 0$ we get
$\eqref{Pmisto2}^{(I)}f_0^- = 0$.
Then,   since $  \dot {\sL}_{0,0}^{(II)} f_0^- = f_0^+$,
\begin{align*}
\eqref{Pmisto2}^{(II)}f_0^- &\stackrel{\eqref{primainversione1}-\eqref{primainversione2}}{=}   
 \frac{1}{2\pi\im} \oint_\Gamma \frac{({\sL}_{0,0}-\lambda)^{-1}}{\lambda} \ {\sL}_{0,0}'  \Big(-\frac1{\lambda}f_0^+ + \frac{1}{\lambda^2}
f_0^-  \Big)  \de\lambda \\
& \stackrel{\eqref{derivoeps},\eqref{primainversione3}}{=} 
-\frac{2}{2\pi\im} \oint_\Gamma \frac{1}{\lambda^2}  \frac{1}{\lambda^2+4} (-2f_{-1}^+ -\lambda f_{-1}^-)  \de\lambda = \frac12 f_{-1}^-= \frac12 \vet{\sin(x)}{\cos(x)} \, ,
 \end{align*}
 which is the only real term of $ (\pa_\mu \pa_\e  f_0^-)(0,0)$ in \eqref{struttura2}.
 Finally by \eqref{primainversione1} and \eqref{cLmisto}
 \begin{align*}
\eqref{Pmisto3} f_0^- = 
\frac{2 \im  }{2\pi\im} \oint_\Gamma ({\sL}_{0,0}-\lambda)^{-1}
\frac{1}{\lambda} 
 \vet{0}{ \cos (x)}
  \de\lambda \  = - \frac{\im}{2}\vet{\cos(x)}{-\sin(x)} 
\end{align*}
by \eqref{primainversione1}, \eqref{primainversione3} and the residue theorem.
In conclusion
 $\dot P_{0,0}' f_0^- =  \frac12\footnotesize \vet{\sin(x)}{\cos(x)} - \frac{ \im}{2} \vet{\cos(x)}{-\sin(x)}  \in \cU $ and, since $P_{0,0}\vert_{\cU} = 0$, we find that 
  $\left(\dot P_{0,0}'- \frac12 P_{0,0}\dot P_{0,0}' \right)  f_0^- = 
   \frac12 \footnotesize\vet{\sin(x)}{\cos(x)}  - \frac{ \im}{2} \vet{\cos(x)}{\sin(x)} $.
\end{proof}
   
   This
  completes the proof of Lemma \ref{expansion1}.

 \begin{footnotesize}

 \end{footnotesize}

\begin{thebibliography}{99}
 
 \bibitem{Ak}
B. Akers,
\emph{Modulational instabilities of periodic traveling waves in deep water}.
Phys. D 300, 26-33, 2015. 


 	\bibitem{Nic}
B. Akers and D. Nicholls. 
\emph{Spectral stability of deep two-dimensional gravity water waves: repeated eigenvalues}.
SIAM J.  App. Math.,  72(2): 689--711, 2012.
 	
 	\bibitem{Arn}
V.I. Arnold. 
{\it The complex Lagrangian Grassmanian}, Func. Anal. Appl. 34 208-210, 2000.

				
			\bibitem{BBHM} P. Baldi, M. Berti, E. Haus, R. Montalto, 
 		\emph{Time quasi-periodic gravity water waves
 			in finite depth}. Inv. Math. 214 (2): 739--911, 2018.
 				
\bibitem{toda}
D. Bambusi  and A. Maspero.
\emph{ Birkhoff coordinates for the Toda Lattice in the limit of infinitely many particles with an application to FPU.}
 J. Funct. Anal., 270(5): 1818--1887, 2016.	
 				
\bibitem{Benjamin} T.  Benjamin. 
{\it Instability of periodic wave trains in nonlinear dispersive systems}. Proceedings of the Royal Society of London, A, Vol. 299, No. 1456, 
%A Discussion on Nonlinear Theory of Wave Propagation in Dispersive Systems,
 pp. 59-75, 1967.
	
			\bibitem{BF}
			T.  Benjamin  and J. Feir.
 		\emph{The disintegration of wave trains on deep water}, Part 1. Theory. J. Fluid Mech. 27(3): 417-430, 1967.
 				
	
 		\bibitem{BFM}
 		M. Berti, L. Franzoi  and A. Maspero.
 		\emph{Traveling quasi-periodic water waves with constant vorticity},  Archive for Rational Mechanics, 240: 99--202, 2021. 
% 		DOI:  https://doi.org/10.1007/s00205-021-01607-w
 	
		\bibitem{BFM2}
 		M. Berti, L. Franzoi  and A.  Maspero.
 		\emph{Pure gravity traveling quasi-periodic water waves with constant vorticity}, 
		\texttt{arXiv:2101.12006}, 2021, to appear on Communications in Pure and Applied Mathematics.


\bibitem{BMV}  M. Berti,  A. Maspero  and P.  Ventura.
	\emph{ On the analyticity of the Dirichlet-Neumann operator and Stokes waves}, to \texttt{arXiv:2201.04675}, to  appear on Atti Accad. Naz. Lincei Rend. Lincei Mat. Appl.
 	
	 		\bibitem{BM} 
 		M. Berti  and R. Montalto.
 		\emph{Quasi-periodic standing wave solutions of gravity-capillary water waves}, 
 		Volume 263, MEMO 1273, Memoires AMS, ISSN 0065-9266, 2020.
 		
 		
		
 		\bibitem{BrM}
T.	Bridges   and A. Mielke. 
	\emph{ A proof of the Benjamin-Feir instability.} Arch. Rational Mech. Anal. 133(2): 145--198, 1995.

\bibitem{BHJ}
J. Bronski,  V. Hur  and M.  Johnson.
\emph{ Modulational Instability in Equations of KdV Type.} In: Tobisch E. (eds) New Approaches to Nonlinear Waves. Lecture Notes in Physics, vol 908. Springer,  2016.

 		\bibitem{BJ} 
	J. Bronski  and M. Johnson.
		\emph{The modulational instability for a generalized Korteweg-de Vries equation.} Arch. Ration. Mech. Anal. 197(2): 357--400, 2010.
		
		\bibitem{BuT}
		B. Buffoni and J. Toland.
		\emph{Analytic Theory of Global Bifurcation}. Princeton University Press, 2016.
		
		
%		\bibitem{CSS}
%	W.	Craig, U. Schanz,  and C. Sulem.
%		\emph{The modulational regime of three-dimensional water waves and the Davey-Stewartson system.}
%		 Annales de l'I.H.P. Analyse non lin\`eaire 14(5): 615--667, 1997.

\bibitem{ChS}
G. Chen and Q. Su.
\emph{Nonlinear modulational instabililty of the Stokes waves in 2d full water waves.}
	arXiv:2012.15071.


\bibitem{CS}
W. Craig and C. Sulem. 
\emph{Numerical simulation of gravity waves.}
J. Comput. Phys., 108(1): 73--83, 1993.

	\bibitem{CDT}
	R. Creedon, B. Deconinck, O. Trichtchenko. 
 {\it High-Frequency Instabilities of Stokes Waves}. 
Journal of Fluid Mechanics, 937, A24. doi:10.1017/jfm.2021.1119, 2022. 
% arXiv:2107.11489. 


 		 \bibitem{DO} 
		 B. Deconinck  and K. Oliveras.
		 \emph{ The instability of periodic surface gravity waves.}
		  J. Fluid Mech., 675: 141--167, 2011.
 	
\bibitem{DU}
B. Deconinck and   J. Upsal.
\emph{The Orbital Stability of Elliptic Solutions of the Focusing Nonlinear Schr\"odinger Equation.}
SIAM J.  Math.  Anal.,  
52(1):  1--41, 2020.

\bibitem{EM}
W.N. Everitt and L. Markus. 
{\it Complex symplectic geometry with applications to ordinary differential operators}, 
Trans. Amer. Math. Soc. 351 4905-4945 (1999).


\bibitem{FG} R. Feola and F. Giuliani.
 {\it Quasi-periodic traveling waves on an infinitely deep fluid under gravity}. 	arXiv:2005.08280, to appear on Memoirs of the American Mathematical Society. 


\bibitem{FMMX} H. Fa{\ss}bender, S. Mackey,  N. Mackey and H. Xu.
{\it Hamiltonian square roots of skew-Hamiltonian matrices}.
Linear Algebra and its Applications, 287(1): 125--159, 1999.

		\bibitem{GH} 
		T. Gallay  and M. Haragus.
		\emph{ Stability of small periodic waves for the nonlinear Schr\"odinger equation.}
		 J. Differential Equations, 234: 544--581, 2007.

		\bibitem{HK}
		M. Haragus  and T.  Kapitula.
		\emph{On the spectra of periodic waves for infinite-dimensional Hamiltonian systems}. Phys. D, 237: 2649--2671, 2008.
 	
	\bibitem{HJ} 
	V. Hur and M. Johnson. 
	\emph{Modulational instability in the Whitham equation for water waves.} Stud. Appl. Math. 134(1): 120--143, 2015.
	

	
\bibitem{HP} 
V. Hur and A. Pandey. 
\emph{Modulational instability in nonlinear nonlocal equations of regularized long wave type.} Phys. D, 325: 98--112, 2016.

\bibitem{HY}
V. Hur and  Z. Yang.
\emph{Unstable Stokes waves}. arXiv:2010.10766.

\bibitem{IK}
G. Iooss and P. Kirrmann. 
{\it Capillary gravity waves on the free surface of an inviscid fluid of infinite depth }, 
Arch. Rat. Mech. Anal. 136 1-19, 1996.


\bibitem{Li}
M. J. Lighthill,
{\it Contribution to the theory of waves in nonlinear dispersive systems},
IMA Journal of Applied Mathematics, 1, 3, 269-306, 1965.


\bibitem{JLL}
J. Jin, S. Liao and Z. Lin. 
\emph{Nonlinear modulational instability of dispersive PDE models.} Arch. Ration. Mech.
Anal. 231(3): 1487-–1530, 2019.



\bibitem{J} 
M. Johnson. 
\emph{Stability of small periodic waves in fractional KdV type equations.} SIAM J. Math. Anal. 45: 2529--3228, 2013.
	
 
\bibitem{K}
T.	Kappeler. 
	\emph{Fibration of the phase space for the Korteweg-de Vries equation.} Annales de l'institut Fourier 41(3): 539--575,  1991.
 
 
 		\bibitem{Kato}
	T. Kato. 
	\emph{Perturbation theory for linear operators.}
	% Die Grundlehren der mathematischen Wissenschaften, Band 132 
Springer-Verlag %New York, Inc., New York, 
1966.
 	
	\bibitem{KDZ}
	A. O. Korotkevich, A. I. Dyachenko and V. E. Zakharov,
{\it Numerical simulation of surface waves instability on a homogeneous grid},
Physica D: Nonlinear Phenomena,
Volumes 321-322, 51-66, 2016. 
%Pages 51-66, %ISSN 0167-2789, https://doi.org/10.1016/j.physd.2016.02.017.
	
 	\bibitem{KP}
 	S. Kuksin and  G. Perelman.
 	\emph{Vey theorem in infinite dimensions and its application to KdV}.
Discrete Cont. Dyn. Syst. 27(1):1--24, 2010.

\bibitem{LBJM}
K. Leisman, J.  Bronski, M.  Johnson,  and
R. Marangell.
\emph{
 Stability of Traveling Wave Solutions of Nonlinear Dispersive Equations of NLS Type.}
 \newblock{ Arch. Rational Mech. Anal.}, 240: 927-969, 2021.


 \bibitem{LC}
T. 		Levi-Civita.
 		\newblock {\it D\'etermination rigoureuse des ondes permanentes d' ampleur finie,} 
 		\newblock { Math. Ann.} 93: 264-314, 1925.
 		
\bibitem{Lewy}
H. Lewy.
 	 		\newblock {\it A note on harmonic functions and a hydrodynamical application,}
 	 		\newblock{Proc. Amer. Math. Soc.},  3: 111--113, 1952.	
 		
 		
 	\bibitem{Ma_tame}	
 		A. Maspero.
 		\emph{Tame majorant analyticity for the Birkhoff map of the defocusing Nonlinear Schr\"odinger equation on the circle.} Nonlinearity, 31(5): 1981--2030, 2018.
 		
 		\bibitem{Nek}
A. 		Nekrasov.  {\it On steady waves}.  Izv. Ivanovo-Voznesenk. Politekhn. 3, 
 		1921.
 	
 		\bibitem{NR}
D. 	Nicholls and F. Reitich. {\it On analyticity of travelling water waves},
Proc. R. Soc. A, 461: 1283-130, 2005. 
	
		\bibitem{NS} 
		H. Nguyen and W. Strauss.
		\emph{Proof of modulational instability of Stokes waves in deep water}. To appear in Comm.  Pure Appl.  Math., 2020.
 	
	\bibitem{Olver}
	P.J. Olver. 
	{\it Hamiltonian perturbation theory and water waves}, 
	Cont. Math., Amer. Math. Society 28 231-249, 1984.
			
		\bibitem{RT} 
		F. Rousset and N. Tzvetkov. 
		{\it Transverse instability of the line solitary water-waves}. Invent. Math. 184: 257-388, 2011.
		
		\bibitem{SHCH}
		 H. Segur, D. Henderson, J. Carter and J. Hammack. 
		 \emph{Stabilizing the Benjamin-Feir instability}.
		  J. Fluid Mech. 539: 229--271, 2005.
		
			
 		\bibitem{stokes}
G. Stokes. {\it On the theory of oscillatory waves}.
 		Trans. Cambridge Phil. Soc. 8: 	441--455, 1847.
 		
 		
 		\bibitem{Struik}
 	D. 	Struik. 
 	 {\it D\'etermination rigoureuse des ondes irrotationelles p\'eriodiques dans un canal \'a profondeur
 			finie}.  Math. Ann. 95: 595--634, 1926.
 		
% 	
% 		\bibitem{To}
% 		Toland J. F.,  
% 		{\it On the existence of a wave of greatest height and Stokes conjecture}, 
% 		Proc. Roy. Soc. London Ser. A 363, 1715, 469-485, 1978.
% 		
 		
 		\bibitem{Wh}
		G.B. Whitham.   {\it Linear and Nonlinear Waves}. J. Wiley-Sons, 
		New York, 1974.
% 		Wahl\'en E., {\it Steady periodic capillary-gravity waves with vorticity}, 
% 		SIAM J. Math. Anal. 38, 921-943, % (electronic), 
% 		2006.
 		
 		



\bibitem{Z0}
V. Zakharov. {\it The instability of waves in nonlinear dispersive media},
J. Exp.Teor.Phys. 24 (4), 740-744, 1967.

%\bibitem{Z1}
%V. E. Zakharov, 
%{\it Stability of periodic waves of finite amplitude on a surface}.
%J. Appl. Mech. Tech. Phys. 9 (2), 190-194, 1968.

 		\bibitem{Zak1} 
 		V.  Zakharov.
 		{\it Stability of periodic waves of finite amplitude on the surface of a deep fluid}. Zhurnal
 		Prikladnoi Mekhaniki i Teckhnicheskoi Fiziki 9(2): 86--94, 1969.

\bibitem{ZK}
V. Zakharov and V.  Kharitonov. {\it Instability of monochromatic waves on the surface of a liquid of arbitrary depth}. J Appl Mech Tech Phys 11, 747-751, 1970.


\bibitem{ZO}
V. Zakharov and L. Ostrovsky. 
\emph{Modulation instability: the beginning.} 
Phys. D,  238(5): 540--548, 2009.

 	 		
 	\end{thebibliography}
\end{document}